\documentclass[UTF-8,reqno]{amsart}
\usepackage{enumerate}
\usepackage[margin=1.0in]{geometry}
\usepackage{amssymb,url,color,  tikz, booktabs}
\usepackage{graphics,graphicx}

\usepackage{mathrsfs}
\usetikzlibrary{decorations.pathreplacing}
\tikzstyle{arrow} = [ultrathick,>=stealth]
\tikzstyle{block} = [rectangle, minimum width=3cm, minimum height=1cm, align=flush center, draw=black, thick]

\makeatletter
\newcommand\RSloop{\@ifnextchar\bgroup\RSloopa\RSloopb}
\makeatother
\newcommand\RSloopa[1]{\bgroup\RSloop#1\relax\egroup\RSloop}
\newcommand\RSloopb[1]%
  {\ifx\relax#1%
   \else
     \ifcsname RS:#1\endcsname
       \csname RS:#1\endcsname
     \else
       \GenericError{(RS)}{RS Error: operator #1 undefined}{}{}%
     \fi
   \expandafter\RSloop
   \fi
  }
\newcommand\X{0}
\newcommand\RS[1]%
  {\begin{tikzpicture}
     [every node/.style=
       {circle,draw,fill,minimum size=1.5pt,inner sep=0pt,outer sep=0pt},
      line cap=round
     ]
   \coordinate(\X) at (0,0);
   \RSloop{#1}\relax
   \end{tikzpicture}
  }
\newcommand\RSdef[1]{\expandafter\def\csname RS:#1\endcsname}
\newlength\RSu
\RSdef{i}{\draw (\X) -- +(90:\RSu) node{};}
\RSdef{l}{\draw (\X) -- +(110:\RSu) node{};}
\RSdef{r}{\draw (\X) -- +(70:\RSu) node{};}
\RSdef{I}{\draw (\X) -- +(90:\RSu) coordinate(\X I);\edef\X{\X I}}
\RSdef{L}{\draw (\X) -- +(120:\RSu) coordinate(\X L);\edef\X{\X L}}
\RSdef{R}{\draw (\X) -- +(60:\RSu) coordinate(\X R);\edef\X{\X R}}

\usepackage{color}
\usepackage[colorlinks=true]{hyperref}
\hypersetup{
    linkcolor=blue,          
    citecolor=red,        
    filecolor=blue,      
    urlcolor=cyan
}

\definecolor{darkergreen}{rgb}{0.0, 0.5, 0.0}

\numberwithin{equation}{section}

\newcommand{\be}{\begin{eqnarray}}
\newcommand{\ee}{\end{eqnarray}}
\newcommand{\ce}{\begin{eqnarray*}}
\newcommand{\de}{\end{eqnarray*}}
\newtheorem{theorem}{Theorem}[section]
\newtheorem{lemma}[theorem]{Lemma}
\newtheorem{remark}[theorem]{Remark}
\newtheorem{definition}[theorem]{Definition}
\newtheorem{proposition}[theorem]{Proposition}
\newtheorem{Examples}[theorem]{Example}
\newtheorem{corollary}[theorem]{Corollary}

\def\eps{\varepsilon}

\def\e{\mathrm{e}}

\def\p{\partial}

\def\[{{\Big[}}
\def\]{{\Big]}}
\def\<{{\langle}}
\def\>{{\rangle}}
\def\({{\Big(}}
\def\){{\Big)}}

\def\bx{{\mathbf{x}}}

\def\dif{{\mathord{{\rm d}}}}

\def\no{\nonumber}
\def\={&\!\!=\!\!&}

\def\bB{{\mathbf B}}
\def\bC{{\mathbf C}}

\def\cC{{\mathcal C}}

\def\cF{{\mathcal F}}

\def\cI{{\mathcal I}}
\def\cJ{{\mathcal J}}

\def\cP{{\mathcal P}}

\def\cR{{\mathcal R}}

\def\mA{{\mathbb A}}
\def\mB{{\mathbb B}}
\def\mC{{\mathbb C}}

\def\mE{{\mathbb E}}

\def\mH{{\mathbb H}}

\def\mL{{\mathbb L}}

\def\mN{{\mathbb N}}

\def\mP{{\mathbb P}}

\def\mR{{\mathbb R}}
\def\mS{{\mathbb S}}

\def\mX{{\mathbb X}}

\def\bB{{\mathbf B}}
\def\bP{{\mathbf P}}

\def\bE{{\mathbf E}}
\def\1{{\mathbf{1}}}

\def\sA{{\mathscr A}}
\def\sB{{\mathscr B}}

\def\sF{{\mathscr F}}

\def\sI{{\mathscr I}}
\def\sJ{{\mathscr J}}

\def\sL{{\mathscr L}}
\def\sM{{\mathscr M}}

\def\sP{{\mathscr P}}

\def\sS{{\mathscr S}}

\def\sW{{\mathscr W}}

\def\E{\mathbb E}

\def\geq{\geqslant}
\def\leq{\leqslant}
\def\ge{\geqslant}
\def\le{\leqslant}

\def\div{\mathord{{\rm div}}}

\def\eps{\varepsilon}

\def\e{\mathrm{e}}

\def\p{\partial}

\def\[{{\Big[}}
\def\]{{\Big]}}
\def\<{{\langle}}
\def\>{{\rangle}}
\def\({{\Big(}}
\def\){{\Big)}}

\def\bx{{\mathbf{x}}}

\def\dif{{\mathord{{\rm d}}}}

\def\no{\nonumber}
\def\={&\!\!=\!\!&}
\def\bt{\begin{theorem}}
\def\et{\end{theorem}}
\def\bl{\begin{lemma}}
\def\el{\end{lemma}}
\def\br{\begin{remark}}
\def\er{\end{remark}}
\def\bx{\begin{Examples}}
\def\ex{\end{Examples}}
\def\bd{\begin{definition}}
\def\ed{\end{definition}}
\def\bp{\begin{proposition}}
\def\ep{\end{proposition}}
\def\bc{\begin{corollary}}
\def\ec{\end{corollary}}

\def\geq{\geqslant}
\def\leq{\leqslant}
\def\ge{\geqslant}
\def\le{\leqslant}

\def\div{\mathord{{\rm div}}}

\def\bP{{\mathbf P}}

 \def\R{\mathbb R}
 \def\R{\mathbb R}    
  
\def\dd{\hh}   
\def\<{\langle} \def\>{\rangle}

\def\dd{{a}}

\def\rrho{{u}}
\def\mv{{\<u\>}}
\def\mvn{{\<u_n\>}}
\def\mvo{{\<u_2\>}}

\def\WW{{W}}

\newcommand{\Prec}{\prec\!\!\!\prec}
\allowdisplaybreaks

\begin{document}

\title{Singular kinetic equations and applications}

\author{Zimo Hao}
\address[Zimo Hao]{School of Mathematics and Statistics, Wuhan University, Wuhan,
Hubei 430072, P.R.China}
\email{zimohao@whu.edu.cn}

\author{Xicheng Zhang}
\address[X. Zhang]{School of Mathematics and Statistics, Wuhan University, Wuhan,
Hubei 430072, P.R.China
}
\email{XichengZhang@gmail.com}

\author{Rongchan Zhu}
\address[R. Zhu]{Department of Mathematics, Beijing Institute of Technology, Beijing 100081, China; Fakult\"at f\"ur Mathematik, Universit\"at Bielefeld, D-33501 Bielefeld, Germany}
\email{zhurongchan@126.com}

\author{Xiangchan Zhu}
\address[X. Zhu]{ Academy of Mathematics and Systems Science,
Chinese Academy of Sciences, Beijing 100190, China; Fakult\"at f\"ur Mathematik, Universit\"at Bielefeld, D-33501 Bielefeld, Germany}
\email{zhuxiangchan@126.com}
\thanks{
	X. Zhang is partially supported by NSFC (No. 11731009). R.Z. is grateful to the financial supports of the NSFC (No. 11922103).
	X.Z. is partially supported by National Key R\&D Program of China (No. 2020YFA0712700), the NSFC (No. 11771037, 12090014, 11688101) and the support by key Lab of Random Complex Structures and
	Data Science, Youth Innovation Promotion Association (2020003), Chinese Academy of Science.  The financial support by the DFG through the CRC 1283 “Taming uncertainty
	and profiting from randomness and low regularity in analysis, stochastics and their applications” is greatly acknowledged.
}

\begin{abstract}
In this paper we study singular kinetic equations on $\mR^{2d}$
	by the paracontrolled distribution
	method introduced in \cite{GIP15}.
We first develop paracontrolled calculus  in the kinetic setting, and  use it to
establish the global well-posedness for the linear  singular kinetic equations under the assumptions that the products of singular terms are well-defined. We also demonstrate how the required products can be defined in the case that singular term is a Gaussian random field by probabilistic calculation. Interestingly,  although the terms in the zeroth  Wiener chaos of regularization approximation are not zero, they converge in suitable weighted Besov spaces and no renormalization is required.
As applications  the global well-posedness for a nonlinear kinetic equation
with singular coefficients is obtained by the entropy method. Moreover, we also solve the martingale problem for nonlinear kinetic distribution dependent stochastic differential equations with singular drifts.

\end{abstract}

\subjclass[2010]{60H15; 35R60}
\keywords{}

\date{\today}

\maketitle

\tableofcontents

\section{Introduction}

In this paper we are concerned with the following nonlinear kinetic equation with singular drifts in $\mR^{2d}$:
\begin{equation}
\label{eq:1} \partial_t u = \Delta_v  u -v \cdot \nabla_x u -b\cdot\nabla_v u-K*\mv\cdot \nabla_v u, \quad u (0) = u_0,
\end{equation}
where $u:\mR_+\times\mR^{2d}\rightarrow \mR$ is a function of time variable $t$, position $x$ and velocity $v$,
$\mv(t,x):=\int_{\mR^{d}}u(t,x,v)\dif v$ stands for the mass, $K: \mR^d\to\mR^d$ is a kernel function and
$$
K*\mv(t,x):=\int_{\mR^d}K(x-y)\mv(t,y)\dif y,
$$
and for some $\alpha\in(\frac12,\frac23)$ and $T>0$,
\begin{align}\label{DD7}
b=(b_1,\cdots,b_d)\in (L_T^\infty\bC^{-\alpha}_a(\rho))^d,
\end{align}
 is a Gaussian random field and the example of $b$ which we have in mind is white noise in $v$ and colored in $x$.
Here $\rho$ is a polynomial weight and $\bC^{-\alpha}_a(\rho)$ stands for the weighted anistrophic H\"older space
introduced in Subsection \ref{sec:2.1}. The aim of this paper is to establish the well-posedness for the above singular SPDE and the associated distributional dependent SDEs (see \eqref{ksde1} below) under suitable assumptions. In  Subsection \ref{sec:1.1} we state the main results under suitable analytic assumptions, which could be verified by probabilistic assumptions on the covariance of $b$ in Section \ref{Sub6}.


The kinetic equation was originally introduced by Landau in 1936 to study the plasma phenomenon in physics,
which is a nonlinear PDE with square and nonlocal second order term
(see \cite{Lan36}, \cite{AV04} and references therein).  As  model equations, we consider the following two linear kinetic equations
\begin{align}\label{DX9}
\sL u:=(\p_t{-}v\cdot\nabla_x-\Delta_v)u=f,
\end{align}
\begin{align*}
\sL^*u=(\p_t+v\cdot\nabla_x-\Delta_v)u=f.
\end{align*}
where 
 $\sL$ is also called Kolmogorov operator since  in \cite{Kol34}, he first wrote down the fundamental solution of $\sL$ (see \eqref{PPT} below).
These two  equations have the following relation:
\begin{align*}
\tau\sL u=\sL^*(\tau u),\quad \tau u(t,x,v)=u(t,x,-v)
\end{align*}
and transform $\tau$ influences nothing in our formulation. 

Now we consider the following scaling transform: for $\lambda>0$ and $a,b,c>0$, let
\begin{align*}
u_\lambda(t,x,v):=\lambda^a u(\lambda^b t,\lambda^cx,\lambda v),\quad f_\lambda(t,x,v):=f(\lambda^b t,\lambda^cx,\lambda v).
\end{align*}
It is easy to check that
 \begin{align}\label{BA1}
\sL u_\lambda=f_\lambda\Longleftrightarrow a=-2, b=2, c=3.
\end{align}
Next we consider the improvement of the regularities in $x$ and $v$ for \eqref{DX9}. Suppose that
for some $\alpha\in(0,1)$ and $\beta,\gamma>0$, there is a constant $C>0$ such that for all $\lambda>0$,
\begin{align}\label{DB0}
[u_\lambda]_{\bC^{\alpha+\gamma}_x}\lesssim_C[f_\lambda]_{\bC_x^{\alpha}}, \ \
[u_\lambda]_{\bC^{\alpha+\beta}_v}\lesssim_C[f_\lambda]_{\bC_v^{\alpha}},
\end{align}
where for any $\gamma>0$,
\begin{align*}
[g]_{\bC_x^{\gamma}}:=\sup_{h\in\mR^d}\|\delta_{x;h}^{([\gamma]+1)}g\|_\infty/|h|^{\gamma}
\end{align*}
with $\delta_{x;h}g(x,v):=g(x+h,v)-g(x,v)$ {and $\delta_{x;h}^{(M+1)}=\delta_{x;h}\delta_{x;h}^{(M)}$}, similarly for $[g]_{\bC_v^{\beta}}$.
Note that
\begin{align*}
[u_\lambda]_{\bC^{\alpha+\gamma}_x}=\lambda^{3(\alpha+\gamma)-2}[u]_{\bC^{\alpha+\gamma}_x},\quad
[f_\lambda]_{\bC_x^{\alpha}}=\lambda^{3\alpha}[u]_{\bC^{\alpha}_x},
\end{align*}
and
\begin{align*}
[u_\lambda]_{\bC^{\alpha+\beta}_v}=\lambda^{\alpha+\beta-2}[u]_{\bC^{\alpha+\beta}_v},\quad
[f_\lambda]_{\bC_v^{\alpha}}=\lambda^{\alpha}[u]_{\bC^{\alpha}_v}.
\end{align*}
Under scaling invariant \eqref{DB0}, we must have
$$
\gamma=2/3,\ \ \beta=2.
$$
In other words, the gains of the regularities for kinetic equation \eqref{DX9} in $x$ and $v$ are $\frac23$ and $2$, respectively.
Thus the following Schauder's estimate is expected:
for any $\alpha,\beta>0$, there is a constant $C=C(\alpha,\beta,d)>0$ such that
\begin{align}\label{GH4}
\|u\|_{L^\infty_T\bC^{\alpha+2/3}_x}+\|u\|_{L^\infty_T\bC^{\beta+2}_v}\lesssim_C \|f\|_{L^\infty_T\bC^{\alpha}_x}+\|f\|_{L^\infty_T\bC^{\beta}_v},
\end{align}
where $\bC^\alpha_x$ and $\bC^\beta_v$ stand for the H\"older spaces in directions $x$ and $v$, respectively.
{Due to different scaling and regularity between $x$ and $v$ variables, we study \eqref{eq:1} in the anistrophic H\"older space (see Subsection \ref{sec:2.1} for definition).


When $\alpha=\beta/3>0$, Schauder's estimate \eqref{GH4} has been studied extensively
in \cite{Lo05}, \cite{Pr09} (see \cite{HWZ20}, \cite{IS21} for nonlocal version),
and the maximal $L^p$-regularity estimates were obtained in \cite{Bo02} (see also \cite{CZ18}, \cite{HM19}
and \cite{ZZ21} for stochastic version).
We mention that the  structure of Lie group was introduced to {define the kinetic H\"older spaces for}  the Schauder estimates 
in \cite{IS21} (see also earlier work \cite{Po04}).
In the current work, 
we introduce the kinetic H\"older space, {which is equivalent to the one introduced in \cite{IS21},} without using the notion of Lie group.

\medskip
One motivation for studying kinetic equation \eqref{eq:1} with distribution valued coefficient $b$
is to develop solution theory for degenerate singular SPDEs. When $\alpha>\frac12$, due to
the singularity of the coefficients $b$ in \eqref{DD7},  the best regularity of the solution to \eqref{eq:1} is
in $L_T^\infty\bC^{2-\alpha}_a$, which makes the linear term $b\cdot \nabla_v u$ not well defined in the classical sense.
Such kind of problems also arise in the understanding of singular SPDEs, such as famous KPZ equations \cite{KPZ86}, which have been intensely studied recently.
{Hairer in \cite{Hai14} developed the regularity structure theory to give a meaning to a large class of  singular SPDEs.
Parallel to that,} a paracontrolled distribution method was proposed by Gubinelli, Imkeller and Perkowski \cite{GIP15},
which is also a powerful tool for studying singular SPDEs. The key idea
of these theories are to use the structure of solutions to give a meaning to the
terms which are not classically defined. These terms are well-defined with the
help of probabilistic calculation and renormalization for the ``enhanced noise", i.e. the noise and the higher
order terms appearing in the decomposition of the equations. Based on these idea the solution theories for quasilinear parabolic singular SPDEs, Schr\"odinger and wave equations driven by singular noise have been developed in \cite{OW19, OSSW18, GH19, OSSW21} and \cite{DW18, GKO18, GKO18a} (see also the references therein).
In this paper we aim to develop paracontrolled distribution calculus for the degenerate kinetic SPDEs  with singular coefficients.


Going back to kinetic equation \eqref{eq:1}, it is natural to work on the whole space since the velocity $v$ physically takes
values in the whole space, where the coefficients $b$, which come from the noise and the renormalized terms, stay in the weighted Besov spaces. This
prevents us from using a fixed point argument in the same space.  To the best of our knowledge,
there are two methods to solve this problem. One is to use a clever construction of exponential weight
depending on time variable proposed  in \cite{HL18}. The other one is to use localization trick developed in \cite{ZZZ20}.
In this paper we follow the localization method in \cite{ZZZ20} to solve this problem.
We deduce  a priori estimates for \eqref{eq:1} and by a compactness argument  obtain the existence of solutions.
The localization argument also implies uniqueness. {We refer  to Section \ref{sec:1.2} for more details on the idea of the proof.}
Compared to the local solutions for singular SPDEs mentioned above, a priori estimates and the global well-posedness  for different parabolic singular SPDEs have been obtained, see \cite{MW17, MW17a, GH18} for the dynamical $\Phi^4_d$-model and \cite{PR18, ZZZ20} for the KPZ equations and singular HJB equations.

\medskip

Another motivation is that equation \eqref{eq:1} can be viewed as the mean field limit of empirical measures for a second order
interacting particle system in random environment. More precisely, consider the following $N$-interacting particle system in $\mR^d$,
where each particle obeys the Newtonian second law perturbed by time Gaussian noise $\dot B_t^i$ and environment noise $W$:
$$
\ddot X_t^{N,i}=W(X_t^{N,i},\dot X^{N,i}_t)+\frac{1}{N}\sum_{j\ne i}K(X_t^{N,i}-X_t^{N,j})+{\sqrt{2}}\dot B_t^i,\ i=1,\cdots,N,
$$
where $(B^i_t)_{i\in\mN}$ is a sequence of $d$-dimensional independent standard Brownian motions on a
stochastic basis $(\Omega,\cF, \bP; (\cF_t)_{t\geq 0})$, $K:\mR^d\to\mR^d$ is the interaction kernel, and $W:\mR^{2d}\to\mR^d$
is a vector-valued distribution and stands for the environmental noise, which acts on all particles. We will see in Section \ref{Sub6} that our condition on $W$ allows for spatial white noise in $v$ direction for $d=1$,
	which may be derived from average of a sequence of i.i.d random variables (see e.g. \cite[Remark 2.2]{PR19}).
The factor $\frac{1}{N}$ in front of the interacting force $K$ is called mean-field scaling which keeps the total mass of order 1.
If we introduce a new velocity variable $V^{N,i}_t:=\dot X^{N,i}_t$ and let $Z^{N,i}_t:=(X_t^{N,i}, V_t^{N,i})$,
then the above second order SDE can be written as the familiar form:
\begin{align}\label{Intro1}
\begin{cases}
\dif X_t^{N,i}=V_t^{N,i}\dif t,\qquad \qquad i=1,2,\cdots,N,\\
\dif V_t^{N,i}=\Big[W(Z_t^{N,i})+\frac{1}{N}\sum_{j\ne i}K(X_t^{N,i}-X_t^{N,j})\Big]\dif t+{\sqrt{2}}\dif B_t^i.
\end{cases}
\end{align}
On the other hand, for each $i\in\mN$, consider the following kinetic distributional dependent SDE (abbreviated as DDSDE)
\begin{align}\label{ksde1}
\dif \bar X^i_t=\bar V^i_t\dif t,\ \
\dif \bar V^i_t=W(\bar Z^i_t)\dif t+(K*\mu_{\bar X^i_t})(\bar X^i_t)\dif t+\sqrt{2}\dif B^i_t,
\end{align}
where $\bar Z^i_t:=(\bar X^i_t,\bar V^i_t)$ and for a probability measure $\mu$ in $\mR^d$,
$$
K*\mu(x):=\int_{\mR^d}K(x-y)\mu(\dif y).
$$
When $W$  and $K$ are global Lipschitz, it is well-known that there are unique solutions to \eqref{Intro1} and \eqref{ksde1},
and the following  propagation of chaos holds (see \cite[Theorem 1.4]{Szn91}): Suppose $Z^{N,i}_0=\bar Z^i_0$ and $\{\bar Z^i_0\}$ are i.i.d. random variables. Then for each $i\in\mN$ and $T>0$,
\begin{align}\label{AB7}
\sup_N\sqrt{N}\bE\left(\sup_{t\in[0,T]}|Z_t^{N,i}-\bar Z_t^{i}|\right)<\infty.
\end{align}
Note that  $(\bar Z^i_\cdot)_{i\in\mN}$ are i.i.d. random processes. 
Let $\mu=(\mu(t))_{t\geq 0}$ be the distribution of
$(\bar Z^i_\cdot)_{i\in\mN}$. 
By It\^o's formula, one sees that $\mu(t)$ solves
the following non-linear Fokker-Planck equation: for any $\phi\in C^2_b(\mR^{2d})$,
\begin{align}\label{Intro3}
\p_t\<\mu,\phi\>=&\<\mu,\Delta_v\phi+v\cdot \nabla_x\phi+(W+K*\<\mu\>)\cdot\nabla_v\phi\>.
\end{align}
With a little confusion of notation with $\<\mu\>$, we also write
$$
\<\mu,\phi\>:=\int_{\mR^{2d}}\phi(z)\mu(\dif z).
$$
Now, let $u_N(t):=\frac1N\sum_{i=1}^N\delta_{Z_t^{N,i}}$ be the empirical distribution measure.
By \eqref{Intro1} and It\^{o}'s formula again, one finds that for any $\phi\in C^2_b(\mR^{2d})$,
\begin{align}\label{MV1}
\dif\<u_N,\phi\>&=\<u_N, \Delta_v\phi+v\cdot \nabla_x\phi+(W+K*\<u_N\>)\cdot\nabla_v\phi\>\dif t+\frac{\sqrt{2}}{N}\sum_{i=1}^N\nabla_v\phi\big(Z_t^{N,i}\big)\dif B_t^i.
\end{align}
In particular,  each term in \eqref{MV1} converges to the corresponding one in \eqref{Intro3} in suitable sense. For examples,
by It\^{o}'s isometry, we have
\begin{align*}
\bE\left|\frac{1}{N}\sum_{i=1}^N\int^t_0\nabla_v\phi\big(Z_s^{N,i}\big)\dif B_s^i\right|^2
=\frac{1}{N^2}\sum_{i=1}^N\bE\int^t_0\big|\nabla_v\phi\big(Z_s^{N,i}\big)\big|^2\dif s
\leq\frac{t\|\nabla_v\phi\|_\infty}{N}\to 0
\end{align*}
Note that if $W,K\in C^\infty_b$, then $\mu(t)$ has a smooth density $u(t,z)$ so that
\begin{align}\label{KFP}
\p_t u=\Delta_v u-v\cdot \nabla_x u-\div_v\big((W+K*\<u\>)u\big).
\end{align}
We also mention that when $W$ depends on the random environment $\omega$, the empirical measure $u_N$ also converges to the solution to equation \eqref{KFP}, which corresponds to the conditional law of $\bar{Z}$ w.r.t. $W$ with $\mu_{\bar{X}}$ in \eqref{ksde1} also given by conditional law. This means that  the conditional propagation of chaos holds (see \cite{CF16} for more details). In particular, if $\div_v W\equiv 0$, then the above equation reduces to the form of \eqref{eq:1}. In physics this assumption is natural which is satisfied if the force only depends on the position. We refer to Section \ref{sec:1.3} for more background, more references in this direction.
In the following we regards \eqref{KFP} and \eqref{eq:1} as random PDEs, i.e. we fix the a.s. path of $W$, and solve the SPDE path by path. 

\subsection{Main results}\label{sec:1.1}

The main goal of this article is to give a meaning to  the kinetic equation \eqref{eq:1} and establish the global well-posedness of \eqref{eq:1}
under \eqref{DD7}.
As mentioned above, since $b\cdot\nabla_v u$ does not make sense, we need to use paracontrolled method and perform renormalizations by probabilistic calculations to give a rigorous meaning to $b\cdot\nabla_v u$.

First of all, we consider the following linear PDE with distributions $b,f$:
\begin{equation}
\label{eq:li} \partial_t u = \Delta_v  u +v \cdot \nabla_x u + b\cdot\nabla_v u+f, \quad u (0) = u_0.
\end{equation}

To state our main results, we first introduce some parameters and notations.
Let $\vartheta:=\frac{9}{2-3\alpha}$ for some $\alpha\in(1/2,2/3)$. For given $\kappa_0<0$,
$\kappa_1\in(0,\frac1{2\vartheta+2}]$, $\kappa_2\in\mR$ and $\kappa_3:=(2\vartheta+1)\kappa_1+\kappa_2$,
in the statement of our main results below, we shall use the following weight functions:
$$
\rho_i(x,v):=(1+|x|^{1/3}+|v|)^{-\kappa_i},\ i=0,1,2,3.
$$
Let $\mB_T^\alpha(\rho_1,\rho_2)$ be the space of renormalized pairs and $\mB_T^\alpha(\rho_1)$ the space of
renormalized vector fields introduced in  Definition \ref{Def216}. Formally  $(b,f)\in \mB_T^\alpha(\rho_1,\rho_2)$ and $b\in\mB_T^\alpha(\rho_1)$ mean $b,f\in L_T^\infty(\bC^{-\alpha}_a(\rho_1))$ and for $\sI=\sL^{-1}$, $b\circ\nabla_v\sI f\in L_T^\infty\bC^{1-2\alpha}_a(\rho_1\rho_2)$, $b\circ\nabla_v\sI b\in L_T^\infty\bC^{1-2\alpha}_a(\rho_1^2),$ are well-defined respectively, which in general could be realized by a probabilistic calculation. Here $\circ$ is the paraproduct introduced in Subsection \ref{ssec:para}. The example we have in mind is a Gaussian random forcing and our assumption allow, for example, when $d=1$, $b$ to be white in $v$ variable and colored in $x$ variable. Compared to the heat semigroup, the interesting point is that the terms in the zeroth  Wiener chaos are not zero and converge in the corresponding weighted Besov space. In fact, the terms in the zeroth  Wiener chaos minus  formally divergence terms which by symmetry  are zero will converge. Hence no renormalization  appears in the smooth approximation  of equation \eqref{eq:1}.

The following result provides the well-posedness of the linear singular PDE \eqref{eq:li}.
 \bt\label{th:main1}
Suppose that $(b,f)\in \mB_T^\alpha(\rho_1,\rho_2)$ and $b\in \mB_T^\alpha(\rho_1)$.
 For any $T>0$ and $\varphi\in\bC^{\gamma}_{a}(\rho_2/\rho_1)$, where $\gamma>1+\alpha$,
 there is a unique paracontrolled solution $u\in \mS^{2-\alpha}_{T,a}(\rho_3)$
 to PDE \eqref{eq:li} in the sense of Definition \ref{def:para1}, where $\mS^{2-\alpha}_{T,a}(\rho_3)$
 is the kinetic H\"older space introduced in Definition \ref{kinetics}.
 \et

In Section \ref{ParaA} we prove this result. Along the way to Theorem \ref{th:main1}, we develop paracontrolled calculus in the kinetic setting and prove a commutator estimate for the kinetic semigroup. We   refer to Section \ref{sec:1.2} for more details on this point. The complete version of  Theorem \ref{th:main1} is given in  Theorem \ref{Th33}.


Next we consider the nonlinear kinetic Fokker-Planck equation \eqref{KFP}.

\bt\label{th:main2} Let $T>0$. Suppose that $W\in \mB_T^\alpha(\rho_1)$ with $\div_v W=0$ and $K\in \bC^{\beta/3}(\mR^d)$ with $\beta>\alpha-1$.
 For $\gamma>1+\alpha$, and any probability density function $u_0$ with $u_0\in L^1(\rho_0)\cap\bC^\gamma_a$ and $T>0$
there exists at least a probability density paracontrolled solution $u\in\mS^{2-\alpha}_{T,a}(\rho_3)$ to equation \eqref{KFP}.

If in addition that $K$ is bounded, then for any initial data $u_0\in L^1(\rho_0)\cap\bC^\gamma_a$ with
$\int u_0\ln u_0<\infty$, the solution is unique.
\et



The complete version of Theorem \ref{th:main2} is given in Theorem \ref{thm72}.
By $\div_v W\equiv0$ we can write \eqref{KFP} in the non-divergence form
$$
\p_t u=\Delta_v u-v\cdot \nabla_x u-(W+K*\<u\>)\cdot\nabla_vu
$$
and Theorem \ref{th:main1} can be applied. As mentioned above the solution to \eqref{KFP} can be viewed as a probability density. Hence in this paper we concentrate on such kind of solutions. Formally from the equation we see the integral of solution is a constant. Also if the initial value is nonnegative, then a maximum principle implies the solution is always nonnegative.
As usual, the key point to prove this theorem is to establish
the a priori estimates \eqref{Rev4} and \eqref{Rev5} in Section \ref{nonlinear} about entropy.
Compared with the previous work in \cite{JW16}, our assumptions is more flexible.
We refer to Section \ref{sec:1.2} for details on the idea of the proof.

Finally, as an application we also obtain the well-posedness for  the associated nonlinear martingale problem of \eqref{ksde1}.
\bt\label{th:Main8}
Let $T>0$. Suppose that $W\in \mB_T^\alpha(\rho_1)$ and $K\in \bC^{\beta/3}(\mR^d)$ with $\beta>\alpha-1$. For any initial probability distribution $\nu$ with finite moment $\int|z|_a^\delta\nu(\dif z)<\infty$, where  $\delta>\frac{(4\vartheta+4)\kappa_1}{2-\alpha}$, there exists a  martingale solution to nonlinear SDE \eqref{ksde1} starting from $\nu$.
Moreover, if $K$ is bounded measurable, the solution is unique.
\et

The complete version of Theorem \ref{th:Main8} is given in Theorem \ref{Main8}.
Our martingale problem is considered in the sense of Either and Kurtz \cite[Section 4.3, p173]{EK86}, which is a  general notion.
The usual martingale problem is that for all functions $u$ in the domain of generator $\sL^\mu:=\Delta_v+v\cdot\nabla_x+(b+K*\mu_t)\cdot \nabla_v$, the process $u(t,X_t,V_t)-u(0,x,v)-\int_0^t(\partial_t+ \sL^\mu)u(s, X_s,V_s)ds$ with $\mu_t=\textrm{Law} (X_t)$ is a martingale. However, due to singularity of $b$, smooth function might be not in the domain of $\sL_\mu$. We can find such $u$ by solving the Kolmogorov backward equation. We refer to Section \ref{sec:1.2} for more details on this point. This type of martingale problem has been treated in \cite{DD16, CC18, KP20} for linear non-degenerated singular SDEs.
To the best of our knowledge,  this is the first well-posedness result for singular degenerate nonlinear SDEs.

\subsection{Sketch of proofs and structure of the paper}\label{sec:1.2}
In Section \ref{section2}, we recall some facts about the anisotropic weighted Besov spaces and the associated paracontrolled calculus.
In particular, a quite useful characterization of anisotropic weighted Besov spaces is stated in Theorem \ref{Th26}, whose proof is given
in Appendix \ref{AnWB}.


For the kinetic semigroup, we introduce a new weighted kinetic H\"older space associated with
the transport term $v\cdot\nabla_x$ (see  Definition \ref{kinetics}).  
On this space, Schauder's estimate for the kinetic semigroup is established
(see Lemma \ref{Le11}).
The key point to use the paracontrolled calculus for the kinetic equation \eqref{eq:1} is a commutator estimate for the kinetic semigroup which we establish in Subsection \ref{sec:commutator}. Note that it seems  impossible to show a commutator estimate in the form $[\sL_v,f\prec]g$ as in \cite{GIP15} for $\sL_v:=\Delta_v+ v\cdot\nabla_x$, since the loss of regularity from $\sL_v$ and the gain of regularity from the kinetic semigroup do not match  i.e.   the kinetic operator  loses  $1$ regularity in $x$ direction while the Schauder estimate for the kinetic semigroup only gains $2/3$ regularity in $x$ direction. Moreover, the commutator  for the kinetic semigroup under the action of block operator $\cR^a_j$ is not like the heat semigroup and there is an extra transport term left, which leads to a commutator estimate  in the kinetic H\"{o}lder space  introduced in Definition \ref{kinetics} (see Lemma \ref{commutator1}). We refer the readers to the argument at the beginning of Section \ref{sec:commutator} for more details on this point.
In Subsection \ref{sec:2.5}, we give the notion of renormalized pairs as in \cite{ZZZ20} as mentioned in Subsection \ref{sec:1.1}.

Sections \ref{ParaA} and \ref{nonlinear} are devoted to well-posedness of equations \eqref{eq:li} and \eqref{KFP}. We first use paracontrolled calculus in the kinetic setting, characterization of the weighted H\"older space
and localization trick developed in \cite{ZZZ20} to derive uniform bounds in a
polynomial growth weighted Besov space for the solutions to the linear equation \eqref{eq:li}.
 	The new point is that we
 prove a localization result for paracontrolled solution (see Proposition \ref{Pr42}).
 This localization property allows us to establish  a priori estimate \eqref{MN1} for {\it any} paracontrolled solution of \eqref{eq:li},
 which automatically yields the uniqueness. Note that the proof of the uniqueness in \cite{ZZZ20} is to adopt the exponential weight technique
 developed in \cite{HL18}. 
  For the nonlinear equation 
  mentioned above, we concentrate on probability density solutions. In this case, to prove existence of solutions and the convergence of the nonlinear term in \eqref{eq:1}, we need to show the convergence of the approximation solutions in $L^1$-space, which follows from a moment estimate for some SDEs by a probabilistic method. Usually people obtained such kind moment estimate for distributional drift SDE by using the Zvonkin transform to kill the singular drift term (see e.g. \cite{ZZ18}). However,  the required $\bC^1$-diffeomorphism in  Zvonkin transform cannot be constructed since in $x$-direction the regularity cannot be  $\bC^1$. In Section \ref{nonlinear} we use Theorem \ref{Th33} to deduce a Krylov type estimate,  which can be used to control the distribution drift (see Lemma \ref{lem:priori}). The uniqueness proof follows from a priori entropy estimate  and $L^1$-estimate. To deal with the distributional drift term, we use  linear approximations and Theorem \ref{Th33}.

In Section \ref{kddsde} we consider the martingale problem associated with \eqref{ksde1} and establish the well-posedness. As mentioned in Subsection \ref{sec:1.1}, we solve this martingale problem by analyzing the Kolomogorov backward equation. Since this is a nonlinear martingale problem, the corresponding Kolomogorov equation should be nonlinear. However,  it is not known a-priori that the law of the solutions to \eqref{ksde1} is absolutely continuous w.r.t. Lebesgue measure.  As a result, we consider the linear equation for fixed $\mu$ and we can apply Theorem \ref{th:main1} directly. More precisely, we consider the following equation for fixed $\mu:[0,T]\rightarrow \mathcal{P}(\mathbb{R}^d)$:
\begin{align}\label{eq:kolomogorov}\partial_t u+\sL^\mu u=f,\quad u(T)=u^T,\end{align}
for a sufficiently large class of functions $f$ and $u^T$, and therefore we replace the martingale problem with the requirement that the process $u(t,X_t,V_t)-u(0,x,v)-\int_0^tf(s, X_s,V_s)\dif s$ with $\mu_t=\textrm{Law} (X_t)$ is a martingale.
For the existence of a martingale solution,
we use the standard tightness argument. Moreover, to obtain the convergence, we prove  the continuity of the nonlinear term 
(see Lemma \ref{Le65}).
For the uniqueness of martingale solutions, we first show the uniqueness of the solutions to the linear equations (i.e. $K\equiv 0$),
and then use Girsanov's transformation and Gronwall's inequality. 

Section \ref{Sub6} is concerned with the probabilistic analysis connected to the construction of the stochastic objects needed in the sequel. More precisely, we consider a class of stationary Gaussian distributions $X$ of class $\bC_a^{-\alpha}(\rho_\kappa)$. This class includes one dimensional spatial white noise in $v$ direction and colored in $x$ direction; any covariance operator {$|\partial_x|^{-\lambda}$} with $\lambda>5/9$ when $d=1$ is admissible. For such $X$ we construct the generalized products $\nabla_v \sI X\circ X$  as probabilistic limits of smooth approximations. Some proofs used in Section \ref{Sub6} are put in Appendix \ref{sec:app}. 

\subsection{Further relevant literature}\label{sec:1.3}

	
	The study of mean field limit and propagation of chaos for interacting particle system originated from McKean \cite{McK67}, see for instance the classical reference \cite{Szn91}. 	As mentioned above, DDSDE which is also called McKean-Vlasov equation is closely related to mean filed limit.
	To the best of our knowledge,  Vlasov \cite{Vl68} first proposed
	McKean-Vlasov's equations, which  arise in many applications, such as multi-agent systems (see \cite{BRTV98,BT97}),  filtering (see \cite{CX10}) and so on.
	Recently the research on the mean field limit for the $1$st order system, with singular interaction kernels has experienced immense
improvements  including those results
focusing on the vortex model \cite{Osa86,FHM14} and  more general singular kernels as in
\cite{JW18} and Serfaty \cite{Ser20b}.
 When $W\equiv0$ and $K(x)\in L^\infty(\mR^d)$, Jabin and Wang \cite{JW16} studied the well-posedness of
PDE \eqref{Intro3} and propagation of chaos.  In the pioneering work by Funaki \cite{Fu84}  the martingale problem for a  non-linear PDE is clearly formulated. 
After that global well-posedness of DDSDE  has been studied a lot in the literature (see \cite{MV16}  \cite{Wa18} \cite{RZ21} and references therein). In the case where there is a common environmental noise influencing each particles, this suggests  particle systems with common noise like \eqref{Intro1}
and there are also a lot of work concerning the mean field limit of  particle systems with common noise and the limiting  DDSDE  (see e.g. \cite{CF16, R20, HSS21} and reference therein). However, so far as we know, most work concentrate on the first order system, which is related to a parabolic SPDEs, and the related common noise $W$ is trace-class type noise, i.e. the  noise $W$ is function valued w.r.t. spatial variable.

In many applications such as control problems and Coulomb potential from physics, the coefficients for the related DDSDE are very singular. Hence, studying the nonlinear kinetic equation and DDSDE with singular coefficients counts for much.
		In the present paper, we can obtain global well-posedness for these nonlinear equations with  singular environmental noise $W$, which so far as we know, has not been obtained in the literature. In this paper we do not 
show the propagation of chaos like
\eqref{AB7} when environmental noise distribution $W$ is allowed.
This will be studied in future work.




The study of SDEs with distributional drifts 
has also attracted much interest in recent years (see \cite{DD16, ZZ18, CC18, KP20} etc.).
Such singular diffusions arise as models for stochastic processes in random media. When $d=1$, based on the rough path method, Delarue and Dielthe \cite{DD16} studied the  SDE with rough drift. In \cite{CC18},
based on the theory of paracontrolled calculus, Cannizzaro and Chouk proved the well-posedness for the martingale problem with singular drift in higher dimensions (see also \cite{KP20} when Brownian motion is replaced by $\alpha$-stable processes).
For the second order system \eqref{ksde1},
{to the best our knowledge,} there is no such kind of result. Finally,  we also mention that when $K\equiv 0$, the strong and weak well-posedness of SDE \eqref{ksde1} with H\"older drift $W$ was studied in \cite{Ch17}
\cite{WZ16} and \cite{Zh18}.

\subsection{Notations and conventions}\label{sec:1.4}

Throughout this paper, we use $C$ or $c$ with or without subscripts to denote an unrelated constant, whose value
may change in different places. We also use $:=$ as a way of definition. By $A\lesssim_C B$ and $A\asymp_C B$
or simply $A\lesssim B$ and $A\asymp B$, we mean that for some unimportant constant $C\geq 1$,
$$
A\leq C B,\ \ C^{-1} B\leq A\leq CB.
$$
For convenience, we  collect some commonly used notations and definitions below.
$$
\begin{tabular}{c|c}\toprule
$\bB^{s,a}_{p,q}(\rho)$: weighted Besov space (Def. \ref{Def25}) & $\bB^{s,a}_{p,q}:=\bB^{s,a}_{p,q}(1)$\\ \midrule
$\bC^s_a(\rho):=\bB^{s,a}_{\infty,\infty}(\rho)$ , $\mC^s_{T,a}(\rho):=L^\infty([0,T];\bC^s_a(\rho))$ &  $\bC^s_a:=\bC^s_a(1)$ \\ \midrule
$\mS^\alpha_{T,a}(\rho)$: Kinetic H\"older space \eqref{SS0}  &$\mS^\alpha_{T,a}:=\mS^\alpha_{T,a}(1)$ \\ \midrule
$\mB^\alpha_{T}(\rho)$: Space of renormalized pair (Def. \ref{Def216}) & $\mB^\alpha_{T}:=\mB^\alpha_{T}(1)$\\ \midrule
\bottomrule
\end{tabular}
$$
$$
\begin{tabular}{c|c}\toprule
$f\prec g, f\succ g, f\circ g$: Paraproduct (Sec. \ref{ssec:para})  &  $f\succcurlyeq g:=f\succ g+f\circ g$ \\ \midrule
${\rm com}(f,g,h):=(f\prec g)\circ h-f(g\circ h)$ (Sec. \ref{ssec:para})  &
$\sL_\lambda:=\p_t-v\cdot\nabla_x-\Delta_v+\lambda$
\\\midrule
$\Gamma_tz:=(x+tv,v)$, $\Gamma_t f(z):=f(\Gamma_t z)$ & $\sI_\lambda:=\sL_\lambda^{-1}$\\ \midrule
$P_tf=\Gamma_tp_t*\Gamma_tf=\Gamma_t(p_t*f)$: Kinetic semigroup & $B^a_r:=\{x:|x|_a\leq r\}$\\ \midrule
$\varrho(x,v):=((1+|x|^2)^{1/3}+1+|v|^2)^{-1/2}$& $\sP_{\rm w}:=\{\varrho^\kappa, \kappa\in\mR\}$\\ \midrule
Commutator: $[\sA_1,\sA_2]f:=\sA_1(\sA_2 f)-\sA_2(\sA_1f)$ & $\mN_0:=\mN\cup\{0\}$ \\\midrule
{$\delta_hf(x):=f(x+h)-f(x)
$} & {$\delta_h^{(k)}:=\delta_h\delta_h^{(k-1)}$ }\\\midrule
\bottomrule
\end{tabular}
$$

\section{Preliminaries}\label{section2}

In this section we introduce the basic notations and recall various preliminary results concerning weighted anisotropic
Besov spaces (see \cite{Di96}, \cite{Tri06}).
Since the precise results that we need are difficult to locate in the literature, and for the readers' convenience,
we give some details of
the proofs in Subsection \ref{sec:2.1}.
In Subsection \ref{ssec:para} we present paraproduct calculus on the anisotropic
Besov spaces which follows in the same way as the classical argument.

Throughout this section we fix $N\in\mN$. Let $\sS({ \mR^N})$ be the Schwartz space of all rapidly decreasing functions on ${ \mR^N}$, and $\sS'({ \mR^N})$
the dual space of $\sS({ \mR^N})$ called Schwartz generalized function (or tempered distribution) space. Given $f\in\sS({ \mR^N})$,
the Fourier transform $\hat f$ and inverse Fourier transform $\check f$ are defined, respectively, by
\begin{align*}
\hat f(\xi)&:=\frac{1}{(2\pi)^{N/2}}\int_{{ \mR^N}} \e^{-{\rm i}\xi\cdot x}f(x)\dif x, \quad\xi\in{ \mR^N},\\
\check f(x)&:=\frac{1}{(2\pi)^{N/2}}\int_{{ \mR^N}} \e^{{\rm i}\xi\cdot x}f(\xi)\dif\xi, \quad x\in{ \mR^N}.
\end{align*}
Fix $n\in\mN$. Let $m=(m_1,\cdots,m_n)\in\mN^n$ with $m_1+\cdots+m_n=N$ and $a=(a_1,\cdots,a_n)\in[1,\infty)^n$ be also fixed.
We introduce the following distance in ${ \mR^N}$ by
$$
|x-y|_a:=\sum_{i=1}^n|x_i-y_i|^{1/a_i},\ x_i,y_i\in\mR^{m_i},
$$
where $|\cdot|$ denotes the Euclidean norm in $\mR^{m_i}$.
For $x=(x_1,\cdots, x_n)$, $t>0$ and $s\in\mR$, we denote
\begin{align}\label{ND0}
t^{s a} x:=(t^{s a_1}x_1,\cdots, t^{s a_n}x_n)\in{ \mR^N},\ \ B^a_t:=\Big\{x\in{ \mR^N}: |x|_a\leq t\Big\}.
\end{align}
Clearly we have
\begin{align}\label{ND90}
|t^a x|_a=t|x|_a,\ \ t\geq 0.
\end{align}
\subsection{Weighted anisotropic Besov spaces}\label{sec:2.1}
To introduce the anisotropic Besov space, we need a symmetric
nonnegative $C^{\infty}-$function $\phi^a_{-1}$ on $\mathbb{R}^N$ with
$$
\phi^a_{-1}(\xi)=1\ \mathrm{for}\ \xi\in B^{a}_{1/2}\ \mathrm{and}\ \phi^a_{-1}(\xi)=0\ \mathrm{for}\ \xi\notin B^{a}_{2/3}.
$$
For $\xi=(\xi_1,\cdots,\xi_n)\in\mathbb{R}^{m_1}\times\cdots\times\mathbb{R}^{m_n}$ and $j\geq 0$, we define
\begin{align}\label{Phj}
\phi^a_j(\xi):=\phi^a_{-1}(2^{-a (j+1)}\xi)-\phi^a_{-1}(2^{-aj}\xi).
\end{align}
By 
definition, one sees that for $j\geq 0$, $\phi^a_j(\xi)=\phi^a_0(2^{-a j}\xi)$ and
$$
\mathrm{supp}\,\phi^a_j\subset B^a_{2^{j+2}/3}\setminus B^a_{2^{j-1}},\quad\sum^n_{j=-1}\phi^a_j(\xi)=\phi^a_{-1}(2^{-(n+1){a}}\xi)\to 1,\quad n\to\infty.
$$
\bd
For given $j\geq -1$, the block operator $\cR^\dd_j$ is defined on $\sS'({ \mR^N})$ by
$$
\cR^\dd_jf(x):=(\phi^\dd_j\hat f)\check{\,\,}(x)=\check\phi^\dd_j* f(x),
$$
with the convention $\cR^\dd_j\equiv0$ for $j\leq-2$.
In particular, for $j\geq 0$,
\begin{align}\label{Def2}
\cR^\dd_jf(x)=2^{a\cdot m j}\int_{{ \mR^N}}\check\phi^a_0(2^{aj}y) f(x-y)\dif y,
\end{align}
where $a\cdot m=a_1m_1+\cdots+a_nm_n$.
\ed

For $j\geq -1$, by definition it is easy to see that
\begin{align}\label{KJ2}
\cR^a_j=\cR^a_j\widetilde\cR^a_j,\ \mbox{ where }\ \widetilde\cR^a_j:=\cR^a_{j-1}+\cR^a_{j}+\cR^a_{j+1},
\end{align}
and $\cR^a_j$ is symmetric in the sense that
$$
\<g, \cR^a_j f\>=\< f,\cR^a_jg\>,\ \ f,g\in\sS'(\mR^N),
$$
where $\<\cdot,\cdot\>$ stands for the dual pair between $\sS'(\mR^N)$ and $\sS(\mR^N)$.
Note that
\begin{align}\label{SX4}
\widetilde{\cR}^\dd_jf(x)=2^{a\cdot m j}\int_{{ \mR^N}}\check{\widetilde{\phi^a_0}}(2^{aj}y) f(x-y)\dif y,\ \ j\geq 1,
\end{align}
where
$$
\widetilde{\phi^a_0}(\xi):={2^{a\cdot m}}\phi_0(2^a\xi)+\phi_0(\xi)+{2^{-a\cdot m}}\phi_0(2^{-a}\xi).
$$
The cut-off low frequency operator $S_k$ is defined by
\begin{align}\label{EM9}
S_kf:=\sum_{j=-1}^{k-1}\cR^a_j f\to f,\ \ k\to\infty.
\end{align}
For $f,g\in\sS'(\R^N)$, define
$$
f\prec g:=\sum_{k\geq -1} S_{k-1}f\cR^a_k g,\ \ f\circ g:=\sum_{|i-j|\leq1}\cR^a_i f\cR^a_jg.
$$
The Bony decomposition of $fg$ is formally given by (cf. \cite{BCD11})
\begin{align}\label{Bony}
fg=f\prec g+ f\circ g+g\prec f.
\end{align}
The key point of Bony's decomposition is
\begin{align}\label{YQ1}
\cR^\dd_j (S_{k-1}f\cR^\dd_k g)=0 \ \mbox{ for }\ |k-j|>3.
\end{align}
Indeed, by Fourier's transform, we have
\begin{align*}
\big(\cR^\dd_j (S_{k-1}f\cR^\dd_k g)\big)\,{\hat{}}=\phi^\dd_j\cdot \sum_{i=-1}^{k-2}(\phi^\dd_i \hat f)*(\phi^\dd_k\hat g).
\end{align*}
Since the support of $\sum_{i=-1}^{k-2}(\phi^\dd_i  \hat f)*(\phi^\dd_k\hat g)$ is contained in $B^a_{2^{k+1}}\setminus B^a_{2^{k}/6}$, we have
$$
\phi^\dd_j\cdot \left(\sum_{i=-1}^{k-2}(\phi^\dd_i  \hat f)*(\phi^\dd_k\hat g)\right)=0,\ \ |k-j|>3,
$$
which in turn implies \eqref{YQ1}.

To introduce the weighted anisotropic Besov spaces, we
recall the following definition about the admissible weights from \cite{Tri06}.
\bd
 A $C^\infty$-smooth function $\rho:\mR^{N}\to(0,\infty)$ is called an admissible weight if
for each $j\in\mN$, there is a constant $C_j>0$ such that
\begin{align}\label{BE1}
|\nabla^j\rho(x)|\leq C_j\rho(x),\ \ \forall x\in\mR^N,
\end{align}
and for some $C, \kappa>0$,
\begin{align}\label{BETA}
\rho(x)\leq C\rho(y)(1+|x-y|^\kappa_a),\ \ \forall x,y\in\mR^N.
\end{align}
The set of all the admissible weights is denoted by $\sW$.
\ed

For $\rho\in\sW$ and $p\in[1,\infty]$, we define
$$
\|f\|_{L^p(\rho)}:=\|\rho f\|_p:=\left(\int_{\mR^N}|\rho(x)f(x)|^p\dif x\right)^{1/p}.
$$
Let $\rho_1,\rho_2,\rho_3$ be three weight functions. Suppose that for some $C_1>0$,
$$
\rho_1(x)\leq C_1\rho_2(y)\rho_3(x-y),\ \ \forall x,y\in\mR^N.
$$
By the classical Young's inequality, we have the following weighted version
\begin{align}\label{WYI}
\|f*g\|_{L^q(\rho_1)}\le C_1C_2\|f\|_{L^r(\rho_2)}\|g\|_{L^p(\rho_3)},
\end{align}
where $r,p,q\in[1,\infty]$ satisfy $1/q+1=1/p+1/r$ and $C_2=C_2(r,p,q)>0$.

Now we introduce the following weighted anisotropic Besov spaces (see \cite{Di96}).
\bd\label{Def25}
Let $\rho\in\sW$, $p,q\in[1,\infty]$ and $s\in\mR$. The weighted anisotropic Besov space $\bB^{s,a}_{p,q}(\rho)$ is defined by
$$
\bB^{s,a}_{p,q}(\rho):=\left\{f\in\sS'(\mR^N): \|f\|_{\bB^{s,a}_{p,q}(\rho)}:=
\left(\sum_{j\ge -1} 2^{sjq}\|\cR^a_j f\|_{L^p(\rho)}^q\right)^{1/q}<\infty\right\}.
$$
For simplicity of notation, we write
$$
\bC^s_a(\rho):=\bB^{s,a}_{\infty,\infty}(\rho),\ \bC^s_a:=\bC^s_a(1),\ \bB^{s,a}_{p,q}:=\bB^{s,a}_{p,q}(1),
$$
and when $a=(1,1,...,1)$ we shall drop the index $a$ in above notations.
\ed

The following inequality of Bernstein's type is quite useful.
\bl\label{Bern}
Let $\rho\in\sW$ be an admissible weight.
\begin{enumerate}[(i)]
\item For any $k\in\mN_0$, $1\le p\le q\le \infty$ and $i=1,2,...,n$, there is a constant $C=C(\rho,m,p,q,a,k,i)>0$ such that for all $j\geq -1$,
\begin{align}\label{Ber}
\|\nabla_{x_i}^k\cR_j^af\|_{L^q(\rho)}\lesssim_C2^{j(a_ik+a\cdot m(\frac{1}{p}-\frac{1}{q}))}\|\cR_j^af\|_{L^p(\rho)},
\end{align}
where $\nabla_{x_i}^k$ denotes the $k$-order gradient with respect to $x_i$, and
\begin{align}\label{Crapp}
\|\cR_j^af\|_{L^p(\rho)}\lesssim_C\|f\|_{L^p(\rho)}.
\end{align}
\item For any $s\in\mR$ and $p\in[1,\infty]$, there is a constant $C=C(\rho,m,p,a)>0$ such that for all $j\geq-1$,
\begin{align}\label{Ber0}
\|J_s\cR_j^af\|_{L^p(\rho)}\asymp_C 2^{sj}\|\cR_j^af\|_{L^p(\rho)},
\end{align}
where $\widehat{J_s f}(\xi):=\Big(\sum_{i=1}^n(1+|\xi_i|^2)^{1/(2a_i)}\Big)^s\hat f(\xi)$.
\end{enumerate}
\el
\begin{proof} We only prove \eqref{Ber} and \eqref{Ber0} for $j\geq 1$. For $j=-1,0$,
they follow directly from definition and $\phi_{-1}^a, \phi_{0}^a\in\sS(\mR^N)$.

(i) By \eqref{KJ2}, \eqref{BETA} and \eqref{WYI}, we have
\begin{align*}
\|\nabla_{x_i}^k\cR_j^af\|_{L^q(\rho)}&=\|\nabla_{x_i}^k\widetilde\cR_j^a\cR_j^af\|_{L^q(\rho)}
=\|(\nabla_{x_i}^k\check{\widetilde{\phi_{j}^a}})*\cR_j^af\|_{L^q(\rho)}\\
&\lesssim \|(1+|\cdot|_a^\kappa)\nabla_{x_i}^k\check{\widetilde{\phi_{j}^a}}\|_{L^r}\|\cR_j^af\|_{L^p(\rho)},
\end{align*}
where $1/p+1/r=1+1/q$, $\kappa$ is from \eqref{BETA} and
$$
\widetilde{\phi_{j}^a}:=\phi_{j-1}^a+\phi_{j}^a+\phi_{j+1}^a.
$$
Since $\kappa\ge0$, by \eqref{SX4} we have
\begin{align*}
\|(1+|\cdot|_a^\kappa)\nabla_{x_i}^k\check{\widetilde{\phi_{j}^a}}\|_{L^r}&\le 2^{a_ikj}
2^{(a\cdot m)j(1-\frac{1}{r})}\left(\int_{\mR^{N}}|\nabla_{x_i}^k
\check{\widetilde{\phi_0^a}}(x)|^r(1+|2^{-aj}x|_a^\kappa)^r\dif x\right)^{1/r}\\
&\leq 2^{j(a_ik+a\cdot m(\frac{1}{p}-\frac{1}{q}))}\left(\int_{\mR^{N}}
|\nabla_{x_i}^k\check{\widetilde{\phi_0^a}}(x)|^r(1+|x|_a^\kappa)^r\dif x\right)^{1/r}.
\end{align*}
Thus we get \eqref{Ber}. For \eqref{Crapp}, it is similar.

(ii) By \eqref{KJ2}, \eqref{BETA} and \eqref{WYI}, we similarly have
\begin{align*}
\|J_s\cR_j^af\|_{L^p(\rho)}\lesssim \|(1+|\cdot|_a^\kappa)J_s\check{\widetilde{\phi^a_j}}\|_{L^1}\|\cR_j^af\|_{L^p(\rho)}.
\end{align*}
Note that by definition and the change of variable,
\begin{align*}
(J_s\check{\widetilde{\phi^a_j}})^{\hat{}}(\xi)=\left(\sum_{i=1}^n(1+|\xi_i|^2)^{1/(2a_i)}\right)^s\widetilde{\phi^a_j}(\xi)=2^{sj}F_{s,j}(2^{-aj}\xi),
\end{align*}
where
$$
F_{s,j}(\xi):=\left(\sum_{i=1}^n(2^{-2a_i j}+|\xi_i|^2)^{1/(2a_i)}\right)^s\widetilde{\phi^a_0}(\xi).
$$
Since ${\rm supp}(\phi^a_0)\subset B^a_{2}\setminus B^a_{1/2}$, we have for any $k\in\mN_0$ and $i=1,\cdots, n$,
$$
\sup_{j\geq 1}\int_{\mR^N}|\nabla^k_{\xi_i}F_{s,j}(\xi)|\dif\xi<\infty,
$$
which in turn implies that
$$
\sup_{j\geq 1}\sup_{x\in\mR^N}(1+|x_i|^k)|\check F_{s,j}(x)|<\infty.
$$
Hence,
\begin{align*}
\|J_s\check{\widetilde{\phi^a_j}}(1+|\cdot|_a^\kappa)\|_{L^1}
&=2^{sj}2^{a\cdot m j}\int_{\mR^{N}}|\check F_{s,j}(2^{aj} x)|(1+|x|_a^\kappa)\dif x\\
&=2^{sj}\int_{\mR^{N}}|\check F_{s,j}(x)|(1+|2^{-aj}x|_a^\kappa)\dif x\\
&\leq 2^{sj}\int_{\mR^{N}}|\check F_{s,j}(x)|(1+|x|_a^\kappa)\dif x\lesssim 2^{sj}.
\end{align*}
Thus, for $j\geq 1$,
$$
\|J_s\cR_j^af\|_{L^p(\rho)}\lesssim_C 2^{sj}\|\cR_j^af\|_{L^p(\rho)}.
$$
Since $J_sJ_{-s}={\rm Id}$, we also have another side inequality.
\end{proof}
\br\rm
By definition and \eqref{Ber0}, one sees that for any $p,q\in[1,\infty]$ and $s,s'\in\mR$,
$J_s$ is an isomorphism between $\bB^{s'+s,a}_{p,q}$ and $\bB^{s',a}_{p,q}$, i.e.,
\begin{align}\label{IS}
J_s\bB^{s'+s,a}_{p,q}=\bB^{s',a}_{p,q}.
\end{align}
\er

As an easy consequence of Bernstein's inequality, we have the following embedding theorem of weighted anisotropic Besov spaces.
\bt\label{Embedding} Let $\rho\in\sW$, $s_1, s_2\in\mR$, $1\le r\le p\le \infty$ be such that
\begin{align*}
s_2=s_1+(a\cdot m)(\tfrac{1}{r}-\tfrac{1}{p}).
\end{align*}
For any $q\in[1,\infty]$, there is a constant $C=C(\rho,m,a,p,q,r,s_1,s_2)>0$ such that
\begin{align}\label{Em}
\|f\|_{\bB^{s_1,a}_{p,q}(\rho)}\le C\|f\|_{\bB^{s_2,a}_{r,q}(\rho)}.
\end{align}
Moreover, for any $1\le q_1\le q_2\le\infty$ and $\rho_2\leq\rho_1$,
\begin{align}\label{Embq}
\|f\|_{\bB^{s,a}_{p,q_2}(\rho_2)}\le\|f\|_{\bB^{s,a}_{p,q_1}(\rho_1)},
\end{align}
and for $\theta\in[0,1]$ and $p_1,p_2\in[1,\infty]$, $\rho_1,\rho_2\in\sW$, $s,s_1,s_2\in\mR$ with
$$
\tfrac\theta{p_1}+\tfrac{1-\theta}{p_2}=\tfrac1p,\ \theta s_1+(1-\theta)s_2=s,
$$
the following interpolation inequality holds,
\begin{align}\label{Embq0}
\|f\|_{\bB^{s,a}_{p,q}(\rho^\theta_1\rho^{1-\theta}_2)}\le\|f\|^\theta_{\bB^{s_1,a}_{p_1,q}(\rho_1)}\|f\|^{1-\theta}_{\bB^{s_2,a}_{p_2,q}(\rho_2)}.
\end{align}
\et
\begin{proof}
\eqref{Em} is straightforward by Lemma \ref{Bern} with $k=0$. \eqref{Embq} and \eqref{Embq0} are direct consequences of the definition
and H\"{o}lder's inequality.
\end{proof}

Now we give a characterization of $\bB^{s,a}_{p,q}(\rho)$. To this end, we introduce the following notations. For $f:\mR^N\to\mR$ and $h\in\mR^N$, the first order difference operator is defined by
$$
\delta_hf(x):=f(x+h)-f(x),
$$
and for $M\in\mN$, the $M$-order difference operator is defined recursively by
$$
\delta^{(M)}_hf(x):= \delta_h\delta^{(M-1)}_hf(x).
$$
By induction, it is easy to see that
\begin{align}\label{Def8}
\delta^{(M)}_hf(x)=\sum^{M}_{k=0}(-1)^{M-k}{M\choose k} f(x+kh),\quad  h\in\mathbb{R}^{N},
\end{align}
where $\binom M k$ is the binomial coefficient.
The following characterization about $\bB^{s,a}_{p,q}(\rho)$ is probably well-known to experts. Since we cannot find
them in the literature, for the readers' convenience, we provide detailed proofs in Appendix \ref{AnWB}.

\bt\label{Th26}
Let $\rho\in\sW$. For any $s\in(0,\infty)$ and $p,q\in[1,\infty]$,
there exists a constant $C=C(\rho, a,m,p,q,s)\geq1$ such that for all $f\in\bB^{s,a}_{p,q}(\rho)$,
\begin{align}\label{FG1}
\|f\|_{\bB^{s,a}_{p,q}(\rho)}\asymp_C \|f\|_{\widetilde\bB^{s,a}_{p,q}(\rho)}\asymp_C \|\rho f\|_{\widetilde\bB^{s,a}_{p,q}},
\end{align}
where
$$
\|f\|_{\widetilde\bB^{s,a}_{p,q}(\rho)}:=\left(\int_{|h|_a\leq 1}
\left(\frac{\big\|\delta^{([s]+1)}_{h}f\big\|_{L^p(\rho)}}{|h|_a^{s}}\right)^q\frac{\dif h}{|h|_a^{a\cdot m}}\right)^{1/q}
+\|f\|_{L^p(\rho)},
$$
where $[s]$ denotes the integer part of $s$.
Moreover, for any $s\in\mR$ and $p,q\in[1,\infty]$,
\begin{align}\label{EqT}
\|f\|_{\bB^{s,a}_{p,q}(\rho)}\asymp_C\|\rho f\|_{\bB^{s,a}_{p,q}}.
\end{align}
\et
\br\rm
For $\rho\equiv1$, the characterization of \eqref{FG1} is proven in \cite[Lemma 2.8]{ZZ21}. In particular, for $s>0$,
since $\bC^{s}_{a}(\rho)=\bB^{s,a}_{\infty,\infty}(\rho)$, by \eqref{FG1}
we have
\begin{align}\label{28}
\|f \|_{\bC^{s}_{a}(\rho)}\asymp_C\|f \|_{\bC^{s/a_1}_{x_1}(\rho)}+\cdots+\|f \|_{\bC^{s/a_n}_{x_n}(\rho)},
\end{align}
where for $i=1,\cdots,n$,
$$
\|f \|_{\bC^s_{x_i}(\rho)}:=\|f\|_{L^\infty(\rho)}+\frac{\sup_{|h_i|\leq 1}\big\|\delta^{([s]+1)}_{h_i}f\big\|_{L^\infty(\rho)}}{|h_i|^{s/a_i}},
$$
and
$$
\delta_{h_i}f(x):=f(\cdots,x_{i-1}, x_i+h_i, x_{i+1},\cdots)-f(\cdots,x_{i-1}, x_i, x_{i+1},\cdots).
$$

\er
As a corollary, we have the following result.
\bc
Let $\rho\in\sW$. For any $\alpha\in\mR$, $s>0$ and $p,q\in[1,\infty]$, there is a constant $C=C(\rho,a,m,p,q,\alpha,s)>0$
such that for all $f\in\bB^{\alpha+s,a}_{p,q}(\rho)$,
\begin{align}\label{FG1c}
\|\delta^{([s]+1)}_h f\|_{\bB^{\alpha,a}_{p,q}(\rho)}\lesssim_C|h|_a^s(1+|h|_a^\kappa)\|f\|_{\bB^{\alpha+s,a}_{p,q}(\rho)},
\end{align}
where $\kappa\geq 0$ is from \eqref{BETA}.
\ec
\begin{proof}
By \eqref{FG1}, for $|h|_a\le1$, we have
$$
\|\delta^{([s]+1)}_h f\|_{L^p(\rho)}\lesssim|h|_a^s\|f\|_{\bB^{s,a}_{p,\infty}(\rho)}.
$$
For $|h|_a\geq 1$, by \eqref{Def8} and \eqref{BETA} we have
$$
\|\delta^{([s]+1)}_h f\|_{L^p(\rho)}\lesssim(1+|h|_a^\kappa)\|f\|_{L^p(\rho)}
\lesssim(1+|h|_a^\kappa)\|f\|_{\bB^{s,a}_{p,\infty}(\rho)}.
$$
Therefore,
\begin{align}\label{Cor28}
\|\delta^{([s]+1)}_h f\|_{L^p(\rho)}\lesssim|h|_a^s(1+|h|_a^\kappa)\|f\|_{\bB^{s,a}_{p,\infty}(\rho)}.
\end{align}
Noting that
\begin{align*}
\|\cR^a_j f\|_{\bB^{s,a}_{p,\infty}(\rho)}=\sup_{k\geq -1}2^{ks}\|\cR^a_k\cR^a_j f\|_{L^p(\rho)}\lesssim 2^{js}\|\cR^a_j f\|_{L^p(\rho)},
\end{align*}
by \eqref{Cor28}, we have
\begin{align*}
\|\delta^{([s]+1)}_h f\|_{\bB^{\alpha,a}_{p,q}(\rho)}^q&=\sum_{j\geq -1}2^{\alpha jq}
\|\delta^{([s]+1)}_h\cR^a_j f\|_{L^p(\rho)}^q\\
&\lesssim |h|_a^{qs}(1+|h|_a^\kappa)^q
\sum_{j\geq -1}2^{\alpha jq}\|\cR^a_jf\|_{\bB^{s,a}_{p,\infty}(\rho)}^q\\
&\lesssim |h|_a^{qs}(1+|h|_a^\kappa)^q \sum_{j\geq -1}2^{(\alpha+s) jq}\|\cR^a_jf\|_{L^p(\rho)}^q\\
&=|h|_a^{qs}(1+|h|_a^\kappa)^q \|f\|_{\bB^{\alpha+s,a}_{p,q}(\rho)}^q.
\end{align*}
The proof is complete.
\end{proof}

{
We finish this subsection with an interpolation lemma for later use.

\bl\label{IIL}
Let $\{T_j\}_{j=-1}^\infty$ be a family of linear operators from $\sS'(\mR^{N})$ to some Banach space $\mX$.
Assume that for some $\beta_0<\beta_1$ and any $j\geq -1$,
there are constants $C_{ij}>0, i=0,1$ such that
\begin{align*}
\|T_j f\|_{\mX}\le C_{ij}2^{-j\beta_i}\|f\|_{\bC^{\beta_i}_a},\quad i=0,1.
\end{align*}
Then for any $\beta\in(\beta_0,\beta_1)$, there is a constant $C=C(a,m,\beta,\beta_0,\beta_1)>0$ such that
\begin{align*}
\|T_j f\|_{\mX}\le C(C_{0j}+C_{1j})2^{-j\beta}\|f\|_{\bC^{\beta}_a},\ \ j\geq -1.
\end{align*}
\el
\begin{proof}
Since for any $k\geq -1$,
\begin{align*}
\|\cR_k^af\|_{\bC^{\beta_i}_a}\lesssim 2^{(\beta_i-\beta)k} \|f\|_{\bC^{\beta}_a},\ \ i=0,1,
\end{align*}
we have by \eqref{EM9} and the assumptions,
\begin{align*}
\|T_j f\|_{\mX}&\le\sum_{k\ge-1}\|T_j\cR_k^af\|_{\mX}
\le C_{0j}2^{-j\beta_0}\sum_{k>j}\|\cR_k^af\|_{\bC^{\beta_0}_a}+C_{1j}2^{-j\beta_1}\sum_{k\leq j}\|\cR_k^af\|_{\bC^{\beta_1}_a}\\
&\lesssim \left(C_{0j}2^{-j\beta_0}\sum_{k>j}2^{(\beta_0-\beta)k}+C_{1j}2^{-j\beta_1}\sum_{k\leq j}2^{(\beta_1-\beta)k}\right)\|f\|_{\bC^{\beta}_a}\\
&\lesssim (C_{0j}+C_{1j})2^{-j\beta}\|f\|_{\bC^{\beta}_a}.
\end{align*}
The proof is complete.
\end{proof}
}


\subsection{Paraproduct calculus}
\label{ssec:para}

In this subsection we recall some basic ingredients in the paracontrolled calculus developed by Bony \cite{Bon81} and \cite{GIP15}.
The first important fact is that the product $fg$ of two distributions $f\in \bC_a^\alpha$ and $g\in \bC_a^\beta$ is well defined
if and only if $\alpha+\beta>0$ as given in the following lemma.

\begin{lemma}\label{lem:para}
	Let $\rho_{1},\rho_{2}\in\sW$. We have for any  $\beta\in\R$,
	\begin{equation}\label{GZ0}
	\|f\prec g\|_{\bC_a^\beta(\rho_{1}\rho_{2})}\lesssim_C\|f\|_{L^\infty(\rho_{1})}\|g\|_{\bC_a^{\beta}(\rho_{2})},
	\end{equation}
	and for any $\alpha<0$ and $\beta\in\mR$,
	\begin{equation}\label{GZ1}
	\|f\prec g\|_{\bC_a^{\alpha+\beta}(\rho_{1}\rho_{2})}\lesssim_C\|f\|_{\bC_a^{\alpha}(\rho_{1})}\|g\|_{\bC_a^{\beta}(\rho_{2})}.
	\end{equation}
	Moreover, for any $\alpha,\beta\in\mR$ with  $\alpha+\beta>0$,
	\begin{equation}\label{GZ2}
	\|f\circ g\|_{\bC_a^{\alpha+\beta}(\rho_{1}\rho_{2})}\lesssim_C\|f\|_{\bC_a^{\alpha}(\rho_{1})}\|g\|_{\bC_a^{\beta}(\rho_{2})}.
	\end{equation}
	In particular, for any $\alpha,\beta\in\mR$ with  $\alpha+\beta>0$,
	\begin{equation}\label{GZ3}
	\|f\cdot g\|_{\bC_a^{\alpha\wedge\beta}(\rho_{1}\rho_{2})}\lesssim_C\|f\|_{\bC_a^{\alpha}(\rho_{1})}\|g\|_{\bC_a^{\beta}(\rho_{2})}.
	\end{equation}
\end{lemma}
\begin{proof}
Totally the same as \cite[Lemma 2.1]{GIP15} and \cite[Lemma 2.14]{GH18}.
\end{proof}

For two abstract operators $\sA_1,\sA_2$ acting on functions,
we shall use the following notation to denote the commutator between $\sA_1$ and $\sA_2$:
$$
[\sA_1,\sA_2]f:=\sA_1(\sA_2 f)-\sA_2(\sA_1f).
$$

We have the following simple commutator estimate (see \cite[Lemma 2.2]{GIP15}).
\bl\label{lem:com1}
For any $\rho_1, \rho_2\in\sW$, $\alpha\in(0,1)$ and $\beta,\gamma\in\mR$, there is a constant $C=C(\rho_1,\rho_2,\alpha,\beta,\gamma,a,m)>0$
such that for all $j\geq -1$,
\begin{align}\label{AV2}
\|[\cR^a_j, f]g\|_{L^\infty(\rho_1\rho_2)}
\lesssim_C 2^{-\alpha j}\|f\|_{\bC_a^\alpha(\rho_1)}\|g\|_{L^\infty(\rho_2)}.
\end{align}
\el

The following lemma is a weighted anisotropic version of Lemma 2.4 in \cite{GIP15}.

\begin{lemma}\label{lem:com2}
	Let $\rho_{1}, \rho_{2}, \rho_{3}\in\sW$. For any $\alpha\in (0,1)$ and $\beta,\gamma\in \R$ with
	$\alpha+\beta+\gamma>0$ and $\beta+\gamma<0$, there exists a trilinear bounded operator $\mathrm{com}$
	on $\bC_a^\alpha(\rho_{1})\times \bC_a^\beta(\rho_{2})\times \bC_a^\gamma(\rho_{3})$ such that
	\begin{align}\label{FA1}
	\|\mathrm{com}(f,g,h)\|_{\bC_a^{\alpha+\beta+\gamma}(\rho_{1}\rho_{2}\rho_{3})}\lesssim_C
	\|f\|_{\bC_a^\alpha(\rho_{1})}\|g\|_{\bC_a^\beta(\rho_{2})}\|h\|_{\bC_a^\gamma(\rho_{3})},
	\end{align}
	where
	$$
	\mathrm{com}(f,g,h)=(f\prec g)\circ h - f(g\circ h).
	$$
In addition, if $\beta<0$ and {	$\alpha+\beta>0$,}  then $[h\circ, f] g$ can be extended to be a bounded linear
operator on $\bC_a^\alpha(\rho_{1})\times \bC_a^\beta(\rho_{2})\times \bC_a^\gamma(\rho_{3})$ with
\begin{align}\label{Le12}
\|[h\circ, f] g\|_{\bC^{\alpha+\beta+\gamma}_a(\rho_1\rho_2\rho_3)}
\lesssim_C \|f\|_{\bC^{\alpha}_a(\rho_1)}\|g\|_{\bC^{\beta}_a(\rho_2)}\|h\|_{\bC^{\gamma}_a(\rho_3)}.
\end{align}
\end{lemma}
\begin{proof}
	By Lemmas \ref{lem:para} and \ref{lem:com1}, estimate of \eqref{FA1} is completely the same as in \cite{GIP15}.
For \eqref{Le12}, note that
\begin{align*}
[h\circ, f] g&=h\circ(gf)-f(h\circ g)
=h\circ(f\succcurlyeq g)+{\rm com}(f,g,h).
\end{align*}
By Lemma \ref{lem:para}, we have
\begin{align*}
\|h\circ(f\succcurlyeq g)\|_{\bC^{\alpha+\beta+\gamma}_a(\rho_1\rho_2\rho_3)}
&\lesssim \|h\|_{\bC^{\gamma}_a(\rho_3)}\|f\succcurlyeq g\|_{\bC^{\alpha+\beta}_a(\rho_1\rho_2)}\\
&\lesssim\|h\|_{\bC^{\gamma}_a(\rho_3)} \|f\|_{\bC^{\alpha}_a(\rho_1)}\|g\|_{\bC^{\beta}_a(\rho_2)},
\end{align*}
which together with \eqref{FA1} yields \eqref{Le12}.
\end{proof}

\section{Kinetic semigroups and commutator estimates}\label{kinetic}
In this section we introduce basic estimates about the kinetic semigroup. Compared to the heat semigroup,
due to the presence of the transport term, the kinetic semigroup does not commutate with block operator $\cR_j^a$ (see \eqref{DE29} below),
which brings some new features.
In Subsection \ref{sec:3.2}, we introduce a kinetic H\"older space which admits
velocity transport in time direction,
as well as a localization characterization for weighted H\"older space proved in \cite{ZZZ20}, which will be used to obtain the well-posedness of linear equation \eqref{eq:li}.
In Subsection \ref{sec:3.3}, we establish the Schauder estimate in kinetic H\"older spaces.
In Subsection \ref{sec:commutator}, we prove a commutator estimate for the
kinetic semigroup which is essential to apply the paracontrolled calculus for the kinetic equations. Finally, in Subsection \ref{sec:2.5} we introduce the renormalized pairs used in the definition of paracontrolled solutions.

In the remainder of this paper, we  consider the following case of the weighted anisotropic Besov spaces:
\begin{align*}
N=2d,~d\in\mN, \ \ n=2,\ \ m_1=m_2=d,\ \ a=(3,1).
\end{align*}
For $t>0$, let $P_t$ be the kinetic semigroup defined by
\begin{align}\label{Sem}
P_t f(z):= \Gamma_tp_t*\Gamma_tf(z)=\Gamma_t(p_t*f)(z),\quad z=(x,v)\in\mR^{2d},
\end{align}
where for $t\in \mR$,
$$
\Gamma_tf(z):=f(\Gamma_t z),\ \ \Gamma_tz:=(x+tv,v),
$$
and
\begin{align}\label{PPT}
p_t(z)=p_t(x, v)=\(\frac{4\pi t^4}{3}\)^{-d/2}\exp\(-\frac{3|x|^2+|3x-2tv|^2}{4t^3}\)
\end{align}
is the density of the following process
\begin{align*}
Z_t:=(X_t,V_t)=\left(\sqrt2\int_0^tB_s\dif s,\sqrt2B_t\right),
\end{align*}
where $B_t$ is a $d$-dimensional standard Brownian motion.
The reason of choosing multi-scale parameter $a=(3,1)$ is the following scaling property (see also \eqref{BA1}):
$$
(X_{\lambda t}, V_{\lambda t})\stackrel{(d)}{=}(\lambda^{\frac32}X_{t}, \lambda^{\frac12}V_{ t}),\ \ \lambda>0.
$$
Note that
\begin{align}\label{DZ9}
\Gamma_t\Gamma_s=\Gamma_{t+s}, \ \ p_t(z)=t^{-a\cdot m/2}p_1(t^{-a/2}z),
\end{align}
and for $\varphi\in C^\infty_b(\mR^{2d})$,
$$
\partial_t P_t\varphi=(\Delta_v+v\cdot \nabla_x )P_t\varphi.
$$

{\bf Notation:} Let $\sP_{\rm w}$ be the set of all polynomial weights with the form:
\begin{align}\label{Rho}
\rho(z)=\varrho(z)^\kappa,\ \ \kappa\in\mR,
\end{align}
where for $z=(x,v)$,
\begin{align}\label{ND2}
\varrho(x,v):=((1+|x|^2)^{1/3}+1+|v|^2)^{-1/2}\asymp(1+|z|_a)^{-1}.
\end{align}
Clearly, for some $C_0=C_0(\kappa,d)>0$,
\begin{align}\label{AS00}
\rho(z)\leq {C_0}\rho(\bar z)(1+|z-\bar z|_a^{|\kappa|}),
\end{align}
and for any $j\in\mN$ and some $C_j=C_j(\kappa,d)>0$,
\begin{align}\label{AS000}
\ |\nabla^j_v\rho(z)|\leq { C_j}\rho(z)\varrho^j(z),\ \ \ |\nabla^j_x\rho(z)|\leq { C_j}\rho(z)\varrho^{2j}(z),
\end{align}
and for any $T>0$, there is a constant $C_T=C(T,\kappa,d)>0$ such that
\begin{align}\label{AS01}
C_T^{-1}\rho(z)\le\Gamma_t\rho(z)\le C_T\rho(z),\ \ z\in\mR^{2d},\ t\in[0,T].
\end{align}
Moreover, for $\rho_1,\rho_2\in\sP_{\rm w}$, we have
$$
\rho_1/\rho_2,\ \rho_1\rho_2,\ \rho_1\vee\rho_2,\ \rho_1\wedge\rho_2\in\sP_{\rm w}.
$$
\subsection{Kinetic semigroup estimates}

In this subsection, we recall the estimate about the heat kernel of Kolmogorov operator
$\Delta_v+v\cdot\nabla_x$ under the action of block operator $\cR^a_j$ and a crucial decomposition \eqref{DE29}
from \cite{HWZ20}. Then we establish the basic properties of the kinetic semigroups in Lemma \ref{Le222}.
First of all, we recall the following two lemmas proven in \cite{HWZ20}.
\bl
 For any $\alpha,\beta,\gamma\ge0$ and $T>0$, there is a $C=C(T,d,\alpha,\beta,\gamma)>0$ such that for all $j\geq-1$ and $t\in(0,T]$,
\begin{align}\label{Cru}
\int_{\mR^{2d}}|x|^\beta|v|^\gamma|\cR^a_j\Gamma_tp_t(x, v)|\dif x\dif v\lesssim_C2^{-(3\beta+\gamma)j}  (t^{1/2}2^{j})^{-\alpha}.
\end{align}
In particular,  for any $\rho\in\sP_{\rm w}$, $T>0$ and $\alpha\geq 0$,
there is a constant $C=C(T,d,\alpha,\rho)>0$ such that for all $j\geq-1$ and $t\in(0,T]$,
\begin{align}\label{Cru2}
\|\cR_j^a\Gamma_tp_t\|_{L^1(\rho)}\lesssim_C (t^{1/2}2^{j})^{-\alpha}\wedge 1.
\end{align}
\el
\begin{proof}
When $j\geq 0$, by  \cite[Lemma 5.1 (5.9)]{HWZ20}, we have for any $n\in\mN_0$,
\begin{align*}
\sJ_j(t)&:=\int_{\mR^{2d}}|x|^\beta|v|^\gamma|\cR^a_j\Gamma_tp_t(x, v)|\dif x\dif v\\
&\lesssim\Big(\hbar^{3n}+\hbar^n\Big)\Big(2^{-(3\beta+\gamma)j}+t^{\frac{3\beta+\gamma}{2}}\Big)\\
&=2^{-(3\beta+\gamma)j}\Big(\hbar^{3n}+\hbar^n\Big)\Big(1+\hbar^{-(3\beta+\gamma)}\Big),
\end{align*}
where $\hbar:=t^{-\frac12} 2^{-j}$. Since $n\in\mN_0$ is arbitrary, we clearly have for any $\alpha\geq 0$,
\begin{align*}
\sJ_j(t)\lesssim_C2^{-(3\beta+\gamma)j}\hbar^\alpha=2^{-(3\beta+\gamma)j}  (t^{1/2}2^{j})^{-\alpha}.
\end{align*}
When $j=-1$, we have
\begin{align*}
\sJ_{-1}(t)\le C\le C T^{\alpha/2}t^{-\alpha/2}.
\end{align*}
Thus we get \eqref{Cru}.
Estimate \eqref{Cru2} follows directly by \eqref{Cru}.
\end{proof}

We recall the following important observation 
from \cite[Lemma 6.7]{HWZ20}.
 \bl\label{HWZ67}
 For $t\geq 0$ and $j\in\mN_0$, define
\begin{align*}
\Theta^{t}_j:=\Big\{\ell\geq -1: 2^{\ell}\leq 2^{4} (2^j+t2^{3j}),\ 2^{j}\leq 2^{4} (2^\ell+t2^{3\ell})\Big\}.
\end{align*}
\begin{enumerate}[(i)]
\item For any $\ell\notin\Theta^{t}_j$, it holds that
\begin{align}\label{EM3}
\cR^\dd_j\Gamma_{t}\cR^\dd_\ell=0.
\end{align}
\item For any $0\not=\beta\in\mR$, there is a constant $C=C(\beta)>0$ such that
\begin{align}\label{DA2}
\sum_{\ell\in\Theta^t_j}2^{\beta\ell}\lesssim_C 2^{j\beta}\big(1+t2^{2j}\big)^{|\beta|},  \  \ j\in\mN_0,\ t\geq 0.
\end{align}
\end{enumerate}
\el
\br\rm
By \eqref{Sem} and \eqref{EM3}, we have the following decomposition of the kinetic semigroup:
\begin{align}\label{DE29}
\cR_j^aP_tf= \sum_{\ell\in\Theta_j^t}\cR_j^a\Gamma_tp_t*\Gamma_t\cR_\ell^af= \sum_{\ell\in\Theta_j^t}\cR_j^a P_t\cR_\ell^af,\ j\in\mN_0.
\end{align}
\er
By \eqref{DE29}, we can show the following basic estimates for the kinetic semigroup $P_t$.
\bl\label{Le222}
\begin{enumerate}[(i)]
\item For any $\rho\in\sP_{\rm w}$, $\alpha\geq 0$, $\beta\in\mR$ and $T>0$, there is a constant $C=C(\rho,T,d,\alpha,\beta)>0$ such that
for all $j\geq-1$, $t\in(0,T]$ and $f\in\bC^\beta_a(\rho)$,
\begin{align}\label{DV1}
\|\cR_j^aP_tf\|_{L^\infty(\rho)}\lesssim_C 2^{-j\beta}(1\wedge (t^{\frac12}2^j)^{-\alpha})\|f\|_{\bC_a^\beta(\rho)}.
\end{align}
In particular, for any $\alpha\geq0$,
\begin{align}\label{IEL28}
\|P_tf\|_{\bC^{\alpha+\beta}_a(\rho)}\lesssim_C t^{-\alpha/2}\|f\|_{\bC^\beta_a(\rho)}.
\end{align}
\item For any $\rho\in\sP_{\rm w}$, $k\in\mN_0$, $\beta<k$ and $T>0$, there is a constant $C=C(T,k,\rho,\beta)>0$
such that for all $t\in(0,T]$ and $f\in\bC^{\beta}_a(\rho)$,
\begin{align}\label{IEL282}
\|\nabla_v^kP_tf\|_{L^\infty(\rho)}\lesssim_C t^{(\beta-k)/2}\|f\|_{\bC^{\beta}_a(\rho)}.
\end{align}
\item For any $\rho\in\sP_{\rm w}$, $T>0$ and $\beta\in(0,2)$, there is a constant $C=C(\rho,d,\beta,T)>0$ such that
for all $t\in[0,T]$ and $f\in\bC^{\beta}_a(\rho)$,
\begin{align}\label{LE283}
\|P_tf-\Gamma_tf\|_{L^\infty(\rho)}\lesssim_C t^{\beta/2}\|f\|_{\bC^\beta_a(\rho)}.
\end{align}
\end{enumerate}
\el
\begin{proof}
(i) By the interpolation lemma \ref{IIL}, we only show \eqref{DV1} for $\beta\ne0$.
Let $\rho$ be as in \eqref{Rho}. For $j\in\mN_0$,  by \eqref{DE29}, \eqref{AS00} and \eqref{WYI}, we have
\begin{align}
\|\cR_j^aP_tf\|_{L^\infty(\rho)}&\leq \sum_{\ell\in\Theta_j^t}\|\cR_j^a\Gamma_tp_t*\Gamma_t\cR_\ell^af\|_{L^\infty(\rho)}\no\\
&\lesssim \|(1+|\cdot|^{|\kappa|}_a)\cR_j^a\Gamma_tp_t\|_{L^1}\sum_{\ell\in\Theta_j^t}\|\Gamma_t\cR_\ell^af\|_{L^\infty(\rho)}.\label{AA6}
\end{align}
Moreover,  by \eqref{AS01}, we have
\begin{align*}
\sum_{\ell\in\Theta_j^t}\|\Gamma_t\cR_\ell^af\|_{L^\infty(\rho)}
&\lesssim \sum_{\ell\in\Theta_j^t}\|\Gamma_t(\rho\cR_\ell^af)\|_{L^\infty}=\sum_{\ell\in\Theta_j^t}\|\rho\cR_\ell^af\|_{L^\infty}\\
&\leq \sum_{\ell\in\Theta_j^t}2^{-{\ell}\beta}\|f\|_{\bC_a^\beta(\rho)}\stackrel{\eqref{DA2}}{\lesssim} 2^{-\beta j}(1+(t4^j)^{|\beta|})\|f\|_{\bC_a^\beta(\rho)}.
\end{align*}
Therefore, by  \eqref{Cru2} and \eqref{AA6}, for any $l\geq 0$,
\begin{align*}
\|\cR_j^aP_tf\|_{L^\infty(\rho)}&\lesssim((t^{1/2}2^j)^{-2l}\wedge 1)2^{-\beta j}(1+(t4^j)^{|\beta|})\|f\|_{\bC_a^\beta(\rho)},
\end{align*}
 which implies \eqref{DV1} for $j\in\mN_0$
 by taking $l=\frac\alpha2+|\beta|$ and $l={\frac\alpha2}$, respectively. For $j=-1$, it is obvious. Moreover, \eqref{IEL28} follows directly by \eqref{DV1}.

(ii) For \eqref{IEL282}, by Lemma \ref{Bern}, we have
 \begin{align*}
\|\nabla_v^kP_tf\|_{L^\infty(\rho)}&\le \sum_{j=-1}^\infty 2^{kj}\|\cR_j^aP_tf\|_{L^\infty(\rho)}
\le \sum_{j=-1}^\infty2^{(k-\beta)j}(1\wedge(t^{\frac12}2^j)^{-{2k}})\|f\|_{\bC_a^\beta(\rho)}\\
&\lesssim \|f\|_{\bC_a^\beta(\rho)}\int^\infty_{-\infty}2^{(k-\beta) s}(1\wedge(t^{\frac12}2^s)^{-{2k}})\dif s\\
&=\|f\|_{\bC_a^\beta(\rho)}t^{(\beta-k)/2}\int^\infty_0s^{k-\beta}(1\wedge s^{-{2k}})\frac{\ln 2\dif s}{s},
\end{align*}
which gives \eqref{IEL282}.

(iii) Note that
\begin{align*}
P_tf-\Gamma_tf\stackrel{\eqref{Sem}}{=}\Gamma_t(p_t*f-f),
\end{align*}
and by $p_t(z)=p_t(-z)$,
\begin{align*}
p_t*f(z)-f(z)=\frac12\int_{\mR^{2d}}p_t(\bar z)(\delta_{\bar z}f(z)+\delta_{-\bar z}f(z))\dif\bar z.
\end{align*}
By  \eqref{AS01}, \eqref{Cor28} and \eqref{AS00}, we have
\begin{align*}
\|P_tf-\Gamma_tf\|_{L^\infty(\rho)}
&\lesssim \|p_t*f-f\|_{L^\infty(\rho)}
\leq\frac12\int_{\mR^{2d}}p_t(\bar z)\|\delta_{\bar z}f+\delta_{-\bar z}f\|_{L^\infty(\rho)}\dif\bar z\\
&\lesssim\left(\int_{\mR^{2d}}p_t(\bar z)|\bar z|_a^\beta(1+|\bar z|_a^{\kappa})\dif\bar z\right)\|f\|_{\bC^\beta_a(\rho)},
\end{align*}
where $\kappa>0$ is from \eqref{AS00} and  we have used  that for $\beta\in[1,2)$,
$$
\delta_{\bar z}f+\delta_{-\bar z}f=\delta^{(2)}_{\bar z}f(\cdot-\bar z).
$$
Thus we obtain \eqref{LE283} by \eqref{DZ9}.
\end{proof}

\subsection{Kinetic H\"older spaces and characterization}\label{sec:3.2}
For $T>0$, $\alpha\in\mR$ and  $\rho\in\sP_{\rm w}$, let $\mC_{T,a}^{\alpha}(\rho)$ be the space of all space-time distributions with finite norm
$$
\|f\|_{\mC_{T,a}^{\alpha}(\rho)}:=\sup_{0\leq t\leq T} \| f(t)\|_{\bC^{\alpha}_a(\rho)}<\infty.
$$
We introduce the following weighted kinetic H\"older space.
\bd[Kinetic H\"older space]\label{kinetics}
Let $\rho\in\sP_{\rm w}$, $\alpha\in(0,2)$ and $T>0$. Define
\begin{align}\label{SS0}
\mS_{T,a}^{\alpha}(\rho):=\Big\{f: \|f\|_{\mS_{T,a}^{\alpha}(\rho)}:=\|f\|_{\mC_{T,a}^\alpha(\rho)}+\|f\|_{\bC_{T;\Gamma}^{\alpha/2}L^\infty(\rho)}<\infty\Big\},
\end{align}
where for $\beta\in(0,1)$,
$$
\|f\|_{\bC_{T;\Gamma}^{\beta}L^\infty(\rho)}:=\sup_{0\leq t\leq T} \|f(t)\|_{L^\infty(\rho)}
+\sup_{s\not=t\in[0,T]} \frac{\|f(t)-\Gamma_{t-s}f(s)\|_{L^\infty(\rho)}}{|t-s|^{\beta}}.
$$
For $\rho=1$, we simply write
$$
\mS_{T,a}^{\alpha}:=\mS_{T,a}^{\alpha}(1),\ \ \bC_{T;\Gamma}^{\beta}L^\infty:=\bC_{T;\Gamma}^{\beta}L^\infty(1).
$$
\ed
\br\rm
(i) In the above definition, the appearance of $\Gamma_t$ reflects the
transport role of $v\cdot\nabla_x$ (see also \eqref{LE283} for the same reason).
It is noticed that this definition is essentially equivalent to the one introduced  in \cite{IS21} by using the language of group.

(ii){Lemma \ref{Le11} below stated the Schauder estimate on kinetic H\"{o}lder space. If $f$ is independent of time, we can check the Schauder estimate holds in the classical H\"{o}lder space.} 

\er
Next we show a localization characterization for $\mS_{T,a}^{\alpha}(\rho)$, which shall be used in Section \ref{ParaA} to deduce global estimate.
Let $\chi$ be a nonnegative smooth function with
\begin{align}\label{ChiCut}
\chi(z)=1,\ \ |z|_a\leq 1/8,\ \ \ \chi(z)=0,\  \ |z|_a>1/4,
\end{align}
and for $r>0$ and $z_0\in\mR^{2d}$,
\begin{align}\label{phicut}
\chi^{z_0}_r(z):=\chi\big(\tfrac{z-z_0}{r^a}\big),\ \ \phi^{z_0}_r(z):=\chi^{z_0}_{r(1+|z_0|_a)}(z),
\end{align}
where we have used the notation \eqref{ND0}.
The following characterization of weighted H\"older spaces is due to \cite[Lemma 3.8]{ZZZ20}.

\bl\label{cha}

Let $\alpha>0$ and $r\in(0,1]$. For any $\rho_1,\rho_2\in \sP_{\rm w}$,
there is a constant $C=C(r,\alpha,d,\rho_1,\rho_2)>0$ such that
\begin{align}\label{ND313}
\|\phi^z_r\|_{\bC^\alpha_a(\rho)}\lesssim_C \rho(z),\ \ z\in\mR^{2d},
\end{align}
and for any $j\in\mN$,
\begin{align}\label{ND4}
\|\nabla^j_v\phi^z_r\|_{\bC^\alpha_a(\rho)}+\|v\cdot\nabla_x\phi^z_r\|_{\bC^\alpha_a(\rho)}\lesssim_C (\varrho\rho)(z).
\end{align}
Moreover,
\begin{align}\label{GD2}
\|f\|_{\bC^\alpha_a(\rho_1\rho_2)}
\asymp_C\sup_{z_0\in\mR^{2d}}\left(\rho_1(z_0)\|\phi^{z_0}_r f\|_{\bC^\alpha_a(\rho_2)}\right)
\end{align}
and
{
\begin{align}\label{GD22}
\|f\|_{L^\infty(\rho_1\rho_2)}
\asymp_C\sup_{z_0\in\mR^{2d}}\left(\rho_1(z_0)\|\phi^{z_0}_r f\|_{L^\infty(\rho_2)}\right).
\end{align}
}
\el
\begin{proof}
{
Firstly, we show \eqref{ND313} and \eqref{ND4} is an easy consequence of \eqref{ND313}. In fact, by \eqref{FG1} we have
\begin{align*}
\|\phi^z_r\|_{\bC^\alpha_a(\rho)}\lesssim &(\|\nabla^{[\alpha]+1}\phi^z_r\|_{L^\infty(\rho)}+\|\phi^z_r\|_{L^\infty(\rho)}\)
\\\lesssim& \sup_{\bar{z}\in\mR^{2d}}\rho(\bar z) \chi\(\frac{\bar{z}-z}{[r(1+|z|_a)]^a}\)+\sup_{\bar{z}\in\mR^{2d}}\rho(\bar z) \chi^{([\alpha]+1)}\(\frac{\bar{z}-z}{[r(1+|z|_a)]^a}\)\lesssim \rho(z),
\end{align*}
where the last step is from $r<1$ and the same argument in \cite[Lemma 3.8]{ZZZ20}.
Then, we have
\begin{align*}
\|f\|_{\bC^\alpha_a(\rho_1\rho_2)}&\lesssim \sup_{z_0}\|\phi^{z_0}_rf\|_{\bC^\alpha_a(\rho_1\rho_2)}\lesssim \sup_{z_0}\|\phi^{z_0}_{2r}\|_{\bC^\alpha_a(\rho_1)}\|\phi^{z_0}_rf\|_{\bC^\alpha_a(\rho_2)}\\
&\lesssim \sup_{z_0}\rho_1(z_0)\|\phi^{z_0}_rf\|_{\bC^\alpha_a(\rho_2)},
\end{align*}
where the first inequality is from \eqref{FG1} and the second inequality is from $\phi^{z_0}_r=\phi^{z_0}_{2r}\phi^{z_0}_r$.

On the other hand,
\begin{align*}
\sup_{z_0}\rho_1(z_0)\|\phi^{z_0}_rf\|_{\bC^\alpha_a(\rho_2)}&\lesssim \sup_{z_0}\rho_1(z_0)\|\phi^{z_0}_{r}\|_{\bC^\alpha_a(\rho_1^{-1})}\|f\|_{\bC^\alpha_a(\rho_1\rho_2)}\lesssim \|f\|_{\bC^\alpha_a(\rho_1\rho_2)}.
\end{align*}
Thus \eqref{GD2} follows. \eqref{GD22} is totally the same.
}
\end{proof}
By definition \eqref{phicut}, the following lemma is elementary.
\bl\label{Le23}
For any $z_0\in\mR^{2d}$ and $|t|\leq r^3<1$, it holds that
\begin{align}\label{AAh}
\phi^{z_0}_{r}\Gamma_t\phi^{z_0}_{8r}=\phi^{z_0}_{r},
\end{align}
and for $j=0,1$, there is a constant $C=C(r,d)>0$ such that
\begin{align}\label{AAq}
\|\Gamma_t\nabla^j_v\phi^{z_0}_r-\nabla^j_v\phi^{z_0}_r\|_{L^\infty}\lesssim_C |t|/(1+|z_0|^{2+j}_a).
\end{align}
\el
\begin{proof}
For $|t|\leq r^3<1$, by Young's inequality, we have
$$
|tv|^{1/3}\leq r\,(\tfrac{2}{3}+\tfrac{|v|}{3}).
$$
Equality \eqref{AAh} follows by
$$
\mbox{\rm supp}(\phi^{z_0}_r)\subset B^a_{r(1+|z_0|_a)/4}(z_0)\subset\Gamma_tB^a_{r(1+|z_0|_a)}(z_0)
$$
and $\phi^{z_0}_{8r}\equiv1$ on $B^a_{r(1+|z_0|_a)}(z_0)$. For \eqref{AAq}, note that for $z=(x,v)$,
\begin{align*}
|\Gamma_t\nabla^j_v\phi^{z_0}_r(z)-\nabla^j_v\phi^{z_0}_r(z)|&\leq \sup_{s\in[0,t]}t|v|\,
|\nabla_x\nabla^j_v\phi^{z_0}_r(\Gamma_sz)|\cdot\1_{\{|\Gamma_sz-z_0|_a\leq r(1+|z_0|_a)/4\}}\\
&\leq t|z|_a\frac{\|\nabla_x\nabla_v^j\chi\|_{L^\infty}}{(r(1+|z_0|_a))^{3+j}}\1_{\{|z-z_0|_a\leq r(1+|z_0|_a)\}}\\
&\lesssim t/(1+|z_0|_a^{2+j}).
\end{align*}
The proof is complete.
\end{proof}

By Lemma \ref{cha}, we have the following characterization for $\mS_{T,a}^{\alpha}(\rho)$.
\bl\label{cha0}
For any $\alpha\in(0,2)$, $r\in(0,1/8)$, $\rho\in\sP_{\rm w}$ and $T>0$, there is a constant $C=C(T,r,\alpha,d,\rho)>0$ such that
\begin{align}\label{GD3}
\|f\|_{\mS_{T,a}^{\alpha}(\rho)}
\asymp_C\sup_{z_0}\left(\rho(z_0)\|\phi^{z_0}_r f\|_{\mS_{T,a}^{\alpha}}\right).
\end{align}
\el
\begin{proof}
By \eqref{GD2}, we only need to prove that for any $\alpha\in(0,1)$,
	\begin{align*}
\|f\|_{\bC_{T;\Gamma}^\alpha L^\infty(\rho)}\asymp\sup_z\left(\rho(z)\|\phi^z_r f\|_{\bC_{T;\Gamma}^\alpha L^\infty}\right).
\end{align*}
By definition and \eqref{AS01} \eqref{GD22}, it suffices to show
\begin{align}\label{ETK1}
\begin{split}
&\sup_{0\le t\le T}\|f(t)\|_{L^\infty(\rho)}+\sup_{0<|t-s|\leq r^3}\sup_z\frac{\rho(z)\|\phi^z_rf(t)-\phi^z_r\Gamma_{t-s}f(s)\|_{L^\infty}}{|t-s|^\alpha}\\
&\asymp\sup_{0\le t\le T}\|f(t)\|_{L^\infty(\rho)}+\sup_{0<|t-s|\leq r^3}\sup_z\frac{\rho(z)\|\phi^z_rf(t)-\Gamma_{t-s}(\phi^z_rf)(s)\|_{L^\infty}}{|t-s|^\alpha}.
\end{split}
\end{align}
Since $\Gamma_{t-s}(\phi^z_rf)(s)=\Gamma_{t-s}\phi^z_r\, \Gamma_{t-s}f(s)$,
by Lemma \ref{Le23} one sees that
\begin{align*}
&\sup_{0<|t-s|\leq r^3}\sup_z\frac{\rho(z)\|\phi^z_r\Gamma_{t-s}f(s)-\Gamma_{t-s}(\phi^z_rf)(s)\|_{L^\infty}}{|t-s|^\alpha}\\
&\quad=\sup_{0<|t-s|\leq r^3}\sup_z\frac{\rho(z)\|\phi^z_r\Gamma_{t-s}(\phi^z_{8r}f(s))-\Gamma_{t-s}(\phi^z_r\phi^z_{8r}f)(s)\|_{L^\infty}}{|t-s|^\alpha}\\
&\quad\lesssim\sup_{0<|t-s|\leq r^3}\sup_z\frac{\rho(z)\|\Gamma_{t-s}(\phi^z_{8r}f(s))\|_{L^\infty}\|\phi^z_r-\Gamma_{t-s}\phi^z_r\|_{L^\infty}}{|t-s|^\alpha}\\
&\quad\lesssim\sup_{0<|t-s|\leq r^3}\sup_z\frac{\rho(z)\|\phi^z_{8r}f(s)\|_{L^\infty}|t-s|}{|t-s|^\alpha}\lesssim\sup_{0\le s\le T}\|f(s)\|_{L^\infty(\rho)}.
\end{align*}
The proof is complete.
\end{proof}
Next we give a result regarding derivatives in kinetic H\"{o}lder spaces.
\bl\label{Lem35}
For any $\alpha\in(1,2)$, $T>0$ and $\rho\in\sP_{\rm w}$, there is a constant $C=C(\rho,T,\alpha,d)>0$ such that for all $f\in\mS_{T,a}^{\alpha}(\rho)$,
\begin{align}\label{Lem34}
\|\nabla_vf\|_{\mS_{T,a}^{\alpha-1}(\rho)}\lesssim_C\| f\|_{\mS_{T,a}^{\alpha}(\rho)}.
\end{align}
\el
\begin{proof}
First of all, we prove \eqref{Lem34} for $\rho=1$.
Fix $\alpha\in(1,2)$, $s,t\in[0,T]$ and $z\in\mR^{2d}$.
By definition, one sees that for all $\bar z=(0,\bar v)\in\mR^{2d}$,
\begin{align*}
\cI_1&:=|f(t,\bar z+z)-f(s,\Gamma_{t-s} z)-\bar v\cdot(\nabla_v f)(s,\Gamma_{t-s}z)|\\
&\leq |f(t,\bar z+z)-f(s,\Gamma_{t-s}(\bar z+z))|+|f(s,\Gamma_{t-s}(\bar z+z))-f(s,\bar z+\Gamma_{t-s}z)|\\
&\quad+|f(s,\bar z+\Gamma_{t-s}z)-f(s,\Gamma_{t-s}z)-\bar v\cdot(\nabla_v f)(s,\Gamma_{t-s}z)|\\
&\lesssim |t-s|^{\frac\alpha2}\|f\|_{\bC_{T;\Gamma}^{\alpha/2}L^\infty}+|(t-s)\bar v|^{\frac{\alpha}{3}}\|f(s)\|_{\bC^{\alpha/3}_x}
+|\bar v|^\alpha\|\nabla_v f(s)\|_{\bC^{\alpha-1}_v}\\
&\lesssim(|t-s|^{\frac\alpha2}+|\bar v|^\alpha)\|f\|_{\mS_{T,a}^{\alpha}},
\end{align*}
{where we used Young's inequality in the last step.}
By exchanging $s\leftrightarrow t$ and in place of $z$ by $\Gamma_{t-s}z$, we also have
$$
\cI_2:=|f(s,\bar z+\Gamma_{t-s}z)-f(t,z)-\bar v\cdot(\nabla_v f)(t,z)|\lesssim(|t-s|^{\frac\alpha2}+|\bar v|^\alpha)\|f\|_{\mS_{T,a}^{\alpha}}.
$$
Let $\omega$ be the unit vector in $\mR^d$ so that
\begin{align*}
\cI:=|(\nabla_v f)(t,z)-(\nabla_v f)(s,\Gamma_{t-s}z)|=\omega\cdot[(\nabla_v f)(t,z)-(\nabla_v f)(s,\Gamma_{t-s}z)].
\end{align*}
Let $\bar v=(t-s)^{\frac12}\omega$ and $\bar z=(0,\bar v)$. Then
\begin{align*}
(t-s)^{\frac12}\cI&\leq\cI_1+\cI_2+|f(t,\bar z+z)-f(s,\bar z+\Gamma_{t-s}z)|+|f(t,z)-f(s,\Gamma_{t-s} z)|\\
&\leq\cI_1+\cI_2+|(t-s)\bar v|^{\frac{\alpha}{3}}\|f(s,\cdot)\|_{\bC^\alpha_a}+2\|f(t)-\Gamma_{t-s}f(s)\|_{L^\infty}\\
&\lesssim (t-s)^{\frac\alpha2}\|f\|_{\mS_{T,a}^{\alpha}}.
\end{align*}
Hence,
\begin{align}\label{AA1}
\cI=|(\nabla_v f)(t,z)-(\nabla_v f)(s,\Gamma_{t-s}z)|\lesssim (t-s)^{\frac{\alpha-1}2}\|f\|_{\mS_{T,a}^{\alpha}}.
\end{align}
Moreover, by Bernstein's inequality in Lemma \ref{Bern}, it is clear that
$$
\|\nabla_v f\|_{\mC_{T,a}^{\alpha-1}}\lesssim\|f\|_{\mC_{T,a}^{\alpha}},
$$
which together with \eqref{AA1} implies \eqref{Lem34} for $\rho=1$.

Next, for $\beta\in(0,2)$, note that by \eqref{AAq} and the definition {of kinetic H\"older space},
$$
\|\nabla_v\phi^z_r g\|_{\bC_{T;\Gamma}^{\beta/2}L^\infty}\lesssim (1+|z|_a)^{-1}\|g\|_{\bC_{T;\Gamma}^{\beta/2}L^\infty}
$$
and
$$
\|\nabla_v\phi^z_r g\|_{\bC^{\beta}_a}\lesssim \|\nabla_v\phi^z_r\|_{\bC^{\beta}_a}\|g\|_{\bC^{\beta}_a}\lesssim
(1+|z|_a)^{-1}\|g\|_{\bC^{\beta}_a}.
$$
Hence,
\begin{align}\label{STr}
\|\nabla_v\phi^z_r g\|_{\mS^{\beta}_{T,a}}\lesssim (1+|z|_a)^{-1}\|g\|_{\mS^{\beta}_{T,a}}.
\end{align}
Now, for any $r\in(0,1/16)$, by Lemma \ref{cha0} {and \eqref{STr}} we have
\begin{align*}
\|\nabla_vf\|_{\mS_{T,a}^{\alpha-1}(\rho)}&\asymp\sup_{z}\rho(z)\|\phi^z_r\nabla_vf\|_{\mS_{T,a}^{\alpha-1}}\\
&\lesssim\sup_{z}\rho(z)\|\nabla_v(\phi^z_rf)\|_{\mS_{T,a}^{\alpha-1}} +\sup_{z}\rho(z)\|\nabla_v\phi^z_r f\|_{\mS_{T,a}^{\alpha-1}}\\
&\lesssim\sup_{z}\rho(z)\|\phi^z_rf\|_{\mS_{T,a}^{\alpha}} +\sup_{z}\rho(z)\|\nabla_v\phi^z_r (\phi^z_{2r}f)\|_{\mS_{T,a}^{\alpha-1}}\\
&\lesssim \| f\|_{\mS_{T,a}^{\alpha}(\rho)}+\sup_{z}\rho(z)\|\phi^z_{2r} f\|_{\mS_{T,a}^{\alpha-1}}\lesssim \| f\|_{\mS_{T,a}^{\alpha}(\rho)}.
\end{align*}
{Here in the second inequality we used \eqref{Lem34} for $\rho=1$ we proved above.} The proof is complete.
\end{proof}

\subsection{Schauder's estimates}\label{sec:3.3}

For given $\lambda\ge0$ and $f\in L^1_{\textrm{loc}} (\mR_+;\sS'(\mR^{2d}))$, we consider the following model kinetic equation:
\begin{align*}
\sL_\lambda u:=(\p_t-\Delta_v+\lambda{-}v\cdot\nabla_x)u=f,\quad u(0)=0.
\end{align*}
By Duhamel's formula, the unique solution of the above equation is given by
 \begin{align}\label{TY1}
u(t,\cdot)=\int_0^t\e^{-\lambda(t-s)}P_{t-s}f(s,\cdot)\dif s:=\sI_\lambda f(t,\cdot).
\end{align}
In other words, $\sI_\lambda$ is the inverse of $\sL_\lambda$. For $q\in[1,\infty]$, $T>0$ and a Banach space $\mB$, we write
$$
 L^q_T(\mB):=L^q([0,T];\mB).
$$

Now we can show the following Schauder estimate.
\bl\label{Le11}
(Schauder estimates)
Let $\rho\in\sP_{\rm w}$, $\beta\in(0,2)$ and $\theta\in(\beta,2]$. For any $q\in[\frac{2}{2-\theta},\infty]$ and $T>0$, there is a constant $C=C(d,\beta,\theta,q,T)>0$
such that for all $\lambda\ge0$ and $f\in L^q_T\bC^{-\beta}_a(\rho)$,
\begin{align}\label{EG1}
\|\sI_\lambda f\|_{\mS^{\theta-\beta}_{T,a}(\rho)}\lesssim_C (\lambda\vee1)^{\frac{\theta}{2}+\frac{1}{q}-1}\|f\|_{ L^q_T\bC^{-\beta}_a(\rho)}.
\end{align}
\el
\begin{proof}
Let $q\in[\frac{2}{2-\theta},\infty]$ and $\frac1p+\frac1q+\frac\theta2=1$. By \eqref{DV1} and H\"{o}lder's inequality, we have for $\beta\in\mR$,
\begin{align*}
\|\cR^a_j\sI_\lambda f(t)\|_{L^\infty(\rho)}&\lesssim \int_0^t\e^{-\lambda(t-s)}2^{j\beta}(1\wedge((t-s)4^j)^{-2})\|f(s)\|_{\bC_a^{-\beta}(\rho)}\dif s\\
&\lesssim 2^{j\beta}\(\int_0^t\e^{-\lambda ps}\dif s\)^{\frac{1}{p}}\(\int^t_0(1\wedge(s4^j)^{-\frac4\theta})\dif s\)^{\frac\theta 2}
\|f\|_{ L^q_T\bC^{-\beta}_a(\rho)}\\
&\lesssim 2^{j(\beta-\theta)}(\lambda\vee1)^{-\frac{1}{p}}\(\int^\infty_0(1\wedge s^{-\frac4\theta})\dif s\)^{\frac\theta 2}
\|f\|_{ L^q_T\bC^{-\beta}_a(\rho)}.
\end{align*}
This implies that for $\beta\in\mR$,
 \begin{align}\label{SchS}
\|\sI_\lambda f\|_{\mC^{\theta-\beta}_{T,a}(\rho)}\lesssim (\lambda\vee1)^{\frac{\theta}{2}+\frac{1}{q}-1}\|f\|_{ L^q_T\bC^{-\beta}_a(\rho)}.
\end{align}
 On the other hand, let $u:=\sI_\lambda f$. For any $0\le t_1<t_2\le T$, we have
 \begin{align*}
 u(t_2)-\Gamma_{t_2-t_1}u(t_1)&=\int_0^{t_1}\(\e^{-\lambda(t_2-s)}-\e^{-\lambda(t_1-s)}\)P_{t_2-s}f(s)\dif s\\
 &+\(P_{t_2-t_1}-\Gamma_{t_2-t_1}\)\sI_\lambda f(t_1)+\int_{t_1}^{t_2}\e^{-\lambda(t_2-s)}P_{t_2-s}f(s)\dif s\\
 &:=I_1+I_2+I_3.
\end{align*}
Let
$$
q':=q/(q-1).
$$
For $I_1$, by \eqref{IEL282} and H\"older's inequality, we have for $\beta>0$,
\begin{align*}
\|I_1\|_{L^\infty(\rho)}&\le |\e^{-\lambda(t_2-t_1)}-1|\int_0^{t_1}\e^{-\lambda(t_1-s)}\|P_{t_2-s}f(s)\|_{L^\infty(\rho)}\dif s\\
&\lesssim [\lambda(t_2-t_1)\wedge1]\int_0^{t_1}\e^{-\lambda(t_1-s)}(t_2-s)^{-\frac\beta2}\|f(s)\|_{\bC^{-\beta}_a(\rho)}\dif s\\
&\le [\lambda(t_2-t_1)]^{\frac{\theta}{2}}(t_2-t_1)^{-\frac{\beta}{2}}\(\int_0^{t_1}\e^{-\lambda q's}\dif s\)^{\frac{1}{q'}}\|f\|_{ L^q_T\bC^{-\beta}_a(\rho)}\\
&\lesssim (t_2-t_1)^{\frac{\theta-\beta}{2}}(\lambda\vee1)^{\frac{\theta}{2}+\frac{1}{q}-1}\|f\|_{ L^q_T\bC^{-\beta}_a(\rho)}.
\end{align*}
For $I_2$, by \eqref{LE283} and \eqref{SchS}, we have for $\beta\in(\theta-2,\theta)$,
\begin{align*}
\|I_2\|_{L^\infty(\rho)}&\le (t_2-t_1)^{\frac{\theta-\beta}{2}}\|\sI_\lambda f\|_{\mC^{\theta-\beta}_{T,a}(\rho)}\\
&\lesssim (t_2-t_1)^{\frac{\theta-\beta}{2}}(\lambda\vee1)^{\frac{\theta}{2}+\frac{1}{q}-1}\|f\|_{ L^q_T\bC^{-\beta}_a(\rho)}.
\end{align*}
For $I_3$, by \eqref{IEL282}, H\"older's inequality { and the change of variable}, we have for $\beta\in(0,\theta)$,
{
\begin{align*}
\|I_3\|_{L^\infty(\rho)}&\lesssim \int_{t_1}^{t_2}\e^{-\lambda(t_2-s)}(t_2-s)^{-\frac{\beta}{2}}\|f(s)\|_{\bC^{-\beta}_a(\rho)}\dif s\\
&\lesssim  \(\int_{0}^{t_2-t_1}\e^{-q'\lambda s}s^{-\frac{q'\beta}{2}}\dif s\)^{\frac{1}{q'}}\|f\|_{ L^q_T\bC^{-\beta}_a(\rho)}\\
&\lesssim  \((t_2-t_1)^{\frac{1}{q'}-\frac{\beta}{2}}\wedge\lambda^{\frac{\beta}{2}-\frac{1}{q'}}\)\|f\|_{ L^q_T\bC^{-\beta}_a(\rho)}\\
&\lesssim (t_2-t_1)^{\frac{\theta-\beta}{2}}\lambda^{\frac{\theta}{2}+\frac{1}{q}-1}\|f\|_{ L^q_T\bC^{-\beta}_a(\rho)},
\end{align*}
}
where in the third inequality we have used interpolation inequality { $a^{\gamma}\wedge b^{-\gamma}\le a^{\delta}b^{\delta-\gamma}$ for all $a,b>0$ and $0<\delta\le\gamma$ for $\gamma:=\frac{1}{q'}-\frac{\beta}{2}\ge\frac{\theta-\beta}{2}=:\delta>0$}.
Combining the above calculations, we obtain
$$
\|\sI_\lambda f\|_{C_\Gamma^{(\theta-\beta)/2}L^\infty(\rho)}\lesssim (\lambda\vee1)^{\frac{\theta}{2}+\frac{1}{q}-1}\|f\|_{ L^q_T\bC^{-\beta}_a(\rho)}.
$$
The proof is complete.
\end{proof}

\subsection{Commutator estimates}\label{sec:commutator}
In this subsection we prove important commutator estimates about the kinetic semigroups, {which is essential for applying paracontrolled calculus to the kinetic equations.
Compared with the case of the classical heat semigroups (see \cite{GIP15}),
the kinetic semigroup is not a Fourier multiplier and there is a $\Gamma_t$ in the commutator as stated in the left hand side of \eqref{GA3} below, which leads to a commutator for $\sI_\lambda$ in the kinetic H\"{o}lder space (see Lemma \ref{commutator1} below). In the estimate of the main term $I_j^{(0)}$ in the proof, we find that the commutator for the operator $\Gamma_tp_t*$ gains regularity from $f$, while the commutator for $\Gamma_t$ cannot have this property.
In particular, the decomposition \eqref{DE29} plays a crucial role in the following proof.}
\begin{lemma}\label{commutator}
	Let $\rho_{1},\rho_{2}\in\sP_{\rm w}$. For any $\alpha\in (0,1)$, $\beta\in \R$, $\delta\geq 0$
	and $T>0$, there is a constant $C=C(\rho_1,\rho_2,\alpha,\beta,\delta,T,d)>0$
	such that for all
	$f\in \bC_a^{\alpha}(\rho_{1})$, $g\in \bC_a^{\beta}(\rho_{2})$ and $t\in(0,T]$, $j\geq-1$,
	\begin{align}
	\begin{split}
&	\|\cR^a_j P_t(f\prec g)-\cR^a_j (\Gamma_tf\prec P_tg)\|_{L^\infty(\rho_{1}\rho_{2})}\\
&\qquad\qquad\lesssim_C t^{-\frac\delta2} 2^{-(\alpha+\beta+\delta)j}\|f\|_{\bC_a^\alpha(\rho_1)}\|g\|_{\bC_a^\beta(\rho_2)}.\label{GA3}
\end{split}
	\end{align}
\end{lemma}
\begin{proof}
Without loss of generality, we  only prove \eqref{GA3} for $j\geq 3$ and $\beta\not=0,-\alpha$.
For $\beta=0$ or $-\alpha$, it follows by the interpolation Lemma \ref{IIL}.
First of all, by \eqref{KJ2}, \eqref{DE29} and the definition of $\prec$, we have
$$
\cR^a_jP_t(f\prec g)=\sum_{\ell\sim j}\sum_{i\in \Theta^{t}_\ell}\sum_{k\geq-1}\cR^a_j\cR^a_\ell P_t\cR^a_i(S_{k-1} f\cR^a_k g),
$$
where
$$
\ell\sim j\iff|\ell-j|\leq 3.
$$
Noting that by \eqref{YQ1},
$$
\cR^a_i(S_{k-1} f\cR^a_k g)=0\ \mbox{ for $i\in\Theta^t_\ell$ and $\ k\notin\Theta^{t}_\ell\pm 3$},
$$
where
$$
\Theta^{t}_\ell\pm 3:=\{k\geq 0: |k-i|\leq 3,\ i\in\Theta^t_\ell\},
$$
we further have
\begin{align*}
\cR^a_jP_t(f\prec g)&=\sum_{\ell\sim j}\sum_{i\in \Theta^{t}_\ell}\sum_{k\in\Theta^{t}_\ell\pm 3}\cR^a_j\cR^a_\ell P_t\cR^a_i(S_{k-1} f\cR^a_k g)\\
&\!\!\!\stackrel{\eqref{DE29}}{=}\sum_{\ell\sim j}\sum_{k\in\Theta^{t}_\ell\pm 3}\cR^a_j\cR^a_\ell P_t(S_{k-1} f\cR^a_k g).
\end{align*}
Similarly, by \eqref{YQ1} we also have
\begin{align*}
\cR^a_j(\Gamma_tf\prec P_tg)
=\sum_{\ell\sim j}\cR^a_j(S_{\ell-1}\Gamma_tf\cdot\cR^a_\ell P_tg)=:I^{(1)}_j+I^{(2)}_j,
\end{align*}
where
\begin{align*}
I^{(1)}_j&:=\sum_{\ell\sim j}\cR^a_j(\Gamma_t S_{\ell-1}f\cdot \cR^a_\ell P_t g),\\
I^{(2)}_j&:=\sum_{\ell\sim j}\cR^a_j([S_{\ell-1},\Gamma_t]f\cdot\cR^a_\ell P_tg).
\end{align*}
For $I^{(1)}_j$, by \eqref{DE29} again, we can write
\begin{align*}
I^{(1)}_j&=\sum_{\ell\sim j}\sum_{k\in\Theta^{t}_\ell\pm 3}\cR^a_j(\Gamma_tS_{\ell-1} f\cdot \cR^a_\ell P_t \cR^a_k g)\\
&=\sum_{\ell\sim j}\sum_{k\in\Theta^{t}_\ell\pm 3}\cR^a_j(\Gamma_t(S_{\ell-1}-S_{k-1}) f\cdot \cR^a_\ell P_t \cR^a_k g)\\
&\quad+\sum_{\ell\sim j}\sum_{k\in\Theta^{t}_\ell\pm 3}\cR^a_j(\Gamma_tS_{k-1} f\cdot \cR^a_\ell P_t \cR^a_k g)=:I^{(11)}_j+I^{(12)}_j.
\end{align*}
Combining the above calculations, we obtain
\begin{align*}
\cR^a_jP_t(f\prec g)-\cR^a_j(\Gamma_tf\prec P_tg)=I^{(0)}_j-I^{(11)}_j-I^{(2)}_j,
\end{align*}
where
\begin{align*}
I^{(0)}_j:=\sum_{\ell\sim j}\sum_{k\in\Theta^{t}_\ell\pm 3}\cR^a_j\Big(\cR^a_\ell P_t(S_{k-1} f\cR^a_k g)-\Gamma_tS_{k-1} f\cdot \cR^a_\ell P_t \cR^a_k g\Big).
\end{align*}
For $I^{(0)}_j$, let $F_k:=S_{k-1} f$ and $G_k:=\cR^a_{k} g$. Note that
\begin{align*}
J_{k\ell}&:=\cR^a_\ell P_t(S_{k-1} f\cR^a_k g)-\Gamma_tS_{k-1} f\cdot \cR^a_\ell P_t \cR^a_k g\\
&\stackrel{\eqref{Sem}}{=}(\cR^a_\ell\Gamma_tp_t)*\Gamma_t(F_kG_k)-\Gamma_tF_k(\cR^a_\ell\Gamma_tp_t*\Gamma_tG_k)\\
&=\int_{\mR^{2d}}\cR^a_\ell\Gamma_tp_t(\bar z)(\Gamma_t F_k(z-\bar z)-\Gamma_tF_k(z))\Gamma_tG_k(z-\bar z)\dif \bar z.
\end{align*}
By \eqref{Cor28}, \eqref{AS01} and \eqref{Cru}, there exists $\delta_0>0$ such that for any $m\geq 0$,
\begin{align*}
\|J_{k\ell}\|_{L^\infty(\rho_1\rho_2)}
&\lesssim\left(\int_{\mR^{2d}}|\cR^a_\ell\Gamma_tp_t(\bar z)||\Gamma_t\bar z|^\alpha_a (1+|\bar z|^{\delta_0}_a)\dif \bar z\right)
\|F_k\|_{\bC^\alpha_a(\rho_1)}\|G_k\|_{L^\infty(\rho_2)}\\
&\lesssim2^{-\alpha\ell }(1\wedge (t4^{\ell})^{-m})\|f\|_{\bC^\alpha_a(\rho_1)}2^{-\beta k}\|g\|_{\bC_a^{\beta}(\rho_{2})}.
\end{align*}
Hence, by \eqref{DA2}, for $\beta\not=0$,
\begin{align*}
\|I^{(0)}_j\|_{L^\infty(\rho_1\rho_2)}&\lesssim \sum_{\ell\sim j}\sum_{k\in\Theta^{t}_\ell\pm 3}
2^{-\alpha\ell -\beta k}(1\wedge (t4^\ell)^{-|\beta|-\frac\delta2})
\|f\|_{\bC^\alpha_a(\rho_1)}\|g\|_{\bC_a^{\beta}(\rho_{2})}\\
&\lesssim2^{-j(\alpha +\beta)}(1+(t4^j))^{|\beta|}(1\wedge (t4^j)^{-|\beta|-\frac\delta2})\|f\|_{\bC^\alpha_a(\rho_1)}\|g\|_{\bC_a^{\beta}(\rho_{2})}\\
&\lesssim2^{-j(\alpha +\beta)}(t4^j)^{-\frac\delta2}\|f\|_{\bC^\alpha_a(\rho_1)}\|g\|_{\bC_a^{\beta}(\rho_{2})}.
\end{align*}
For $I^{(11)}_j$, by \eqref{DV1} we have for any $m\geq 0$,
\begin{align*}
\|I^{(11)}_j\|_{L^\infty(\rho_1\rho_2)}&\lesssim\sum_{\ell\sim j}\sum_{k\in\Theta^{t}_\ell\pm 3}\|\Gamma_t(S_{\ell-1}-S_{k-1}) f\cdot \cR^a_\ell P_t \cR^a_k g\|_{L^\infty(\rho_1\rho_2)}\\
&\lesssim\sum_{\ell\sim j}\sum_{k\in\Theta^{t}_\ell\pm 3}\|(S_{\ell-1}-S_{k-1}) f\|_{L^\infty(\rho_1)}\|\cR^a_\ell P_t \cR^a_k g\|_{L^\infty(\rho_2)}\\
&\lesssim\sum_{\ell\sim j}\sum_{k\in\Theta^{t}_\ell\pm 3}\sum_{i=k\wedge\ell}^{k\vee\ell}\|\cR^a_i f\|_{L^\infty(\rho_1)}
(1\wedge(t4^\ell)^{-m})\|\cR^a_kg\|_{\bC_a^0(\rho_2)}\\
&\lesssim\sum_{\ell\sim j}\sum_{k\in\Theta^{t}_\ell\pm 3}\left(\sum_{i=k\wedge\ell{-1}}^{k\vee\ell{-2}}2^{-i\alpha}\right)(1\wedge(t4^j)^{-m})
\|f\|_{\bC_a^{\alpha}(\rho_{1})}\|\cR^a_kg\|_{L^\infty(\rho_2)}\\
&\lesssim\sum_{\ell\sim j}\sum_{k\in\Theta^{t}_\ell\pm 3}2^{-(k\wedge\ell)\alpha}2^{-k\beta}(1\wedge(t4^j)^{-m})
\|f\|_{\bC_a^{\alpha}(\rho_{1})}\|g\|_{\bC^\beta_a(\rho_2)}\\
&\lesssim2^{-(\alpha+\beta)j}(1+t4^j)^{\alpha+|\beta|}(1\wedge(t4^j)^{-m})
\|f\|_{\bC_a^{\alpha}(\rho_{1})}\|g\|_{\bC^\beta_a(\rho_2)},
\end{align*}
where in the last step we have used $2^{-(k\wedge\ell)\alpha}\leq 2^{-k\alpha}+2^{-\ell\alpha}$,
\eqref{DA2} and  $\beta\not=0,-\alpha$.
Taking $m=\alpha+|\beta|+\delta/2$, we get
\begin{align*}
\|I^{(11)}_j\|_{L^\infty(\rho_1\rho_2)}&\lesssim2^{-(\alpha+\beta)j}(t4^j)^{-\delta/2}
\|f\|_{\bC_a^{\alpha}(\rho_{1})}\|g\|_{\bC^\beta_a(\rho_2)}.
\end{align*}
For $I^{(2)}_j$, noting that
$$
[S_{\ell-1},\Gamma_t]f=\sum_{i=-1}^{\ell-2}[\cR^a_i,\Gamma_t]f=-\sum_{i=\ell-1}^\infty[\cR^a_i,\Gamma_t]f,
$$
and
$$
[\cR^a_i,\Gamma_t]f(z)=\int_{\mR^{2d}}\check{\phi}^a_i(\bar z)(f(\Gamma_t(z-\bar z))-f(\Gamma_t z-\bar z))\dif\bar z,
$$
by \eqref{AS01}, \eqref{AS00}, \eqref{Cor28} and the definition of $\cR^a_i$, there is a $\delta_0>0$ such that for all $i$,
{\begin{align*}
\|[\cR^a_i,\Gamma_t]f\|_{L^\infty(\rho_1)}&\lesssim\int_{\mR^{2d}}|\check{\phi}^a_i(\bar z)
|(1+|\bar{z}|_a)^{\delta_0}\Big|(f(\Gamma_t(z-\bar z))-f(\Gamma_t z-\bar z))\rho_1(\Gamma_t(z-\bar z))\Big|\dif\bar z
\\&\lesssim\int_{\mR^{2d}}|\check{\phi}^a_i(\bar z)
|(t|\bar v|)^{\frac\alpha3}(1+|\bar{z}|_a)^{\delta_0}(1+t|\bar v|)^{\delta_0}\|f\|_{\bC^{\alpha/3}_x(\rho_1)}\dif\bar z\\
&\lesssim(t2^{-i})^{\frac\alpha3}\|f\|_{\bC^{\alpha}_a(\rho_1)},
\end{align*}}
and
\begin{align*}
\|[S_{\ell-1},\Gamma_t]f\|_{L^\infty(\rho_1)}
\lesssim\sum_{i=\ell-1}^\infty(t2^{-i})^{\frac\alpha3}\|f\|_{\bC^{\alpha}_a(\rho_1)}\lesssim t^{\frac\alpha3}2^{-\frac{\alpha\ell}3 }\|f\|_{\bC^\alpha_a(\rho_1)}.
\end{align*}
Hence, by \eqref{DV1},
\begin{align*}
\|I^{(2)}_j\|_{L^\infty(\rho_1\rho_2)}&\lesssim\sum_{\ell\sim j}\|[S_{\ell-1},\Gamma_t]f\|_{L^\infty(\rho_1)}\|\cR^a_\ell P_tg\|_{L^\infty(\rho_2)}\\
&\lesssim\sum_{\ell\sim j}t^{\frac\alpha3}2^{-\frac{\alpha\ell}3 }\|f\|_{\bC^\alpha_a(\rho_1)}2^{-\beta \ell}
(t4^\ell)^{-\frac\alpha3-\frac\delta2}\|g\|_{\bC^\beta_a(\rho_2)}\\
&\lesssim t^{-\delta/2}2^{-(\alpha+\beta+\delta)j}\|f\|_{\bC^\alpha_a(\rho_1)}\|g\|_{\bC^\beta_a(\rho_2)}.
\end{align*}
The proof is complete.
\end{proof}

Using this lemma we can show the following crucial commutator estimate.

\begin{lemma}\label{commutator1}
	Let $\rho_{1},\rho_{2}\in\sP_{\rm w}$ and $\alpha\in (0,1)$, $\beta\in \R$. For $k=0,1$, $T>0$ and $\theta\in[0,2]$,
	there is a constant $C>0$ such that for all $\lambda\geq 0$,
	\begin{align}
	\|[ \nabla^k_v\sI_\lambda, f \prec ]g\|_{\mC_{T,a}^{ \alpha+\beta+\theta-k}(\rho_1\rho_2)}
	&\lesssim_C(\lambda\vee 1)^{\frac{\theta-2}{2}}\|f\|_{\mS_{T,a}^{\alpha}(\rho_1)}\|g\|_{\mC_{T,a}^\beta(\rho_2)}.\label{GA44}
	\end{align}
\end{lemma}
\begin{proof}
	For $k=0$, by definition \eqref{TY1} of $\sI_\lambda$, we can write
	\begin{align*}
	[\sI_\lambda, f \prec ]g(t)&=\int^t_0\e^{-\lambda(t-s)}\Big(P_{t-s}(f(s)\prec g(s))-f(t)\prec P_{t-s} g(s)\Big)\dif s
	\\&=\int_0^t\e^{-\lambda(t-s)}\Big(P_{t-s}(f(s)\prec g(s))-\Gamma_{t-s}f(s)\prec P_{t-s}g(s)\Big) \dif s
	\\&\quad+\int_0^t\e^{-\lambda(t-s)}(\Gamma_{t-s}f(s)-f(t))\prec P_{t-s} g(s)\dif s
	\\&=:I_1(t)+I_2(t).
	\end{align*}
For $I_1(t)$, by Lemma \ref{commutator} we have
	\begin{align*}
	\|\cR^a_jI_1(t)\|_{L^\infty(\rho_1\rho_2)}
	&\lesssim 2^{-(\alpha+\beta)j}\int_0^t\e^{-\lambda s} ((4^js)^{-2}\wedge 1)\dif s\|f\|_{\mC_{T,a}^\alpha(\rho_1)}\|g\|_{\mC_{T,a}^\beta(\rho_2)}.
	\end{align*}
Note that by H\"older's inequality,
\begin{align}
\int_0^t\e^{-\lambda s} ((4^js)^{-2}\wedge 1)\dif s&\leq\left(\int_0^t\e^{-\frac{2\lambda s}{2-\theta}}\dif s\right)^{\frac{2-\theta}{2}}
\left(\int^t_0((4^js)^{-2}\wedge 1)^{\frac2\theta}\dif s\right)^{\frac\theta2}\no\\
&\lesssim (\lambda\vee1)^{\frac{\theta-2}{2}}2^{-\theta j}.\label{SZ1}
\end{align}
Thus,
$$
\|\cR^a_jI_1(t)\|_{L^\infty(\rho_1\rho_2)}\lesssim (\lambda\vee1)^{\frac{\theta-2}{2}}
2^{-(\alpha+\beta+\theta) j}\|f\|_{\mC_{T,a}^\alpha(\rho_1)}\|g\|_{\mC_{T,a}^\beta(\rho_2)}.
$$
	For $I_2(t)$, for any $\gamma>0$, note that by \eqref{GZ0},
	\begin{align*}
	&\|\cR^a_j((\Gamma_{t-s}f(s)-f(t))\prec P_{t-s} g(s))\|_{L^\infty(\rho_1\rho_2)}\\
	&\quad\lesssim 2^{-(\gamma+\beta)j}\|\Gamma_{t-s}f(s)-f(t)\|_{L^\infty(\rho_1)}\|P_{t-s} g(s)\|_{\bC^{\gamma+\beta}_a(\rho_2)}\\
	&\quad\lesssim 2^{-(\gamma+\beta)j}(t-s)^{\frac{\alpha-\gamma}2}\|f\|_{\mS_{T,a}^{\alpha}(\rho_1)}\|g(s)\|_{\bC^{\beta}_a(\rho_2)},
	\end{align*}
which implies by \eqref{SZ1} again,
	\begin{align*}
       \|\cR^a_jI_2(t)\|_{L^\infty(\rho_1\rho_2)}
&       \lesssim 2^{-(\alpha+\beta)j}\int_0^t\e^{-\lambda s}(1\wedge (s4^j)^{-2})\dif s
\|f\|_{\mS_{T,a}^{\alpha}(\rho_1)}\|g\|_{\mC_{T,a}^\beta(\rho_2)}\\
&\lesssim2^{-(\alpha+\beta+\theta)j}(\lambda\vee1)^{\frac{\theta-2}{2}}\|f\|_{\mS_{T,a}^{\alpha}(\rho_1)}\|g\|_{\mC_{T,a}^\beta(\rho_2)}.
 	\end{align*}
	Thus we obtain \eqref{GA44} for $k=0$. Note that
\begin{align*}
	[\nabla_v \sI_\lambda, f \prec ]g=\nabla_v[\sI_\lambda,f\prec]g+\nabla_vf(t)\prec \sI_\lambda g.
\end{align*}
Estimate \eqref{GA44} for $k=1$ follows by what we have proved and Lemma \ref{lem:para} and \eqref{SchS}.
Thus we complete the proof.
\end{proof}
The following commutator estimate is straightforward by Lemma \ref{commutator1}, Lemmas \ref{lem:para} and Lemma \ref{lem:com2}. Since we will use it many times later, we write it as a lemma.
 \bl\label{Le231}
Let $\rho_1,\rho_2,\rho_3\in\sP_{\rm w}$. For any $\alpha\in(1,2)$, $\gamma\in\mR$ and $\beta<0$ with $\alpha+\beta>1$,  $\alpha+\beta+\gamma>0$ and $1+\beta+\gamma<0$,  we have
$$
\|[b\circ\nabla_v\sI_\lambda,\phi] f\|_{\mC_{T,a}^{\alpha+\beta+\gamma}(\rho_1\rho_2\rho_3)}
\lesssim \|\phi\|_{\mS_{T,a}^{\alpha-1}(\rho_1)}\|f\|_{\mC_{T,a}^{\beta}(\rho_2)}\|b\|_{\mC_{T,a}^{\gamma}(\rho_3)}.
$$
\el

\begin{proof}
Note that
\begin{align*}
[b\circ\nabla_v\sI_\lambda,\phi] f&=b\circ\nabla_v \sI_\lambda (\phi\succcurlyeq f)+b\circ\nabla_v \sI_\lambda (\phi\prec f)-\phi(b\circ\nabla_v \sI_\lambda f)
	\\&=b\circ\nabla_v \sI_\lambda (\phi\succcurlyeq f)+b\circ[\nabla_v \sI_\lambda ,\phi\prec]f+\textrm{com}(\phi,\nabla_v \sI_\lambda f,b).
	\end{align*}
By \eqref{GZ2}, \eqref{EG1} and \eqref{GZ1}, we have
\begin{align*}
\|b\circ\nabla_v \sI_\lambda (\phi\succcurlyeq f)\|_{\mC_{T,a}^{\alpha+\beta+\gamma}(\rho_1\rho_2\rho_3)}
&\lesssim\|\nabla_v \sI_\lambda (\phi\succcurlyeq f)\|_{\mC_{T,a}^{\alpha+\beta}(\rho_1\rho_2)}\|b\|_{\mC_{T,a}^{\gamma}(\rho_3)}
\\&\lesssim\|\phi\succcurlyeq f\|_{\mC_{T,a}^{\alpha+\beta-1}(\rho_1\rho_2)}\|b\|_{\mC_{T,a}^{\gamma}(\rho_3)}
\\&\lesssim\|\phi\|_{\mC_{T,a}^{\alpha-1}(\rho_1)}\|f\|_{\mC_{T,a}^{\beta}(\rho_2)}\|b\|_{\mC_{T,a}^{\gamma}(\rho_3)}.
	\end{align*}
By \eqref{GZ2}, \eqref{GZ3} and \eqref{GA44}, we have
	\begin{align*}
\|b\circ[\nabla_v \sI_\lambda ,\phi\prec]f\|_{\mC_{T,a}^{\alpha+\beta+\gamma}(\rho_1\rho_2\rho_3)}
&\lesssim\|[\nabla_v \sI_\lambda ,\phi\prec]f\|_{\mC_{T,a}^{\alpha+\beta}(\rho_1\rho_2)}\|b\|_{\mC_{T,a}^{\gamma}(\rho_3)}
\\&\lesssim\|\phi\|_{\mS_T^{\alpha-1}(\rho_1)}\|f\|_{\mC_{T,a}^{\beta}(\rho_2)}\|b\|_{\mC_{T,a}^{\gamma}(\rho_3)}.
	\end{align*}
By \eqref{FA1}, \eqref{EG1} and \eqref{GZ3}, we have
	\begin{align*}
\|\textrm{com}(\phi,\nabla_v \sI_\lambda f,b)\|_{\mC_{T,a}^{\alpha+\beta+\gamma}(\rho_1\rho_2\rho_3)}
&\lesssim\|\phi\|_{\mC_{T,a}^{\alpha-1}(\rho_3)}\|\nabla_v \sI_\lambda f\|_{\mC_{T,a}^{1+\beta}(\rho_2)}\|b\|_{\mC_{T,a}^{\gamma}(\rho_1)}
\\&\lesssim\|\phi\|_{\mC_{T,a}^{\alpha-1}(\rho_3)}\|f\|_{\mC_{T,a}^{\beta}(\rho_2)}\|b\|_{\mC_{T,a}^{\gamma}(\rho_1)}.
	\end{align*}
Combining the above calculations, we obtain the desired estimate.
\end{proof}


\subsection{Renormalized pairs}\label{sec:2.5}
In this subsection we introduce the  renormalized pairs. Fix $\alpha\in(\frac12,\frac23)$ and $\rho_1,\rho_2\in\sP_{\rm w}$.
For $T>0$, let $b=(b_1,\cdots, b_d)$ and $f$ be $d+1$-distributions in $\mC_{T,a}^{-\alpha}(\rho_1)$ and $\mC_{T,a}^{-\alpha}(\rho_2)$ respectively.
We introduce the following important quantity for later use: for $q\in[1,\infty]$,
\begin{align}\label{AA9}
\begin{split}
\mA^{b,f}_{T,q}(\rho_1,\rho_2)&:=\sup_{\lambda\geq 0}\|b\circ\nabla_v\sI_\lambda f\|_{L^q_T\bC_a^{1-2\alpha}(\rho_1\rho_2)}+
(\|b\|_{\mC_{T,a}^{-\alpha}(\rho_1)}+1)\|f\|_{L^q_T\bC_a^{-\alpha}(\rho_2)}.
\end{split}
\end{align}
By \eqref{GZ2}, $b(t)\circ\nabla_v\sI_\lambda f(t)$ is not well-defined for $\alpha>\frac12$ since by Schauder's estimate,
we only  have (see Lemma \ref{Le11})
$$
\nabla_v\sI_\lambda f\in \mC_{T,a}^{1-\alpha}(\rho_2).
$$
However, in the probabilistic sense, it is possible to give a meaning to $b\circ\nabla_v\sI_\lambda f$ when
$b, f$ are some Gaussian noises (see {Section \ref{Sub6} for general probabilistic assumptions  and examples for Gaussian noises to satisfy the requirement in Definition \ref{Def216} below}). This motivates us to introduce the following notion.
\bd\label{Def216}
We call the above $(b,f)\in \mC_{T,a}^{-\alpha}(\rho_1)\times \mC_{T,a}^{-\alpha}(\rho_2)$
a renormalized pair if there exists a sequence of $(b_n, f_n)\in L^\infty_T C_b^\infty\times L^\infty_TC_b^\infty$
with
\begin{align}\label{Lim010}
\sup_{n\in\mN}\mA^{b_n,f_n}_{T,\infty}(\rho_1,\rho_2)<\infty
\end{align}
and such that
\begin{align}\label{Lim0}
\lim_{n\to\infty}\(\|b_n-b\|_{\mC_{T,a}^{-\alpha}(\rho_1)}+\|f_n-f\|_{\mC_{T,a}^{-\alpha}(\rho_2)}\)=0,
\end{align}
and  for each $\lambda\geq 0$, there exists a distribution $b\circ\nabla_v\sI_\lambda f\in \mC_{T,a}^{1-2\alpha}(\rho_1\rho_2)$ such that
\begin{align}\label{Lim1}
\lim_{n\to\infty}\|b_n\circ\nabla_v\sI_\lambda f_n-b\circ\nabla_v\sI_\lambda f\|_{\mC_{T,a}^{1-2\alpha}(\rho_1\rho_2)}=0.
\end{align}
The set of all the above renormalized pair is denoted by $\mB^\alpha_T(\rho_1,\rho_2)$.
If for each $i=1,\cdots,d$, $(b,b_i)\in \mB^\alpha_T(\rho_1,\rho_1)$, we simply say $b\in \mB^\alpha_T(\rho_1)$, a renormalized vector field.
\ed

For a renormalized pair $(b,f)\in\mB^\alpha_T(\rho_1,\rho_2)$, it always associates with certain approximation sequence $(b_n,f_n)_{n\in\mN}$. The key point is of course the convergence in \eqref{Lim1},
which in general does not imply that for $(b,f), (b',f)\in\mB^\alpha_T(\rho_1,\rho_2)$,
$$
(b+b',f)\in \mB^\alpha_T(\rho_1,\rho_2).
$$
In other words, $\mB^\alpha_T(\rho_1,\rho_2)$ is not a linear space.
But we have the following easy lemma.
\bl\label{Le317}
For $(b,f)\in\mB^\alpha_T(\rho_1,\rho_2)$ and $b'\in \mC_{T,a}^\beta(\rho_1)$ with $\beta>\alpha-1$, we have
$$
(b+b', f)\in\mB^\alpha_T(\rho_1,\rho_2).
$$
\el
\begin{proof}
Let $(b_n,f_n)_{n\in\mN}$ be the approximation sequence in the definition of $(b,f)\in\mB^\alpha_T(\rho_1,\rho_2)$.
Let $\varphi_n$ be any mollifiers in $\mR^{2d}$ and define $b'_n(t,\cdot):=b'(t,\cdot)*\varphi_n(\cdot)$. By definition, it is easy to see that
$$
\sup_{n\in\mN}\mA^{b_n+b_n',f_n}_{T,\infty}(\rho_1,\rho_2)\leq \sup_{n\in\mN}\Big(\mA^{b_n,f_n}_{T,\infty}(\rho_1,\rho_2)+\mA^{b'_n,f_n}_{T,\infty}(\rho_1,\rho_2)\Big)<\infty.
$$
For any $\gamma\in(\alpha-1,\beta)$, by \eqref{FG1c} we clearly have
$$
\lim_{n\to\infty}\|b'_n-b'\|_{\mC_{T,a}^\gamma(\rho_1)}=0,
$$
and by \eqref{GZ2} and \eqref{Lim0},
\begin{align*}
&\lim_{n\to\infty}\|b'_n\circ\nabla_v\sI_\lambda f_n-b'\circ\nabla_v\sI_\lambda f\|_{\mC_{T,a}^{0}(\rho_1\rho_2)}\\
&\quad\leq\lim_{n\to\infty}\|b'_n\|_{\mC_{T,a}^\gamma(\rho_1)}\|\nabla_v\sI_\lambda(f_n-f)\|_{\mC_{T,a}^{1-\alpha}(\rho_2)}\\
&\qquad+\lim_{n\to\infty}\|b'_n-b'\|_{\mC_{T,a}^\gamma(\rho_1)}\|\nabla_v\sI_\lambda f\|_{\mC_{T,a}^{1-\alpha}(\rho_2)}=0.
\end{align*}
The proof is complete.
\end{proof}

To eliminate the parameter $\lambda$ in \eqref{AA9}, we give the following lemma, {the proof of which follows from \cite[Lemma 2.16]{ZZZ20}.}
\bl\label{lem:lambda} Let $\sI^{t}_s(f)=\int_s^t P_{t-r}f(r)\dif r$. For any $t>0$, we have
\begin{equation}\label{I}
\sup_{\lambda\geq 0}\|b(t)\circ\nabla_v \sI_\lambda f(t)\|_{\bC_a^{1-2\alpha}(\rho)}\leq 2\sup_{s\in [0,t]}\| b(t)\circ \nabla_v \sI^{t}_s(f)\|_{\bC_a^{1-2\alpha}(\rho)}.
\end{equation}
\el

The following localized property about the operation $\circ$ is  useful.
\bl\label{le:loc}
Let $T>0$, $\rho_1,\rho_2,\rho_3,\rho_4\in\sP_{\rm w}$, $\alpha\in(\frac12,\frac23)$ and $\gamma\in(\alpha,1)$.
Suppose
$$
(b,f)\in\mB^\alpha_T(\rho_1,\rho_2),\ \psi\in \mC_{T,a}^{\gamma}(\rho_3),\  \phi\in\mS^{\gamma}_{T,a}(\rho_4).
$$
Then $(b\psi,f\phi)\in\mB^\alpha_T(\rho_1\rho_3,\rho_2\rho_4)$ with approximation sequence $(b_n\psi, f_n\phi)$,
and there is a $C>0$ depending only on $T, \gamma,\alpha,d,\rho_i$
such that for all $\lambda\geq 0$,
\begin{align}\label{EK0}
\begin{split}
&\|(b\psi)\circ\nabla_v\sI_\lambda (f\phi)-\psi\phi( b\circ\nabla_v \sI_\lambda f)\|_{\mC_{T,a}^{ 1+\gamma-2\alpha}
(\rho_1\rho_2\rho_3\rho_4)}\\
&\qquad\lesssim_C\|b\|_{\mC_{T,a}^{-\alpha}(\rho_1)}\|f\|_{\mC_{T,a}^{-\alpha}(\rho_2)}
\|\psi\|_{\mC_{T,a}^{\gamma}(\rho_3)}\|\phi\|_{\mS_{T,a}^{\gamma}(\rho_4)}.
\end{split}
\end{align}
\el
\begin{proof}
By approximation, we only prove the a priori estimate \eqref{EK0}. Note that
	\begin{align*}
	I&:=(b\psi)\circ\nabla_v\sI_\lambda (f\phi)-\psi\phi( b\circ\nabla_v \sI_\lambda f)\\
	&=[(b\psi)\circ\nabla_v\sI_\lambda, \phi] f+\phi[\nabla_v\sI_\lambda f\circ,\psi]b=:I_1+I_2.
	\end{align*}
	For $I_1$, since $1+\gamma-2\alpha>0$ and $1-2\alpha<0$, {$\gamma+1-\alpha>1$} by Lemma \ref{Le231}, we have
	\begin{align*}
	\|I_1\|_{\mC_{T,a}^{1+\gamma-2\alpha}(\rho_1\rho_2\rho_3\rho_4)}
	&\lesssim \|\phi\|_{\mS_{T,a}^{\gamma}(\rho_4)}\|f\|_{\mC_{T,a}^{-\alpha}(\rho_2)}\|b\psi\|_{\mC_{T,a}^{-\alpha}(\rho_1\rho_3)}\\
	&\lesssim\|\phi\|_{\mS_{T,a}^{\gamma}(\rho_4)}\|f\|_{\mC_{T,a}^{-\alpha}(\rho_2)}\|b\|_{\mC_{T,a}^{-\alpha}(\rho_1)}\|\psi\|_{\mC_{T,a}^{\gamma}(\rho_3)}.
	\end{align*}
	For $I_2$, similarly by  \eqref{Le12}, we have
	\begin{align*}
	\|I_2\|_{\mC_{T,a}^{ 1+\gamma-2\alpha}
	(\rho_1\rho_2\rho_3\rho_4)}
	&\lesssim\|\phi\|_{\mC_{T,a}^{ \gamma}
	(\rho_4)}\|[\nabla_v\sI_\lambda f\circ,\psi]b\|_{\mC_{T,a}^{ 1+\gamma-2\alpha}
	(\rho_1\rho_2\rho_3)}\\
	&\lesssim\|\phi\|_{\mC_{T,a}^{ \gamma}(\rho_4)}\|\nabla_v\sI_\lambda f\|_{\mC_{T,a}^{1-\alpha}(\rho_2)}
	\|b\|_{\mC_{T,a}^{-\alpha}(\rho_1)}\|\psi\|_{\mC_{T,a}^{ \gamma}(\rho_3)}.
	\end{align*}
Thus we obtain \eqref{EK0} by \eqref{SchS}. 
\end{proof}

\section{Linear kinetic equations with distribution drifts}\label{ParaA}
Now that the necessary facts about the kinetic semigroup and weighted Besov spaces are established, the next two sections are devoted to the actual construction of the solution to the stochastic kinetic equation.
The aim of this section is to show the well-posedness of the following linear singular kinetic equation: for $\lambda\geq 0$,
\begin{align}\label{PDE7}
\sL_\lambda u:=(\p_t-\Delta_v-v\cdot \nabla_x+\lambda) u=b\cdot\nabla_v u+f,\quad u(0)=\varphi,
\end{align}
where $b=(b_1,\cdots, b_d)$ and $f$ satisfy
that for some $\alpha\in(\frac{1}{2},\frac{2}{3})$ and $\rho_1,\rho_2\in\sP_{\rm w}$,
\begin{align}\label{DW9}
\begin{split}
\left\{
\begin{aligned}
&(b, f)\in  \mB^\alpha_T(\rho_1,\rho_2),\ \ b\in  \mB^\alpha_T(\rho_1),\ \ \ T>0,\\
&\mbox{ have the same approximation sequence $(b_n,f_n)$.}
\end{aligned}
\right.
\end{split}
\end{align}
For simplicity of notations, we shall write
$$
 \ell^b_T(\rho_1):={\sum_{i=1}^d\mA^{b,b_i}_{T,\infty}(\rho_1,\rho_1)+1}.
$$
{We also write
$$ \ell^b_T=\ell^b_T(1)\quad \mA^{b,b}_{T,q}=\sum_{i=1}^d\mA^{b,b_i}_{T,q}(1,1).$$}


In Subsection \ref{sec:5.1}, we first introduce the notion of paracontrolled solutions, and then establish a localization property
for paracontrolled solutions. Such a localization is natural for classical solutions by the chain rule. However, for paracontrolled solutions, it is
quite involved since the renormalized pair is defined in the approximation level.
In Subsection \ref{sub:Schauder}, following the same argument as in  \cite[Section 3]{ZZZ20} {and using estimate and commutators for the kinetic semigroup},
we show the well-posedness for PDE \eqref{DW9} in weighted anisotropic H\"older spaces. We emphasize that  unlike using the exponential weight
technique in \cite[Section 3]{ZZZ20},  the uniqueness
is a direct consequence of the a priori estimate \eqref{MN1} below.

\subsection{Paracontrolled solutions}\label{sec:5.1}

To introduce the paracontrolled solution of PDE \eqref{PDE7},
we make the following paracontrolled ansatz as in \cite{GIP15}:
\begin{align}\label{DT11}
u=P_t\varphi+u^\sharp+\nabla_v u\prec  \sI_\lambda b+\sI_\lambda f,
\end{align}
where $u^\sharp$ solves the following equation
\begin{align}\label{DT110}
 u^\sharp=&\sI_\lambda(\nabla_v u\succ b+b\circ\nabla_v u)+[\sI_\lambda, \nabla_v u\prec] b.
\end{align}
Note that $b\circ\nabla_v u$ is not well-defined in the classical sense. We give its definition by paracontrolled ansatz and renormalized pair as follows:
By \eqref{DT11}, we can write
\begin{align}
b\circ\nabla_v u&=b\circ\nabla_v u^\sharp+b\circ\nabla_v(\nabla_v u\prec  \sI_\lambda b)+b\circ\nabla_v\sI_\lambda f
+b\circ\nabla_v P_t\varphi\no\\
&=b\circ\nabla_v u^\sharp+b\circ(\nabla_v^2 u\prec \sI_\lambda b)+(b\circ\nabla_v \sI_\lambda b)\cdot \nabla_v u\no\\
&\quad+\mathrm{com}(\nabla_v u, \nabla_v \sI_\lambda b,b)+b\circ\nabla_v\sI_\lambda f+b\circ\nabla_v P_t\varphi.\label{FQ2}
\end{align}
This motivates us to introduce the following definition.
\begin{definition}\label{def:para1}
Let $T>0$, $\rho_1\in\sP_{\rm w}$ be a bounded weight and
$\rho_2,\rho_3\in\sP_{\rm w}$ be any weights.
Under \eqref{DW9}  and $\varphi\in \bC_a^{1+\alpha+\varepsilon}(\rho_2/\rho_1)$ for some $\eps>0$, we call
$u\in \mS^{2-\alpha}_T(\rho_3)$ a paracontrolled solution of PDE \eqref{PDE7} corresponding to $(b,f)$
if for some $\rho_4\in\sP_{\rm w}$,
\begin{align}\label{DD0}
u-P_t\varphi-\nabla_v u\prec  \sI_\lambda b-\sI_\lambda f=:u^\sharp\in\mC_{T,a}^{3-2\alpha}(\rho_4)
\end{align}
satisfies \eqref{DT110} with $b\circ\nabla_v u$ given by \eqref{FQ2} which is well-defined by \eqref{BV1} below.
\end{definition}

\br\label{Re42}\rm
In the above definition, if we consider $\bar u=u-P_t\varphi$, then the initial value is reduced to zero.  In this case,
the nonhomogeneous $f$ shall be replaced by
$$
\bar f=f+b\cdot\nabla_v P_t \varphi\in \mC_{T,a}^{-\alpha}(\rho_2).
$$
By Lemma \ref{le:loc} with $\psi=1, \phi=\nabla_vP_t \varphi$ and $\rho_3=1, \rho_4=\rho_2/\rho_1$,
	 \begin{align*}
	 &\|b\circ\nabla_v \sI_\lambda(b\cdot \nabla_vP_t \varphi)\|_{\mC_{T,a}^{1-2\alpha}(\rho_1\rho_2)}
	 \lesssim\|\varphi\|_{\bC^{1+\alpha+\eps}_{a}(\rho_2/\rho_1)}\ell^b_{T}(\rho_1).
	 \end{align*}
Thus, we still have
	 $$
	 (b,\bar f)\in\mB^\alpha_T(\rho_1,\rho_2),
	 $$
and $(\bar u,\bar u^\sharp)$ is a paracontrolled solution of \eqref{PDE7} with $f=\bar f$ and $\bar u(0)=0$,
where
$$
\bar u^\sharp=u^\sharp+\nabla_v P_t\varphi\prec\sI_\lambda b-\sI_\lambda (b\cdot\nabla_v P_t\varphi).
$$
In the following, for simplicity, we may and shall assume $\varphi\equiv0$ by this procedure.
\er

We have the following a priori estimate about the regularity of $u^\sharp$.
\bt\label{Le32}
Let $u\in\mS_{T,a}^{2-\alpha}(\rho_3)$ be a paracontrolled solution to \eqref{PDE7} in the sense of Definition \ref{def:para1}
with $\varphi\equiv 0$. For any $\eps>\frac{2\alpha-1}{2-3\alpha}$
and $\rho_4:=\rho_1^{1+\eps}((\rho_1\rho_3)\wedge\rho_2)$, there is a constant $C=C(T,\eps,\alpha,d, \rho_i,\ell^b_T(\rho_1))>0$
such that for all $\lambda\geq 0$,
\begin{align}\label{CC9}
\|u^\sharp\|_{\mC_{T,a}^{3-2\alpha}(\rho_4)}\lesssim_C \|u\|_{\mS_{T,a}^{2-\alpha}(\rho_3)}+\mA^{b,f}_{T,\infty}(\rho_1,\rho_2).
\end{align}
\et

\begin{proof}
First of all, we show that for  any $\gamma,\beta\in(\alpha,2-2\alpha]$ and $\rho_5\leq(\rho_1\rho_3)\wedge\rho_2$,
	\begin{align}
\|b\circ\nabla_v u\|_{\mC_{T,a}^{1-2\alpha}(\rho_1\rho_5)}&\lesssim
\ell^b_T(\rho_1)\(\|u\|_{\mC_{T,a}^{\alpha+\gamma}(\rho_3)}
+\|u^\sharp\|_{\mC_{T,a}^{\beta+1}(\rho_5)}\)+\mA^{b,f}_{T,\infty}(\rho_1,\rho_2).\label{BV1}
\end{align}
To prove this,  it suffices to estimate each term in \eqref{FQ2}.
	\begin{enumerate}[$\bullet$]
	\item Since $\beta>\alpha$, by \eqref{GZ2}, we have
		\begin{align*}
		\|b\circ\nabla_v u^\sharp\|_{L^\infty(\rho_1\rho_5)}
		&\lesssim\|b\|_{\mC_{T,a}^{-\alpha}(\rho_1)}\|\nabla_v u^\sharp\|_{\mC_{T,a}^{\beta}(\rho_5)}
		\leq\ell^b_T(\rho_1)\|u^\sharp\|_{\mC_{T,a}^{\beta+1}(\rho_5)}.
		\end{align*}
		\item Since $\gamma>\alpha$ and $\gamma+\alpha-2<0$, by \eqref{GZ1}, \eqref{GZ2}, we have
		\begin{align*}
		\|b\circ(\nabla_v^2 u\prec \sI_\lambda b)\|_{\mC_{T,a}^{1-2\alpha}(\rho^2_1\rho_3)}
		&\lesssim
		\|b\|_{\mC_{T,a}^{-\alpha}(\rho_1)}\|\nabla_v^2 u\prec \sI_\lambda b\|_{\mC_{T,a}^{\gamma}(\rho_1\rho_3)}\\
		&\lesssim  \|b\|_{\mC_{T,a}^{-\alpha}(\rho_1)}\|\nabla_v^2 u\|_{\mC_{T,a}^{\gamma+\alpha-2}(\rho_3)}
		\|\sI_\lambda b\|_{\mC_{T,a}^{2-\alpha}(\rho_1)}\\
		&\lesssim  \|b\|^2_{\mC_{T,a}^{-\alpha}(\rho_1)}\|u\|_{\mC_{T,a}^{\gamma+\alpha}(\rho_3)}
		\lesssim \ell^b_T(\rho_1)\|u\|_{\mC_{T,a}^{\alpha+\gamma}(\rho_3)}.
		\end{align*}
		\item Since ${\gamma>\alpha}$, by \eqref{GZ3} we have
		\begin{align*}
		\|\nabla_v u(b\circ\nabla_v \sI_\lambda b)\|_{\mC_{T,a}^{1-2\alpha}(\rho^2_1\rho_3)}
		&\lesssim  \|\nabla_v u\|_{\mC_{T,a}^{\gamma+\alpha-1}(\rho_3)}\|b\circ\nabla_v \sI_\lambda b\|_{\mC_{T,a}^{1-2\alpha}(\rho^2_1)}\\
		&\lesssim \ell^b_T(\rho_1)\|u\|_{\mC_{T,a}^{\alpha+\gamma}(\rho_3)}.
		\end{align*}
		\item Since 
		$\gamma>\alpha$, by \eqref{FA1}, we have
		\begin{align*}
		\|{\rm com}\|_{\mC_{T,a}^{1-2\alpha}(\rho_1^2\rho_3)}&\lesssim \|b\|_{\mC_{T,a}^{-\alpha}(\rho_1)}
		\|\nabla_v u\|_{\mC_{T,a}^{\gamma+\alpha-1}(\rho_3)}\|\nabla_v \sI_\lambda b\|_{\mC_{T,a}^{1-\alpha}(\rho_1)}\\
		&\lesssim\|b\|^2_{\mC_{T,a}^{-\alpha}(\rho_1)}\|u\|_{\mC_{T,a}^{\gamma+\alpha}(\rho_3)}
		\lesssim \ell^b_T(\rho_1)\|u\|_{\mC_{T,a}^{\alpha+\gamma}(\rho_3)}.
		\end{align*}
	\end{enumerate}
	Combining the above estimates and by $\rho_1\rho_5\leq\rho_1^2\rho_3$, we get \eqref{BV1}.
	
	On the other hand, by \eqref{GZ1}, we have
		$$
		\|\nabla_v u\succ b\|_{\mC_{T,a}^{1-2\alpha}(\rho_1\rho_3)}\lesssim
		\|u\|_{\mC_{T,a}^{2-\alpha}(\rho_3)}\|b\|_{\mC_{T,a}^{-\alpha}(\rho_1)},
		$$	
and by \eqref{GA44} with $(k,\theta)=(0,2)$ and \eqref{Lem34},
$$
		\|[\sI_\lambda,\nabla_v u \prec] b\|_{\mC_{T,a}^{3-2\alpha}(\rho_1\rho_3)}\lesssim
		\|\nabla_vu\|_{\mS_T^{1-\alpha}(\rho_3)}\|b\|_{\mC_{T,a}^{-\alpha}(\rho_1)}\lesssim
		\|u\|_{\mS_T^{2-\alpha}(\rho_3)}\|b\|_{\mC_{T,a}^{-\alpha}(\rho_1)}.
		$$
Thus, by  \eqref{DT110}, \eqref{BV1} and Schauder's estimate \eqref{SchS}, thanks to $\rho_5\leq\rho_3$, we obtain
for $\beta\in(\alpha,2-2\alpha)$,
	\begin{align}\label{CC8}
	\|u^\sharp\|_{\mC_{T,a}^{3-2\alpha}(\rho_1\rho_5)}&\lesssim \|u\|_{\mS_T^{2-\alpha}(\rho_3)}
	+\|u^\sharp\|_{\mC_{T,a}^{\beta+1}(\rho_5)}+\mA^{b,f}_{T,\infty}(\rho_1,\rho_2).
	\end{align}
	For $\eps>\frac{2\alpha-1}{2-3\alpha}$, one can choose $\beta$ close to $\alpha$ so that
	$$
	\theta:=\tfrac{\eps}{1+\eps}=\tfrac{\alpha+\beta-1}{1-\alpha}.
	$$
Let
	$$
	\rho_4:=\rho_1^{1+\eps}((\rho_1\rho_3)\wedge\rho_2),\ \ \rho_5:=\rho_4^\theta((\rho_1\rho_3)\wedge\rho_2)^{1-\theta}.
	$$
Noting that $\rho_1\rho_5=\rho_4$, by \eqref{Embq0} and Young's inequality, we have for any $\delta>0$,
	\begin{align*}
	\|u^\sharp\|_{\mC_{T,a}^{\beta+1}(\rho_5)}
	&\lesssim \|u^\sharp\|_{\mC_{T,a}^{3-2\alpha}(\rho_4)}^{\theta}
	\|u^\sharp\|_{\mC_{T,a}^{2-\alpha}((\rho_1\rho_3)\wedge\rho_2)}^{1-\theta}\\
	&\leq \delta\|u^\sharp\|_{\mC_{T,a}^{3-2\alpha}(\rho_4)}
	+C_\delta\|u^\sharp\|_{\mC_{T,a}^{2-\alpha}((\rho_1\rho_3)\wedge\rho_2)}.
	\end{align*}
Substituting this into \eqref{CC8} and by $\rho_4=\rho_1\rho_5$ and letting $\delta$ small enough, we get
	\begin{align}\label{BV2}
	\|u^\sharp\|_{\mC_{T,a}^{3-2\alpha}(\rho_4)}\lesssim \|u\|_{\mS_T^{2-\alpha}(\rho_3)}+\|u^\sharp\|_{\mC_{T,a}^{2-\alpha}((\rho_1\rho_3)\wedge\rho_2)}
	+\mA^{b,f}_{T,\infty}(\rho_1,\rho_2).
	\end{align}
	On the other hand, by \eqref{DD0}, \eqref{GZ1} and \eqref{EG1} we have
	\begin{align}\label{b:sharp1}
	\|u^\sharp\|_{\mC_{T,a}^{2-\alpha}((\rho_1\rho_3)\wedge\rho_2)}
	&\lesssim\|u\|_{\mC_{T,a}^{2-\alpha}(\rho_3)}+\|\nabla_v u\prec\sI_\lambda b\|_{\mC_{T,a}^{2-\alpha}(\rho_1\rho_3)}
	+\|\sI_\lambda f\|_{\mC_{T,a}^{2-\alpha}(\rho_2)}\no\\
	&\lesssim\|u\|_{\mC_{T,a}^{2-\alpha}(\rho_3)}+\|\nabla_v u\|_{L^\infty_T(\rho_3)}\|\sI_\lambda b\|_{\mC_{T,a}^{2-\alpha}(\rho_1)}
	+\|f\|_{\mC_{T,a}^{-\alpha}(\rho_2)}\no\\
	&\lesssim\sqrt{\ell^b_T(\rho_1)}\|u\|_{\mC_{T,a}^{2-\alpha}(\rho_3)}+\|f\|_{\mC_{T,a}^{-\alpha}(\rho_2)}.
	\end{align}
Substituting this into \eqref{BV2}, we complete the proof.
\end{proof}

{For the uniqueness part we need} the following localization result about the paracontrolled solutions.
\bp\label{Pr42}
Let $u$ be a paracontrolled solution to PDE \eqref{PDE7} with $\varphi=0$. Let $\phi,\psi\in C^\infty_c(\mR^{2d})$ with $\psi\equiv 1$ on the support of $\phi$.
Then $\bar u:=u\phi\in\mS_{T,a}^{2-\alpha}$ is also a paracontrolled solution to PDE \eqref{PDE7} corresponding to $(\bar b,g)\in\mB_T^\alpha$,
where
$$
\bar b:=b\psi,\ \ g:=\phi f-u\Delta_v\phi-2\nabla_v\phi\cdot\nabla_vu-(v\cdot\nabla_x\phi)u-(b\cdot \nabla_v\phi)u.
$$
\ep
\begin{proof}
Without loss of generality we assume that $\lambda=0$.
First of all, by Lemmas \ref{Le317} and \ref{le:loc}, $(\bar b,g)\in\mB_T^\alpha$.
By definition, one needs to show that
\begin{align}\label{An22}
\bar u-\nabla_v \bar u\prec  \sI \bar b-\sI g=:\bar u^\sharp\in\mC_{T,a}^{3-2\alpha}
\end{align}
satisfies
\begin{align}\label{An23}
 \bar u^\sharp=\sI(\nabla_v\bar u\succ \bar b+\bar b\circ\nabla_v \bar u)+[\sI, \nabla_v\bar u\prec] \bar b,
\end{align}
with
\begin{align}
\begin{split}
\bar b\circ\nabla_v\bar u&:=\bar b\circ\nabla_v \bar u^\sharp+
\bar b\circ(\nabla_v^2 \bar u\prec \sI\bar b)+(\bar b\circ\nabla_v \sI\bar b)\cdot \nabla_v\bar u\\
&\quad+\mathrm{com}(\nabla_v\bar u, \nabla_v \sI\bar b,\bar b)+\bar b\circ\nabla_v\sI g.\label{An24}
\end{split}
\end{align}
Since $u$ is a paracontrolled solution, by definition we have
\begin{align}\label{AN3}
u=\sI(b\star\nabla_vu+f),
\end{align}
where
\begin{align}\label{AG9}
b\star\nabla_v u:=\nabla_v u\succ b+b\circ\nabla_v u+\nabla_v u\prec b.
\end{align}
Let $(b_n,f_n)\in L^\infty_TC^\infty_b$ be as in \eqref{DW9}.
We introduce an approximation of $u$ by
\begin{align}\label{AG009}
u_n:=u^\sharp+\nabla_v u\prec  \sI b_n+\sI f_n,\quad \ \bar b_n:=b_n\psi,\quad \ \bar u_n:=u_n\phi,
\end{align}
and
\begin{align}\label{NP1}
\bar b\circledcirc\nabla_v \bar u:=(b\star\nabla_v u)\phi+(b\cdot\nabla_v \phi) u-\nabla_v \bar u\prec \bar b-\nabla_v \bar u\succ \bar b.
\end{align}
{In the classical case, it is easy to see $\bar b\circledcirc\nabla_v \bar u=\bar b\circ  \nabla_v \bar u$. In the paracontrolled case this is not obvious and we introduce $\bar b\circledcirc\nabla_v \bar u$ which can be easily checked as limit of $\bar b_n\circ\nabla_v \bar u_n$ (see step (ii) below). Moreover, it is not hard to prove that $\bar{u}^\sharp$ satisfies \eqref{An23} with $\bar b\circ \nabla_v \bar u$ replaced by $\bar b \circledcirc\nabla_v \bar u$ (see step (iii) below). Finally we use approximations to prove $\bar b\circledcirc\nabla_v \bar u=\bar b\circ  \nabla_v \bar u$ (see step (iv) below).}
Our proof is divided into the following four steps:
\medskip\\
({\it i}) We show that $u_n$ is a suitable approximation of $u$ and for some $\rho\in\sP_{\rm w}$,
\begin{align}\label{Lim01}
\lim_{n\to\infty}\|b_n\cdot\nabla_v u_n-b\star\nabla_v u\|_{\mC^{-\alpha}_{T,a}(\rho)}=0.
\end{align}
({\it ii}) We prove $\bar b\circledcirc\nabla_v \bar u\in \mC^{1-2\alpha}_{T,a}$ and
\begin{align}\label{ES9}
\lim_{n\to\infty}\|\bar b_n\circ\nabla_v \bar u_n-\bar b\circledcirc\nabla_v \bar u\|_{\mC_{T,a}^{-\alpha}}=0.
\end{align}
({\it iii}) We show that for $\bar u^\sharp$ being defined by \eqref{An22} satisfies the following,
\begin{align}\label{NP55}
\mC_{T,a}^{3-2\alpha}\ni \bar u^\sharp=\sI(\nabla_v\bar u\succ \bar b+\bar b\circledcirc\nabla_v \bar u)+[\sI, \nabla_v\bar u\prec] \bar b.
\end{align}
({\it iv}) With $\bar b\circ\nabla_v \bar u$ being defined by \eqref{An24}, we prove
\begin{align}\label{EQU}
\bar b\circledcirc\nabla_v \bar u=\bar b\circ\nabla_v \bar u.
\end{align}
{\it Proof of (i): }
First of all,  by \eqref{DD0}, \eqref{AG009}, \eqref{GZ0} and \eqref{EG1}, we have
\begin{align*}
\|u_n-u\|_{\mC_{T,a}^{2-\alpha}((\rho_1\rho_3)\wedge\rho_2)}
\lesssim\|\nabla_v u\|_{\mL^\infty_T(\rho_3)}\|b_n-b\|_{\mC_{T,a}^{-\alpha}(\rho_1)}+\|f_n-f\|_{\mC_{T,a}^{-\alpha}(\rho_2)},
\end{align*}
which implies by \eqref{Lim0} that
\begin{align}\label{G3}
\lim_{n\to\infty}\|u_n-u\|_{\mC_{T,a}^{2-\alpha}((\rho_1\rho_3)\wedge\rho_2)}=0.
\end{align}
Next, by \eqref{GZ1}, \eqref{G3} and \eqref{Lim0}, we also have for some $\rho\in \sP_{\rm w}$,
\begin{align}
\lim_{n\to\infty}\|b_n\prec\nabla_v u_n-b\prec\nabla_v u\|_{\mC_{T,a}^{1-2\alpha}(\rho)}=0,\label{G1}
\end{align}
and by \eqref{GZ0}, \eqref{G3} and \eqref{Lim0},
\begin{align}
\lim_{n\to\infty}\|b_n\succ\nabla_v u_n-b\succ\nabla_v u\|_{\mC_{T,a}^{-\alpha}(\rho)}=0.\label{G2}
\end{align}
Moreover, note that by \eqref{AG009},
\begin{align*}
b_n\circ\nabla_v u_n&=b_n\circ\nabla_v u^\sharp+b_n\circ(\nabla^2_vu\prec  \sI b_n)+(b_n\circ\nabla_v\sI b_n)\cdot\nabla_v u\\
&\quad+{\rm com}(\nabla_vu,\nabla_v\sI b_n, b_n)+b_n\circ\nabla_v \sI f_n.
\end{align*}
By \eqref{Lim0}, \eqref{Lim1} and Lemmas \ref{lem:para} and \ref{lem:com2}, it is easy to see that each term of the above RHS converges to
the one in \eqref{FQ2} in ${\mC^{1-2\alpha}_{T,a}(\rho)}$ for some $\rho\in\sP_{\rm w}$.
Thus,
\begin{align}\label{Lim00}
\lim_{n\to\infty}\|b_n\circ\nabla_v u_n-b\circ \nabla_v u\|_{\mC^{1-2\alpha}_{T,a}(\rho)}=0.
\end{align}
Since $-\alpha<1-2\alpha$, combining \eqref{G1}, \eqref{G2} and \eqref{Lim00},  we obtain \eqref{Lim01}.
\medskip\\
{\it Proof of  (ii): }
{In this step we first use the chain rule for approximations and then take the limit.}
Since $\psi\phi=\phi$, by the chain rule we have
$$
\bar b_n\cdot\nabla_v\bar u_n=(b_n\psi)\cdot\nabla_v(u_n \phi)=(b_n\cdot\nabla_v u_n)\phi+(b_n\cdot\nabla_v \phi) u_n.
$$
Hence, by Bony's decomposition,
$$
\bar b_n\circ\nabla_v\bar u_n=(b_n\cdot\nabla_v u_n)\phi+(b_n\cdot\nabla_v \phi) u_n-\nabla\bar u_n\prec \bar b_n-\nabla\bar u_n\succ \bar b_n.
$$
Since $\phi,\psi\in C^\infty_c(\mR^{2d})$, by \eqref{Lim0} and \eqref{G3}, we have
$$
\lim_{n\to\infty}\|(b_n\cdot\nabla_v \phi) u_n-(b\cdot\nabla_v \phi) u\|_{\mC_{T,a}^{1-2\alpha}}=0,
$$
and by Lemma \ref{lem:para},
$$
\lim_{n\to\infty}\|\bar b_n\prec\nabla_v\bar u_n-\bar b\prec\nabla_v\bar u\|_{\mC_{T,a}^{1-2\alpha}}=0
$$
$$
\lim_{n\to\infty}\|\bar b_n\succ\nabla_v\bar u_n-\bar b\succ\nabla_v\bar u\|_{\mC_{T,a}^{-\alpha}}=0,
$$
which together with \eqref{Lim01} and \eqref{NP1} yields  \eqref{ES9}.
{On the other hand, we use regularity of $\bar b_n\circ\nabla_v\bar u_n$ to improve the regularity. }
Note that
\begin{align*}
\bar b_n\circ\nabla_v \bar u_n&=(b_n\psi)\circ(\nabla_v u_n \phi)+(b_n\psi)\circ(\nabla_v\phi u_n)\\
&=[(\nabla_v u_n \phi)\circ,\psi]b_n+\psi[b_n\circ,\phi]\nabla_v u_n\\
&\quad+\psi\phi (b_n\circ\nabla_v u_n)+(b_n\psi)\circ(\nabla_v \phi u_n).
\end{align*}
Moreover, by \eqref{ES9}, \eqref{Le12} and \eqref{ES9}, one sees that
\begin{align}\label{NP2}
\|\bar b\circledcirc\nabla_v \bar u\|_{\mC_{T,a}^{1-2\alpha}}\leq\sup_n\|\bar b_n\circ\nabla_v \bar u_n\|_{\mC_{T,a}^{1-2\alpha}}<\infty.
\end{align}
{\it Proof of  (iii): } By the chain rule, we have in the distributional sense
$$
\sL\bar u=\sL(u\phi)=\phi\sL u-u\Delta_v\phi-2\nabla_v\phi\cdot\nabla_v u-(v\cdot\nabla_x\phi) u.
$$
Taking the inverse $\sL^{-1}=\sI$, and by \eqref{AN3} and definition \eqref{NP1}, we get
\begin{align*}
\bar u&=\sI((b\star\nabla_v u+f)\phi-u\Delta_v \phi-2\nabla_v\phi\cdot\nabla_v u-(v\cdot\nabla_x\phi) u)
\\&=\sI(\bar b\circledcirc\nabla_v \bar u+\nabla_v \bar u\prec \bar b+\nabla_v \bar u\succ \bar b+g),
\end{align*}
which, combining with definition \eqref{An22}, yields \eqref{NP55}.
Moreover, since by \eqref{GZ1} and \eqref{NP2},
$$
\nabla_v\bar u\succ \bar b+\bar b\circledcirc\nabla_v \bar u\in\mC^{1-2\alpha}_{T,a},
$$
by \eqref{EG1} and \eqref{GA44},  we clearly have
\begin{align}\label{NP5}
\bar u^\sharp\in \mC_{T,a}^{3-2\alpha}.
\end{align}
{\it Proof of  (iv): }
To show \eqref{EQU},
{we first find a suitable approximation for $\bar b\circ\nabla_v \bar u$. }
Let
$$
g_n:=f_n\phi- u\Delta_v\phi-2\nabla_v\phi\cdot\nabla_v u{-(b_n\cdot \nabla_v\phi)u-(v\cdot \nabla_x\phi)u}.
$$
By Lemmas \ref{Le317} and \ref{le:loc}, one sees that $(b_n\psi,g_n)$ is the approximation sequence of $(\bar b,g)$
and $g_n\rightarrow g$ in $\mC^{-\alpha}_{T,a}$. Noting that
\begin{align*}
&\bar b_n\circ\nabla_v(\bar u^\sharp+\nabla_v \bar u\prec  \sI \bar b_n+\sI g_n)\\
&\quad=\bar b_n\circ\nabla_v \bar u^\sharp+\bar b_n \circ(\nabla_v^2 \bar u\prec \sI\bar b_n)+(\bar b_n\circ\nabla_v \sI\bar b_n)\cdot \nabla_v\bar u\\
&\quad\quad+\mathrm{com}(\nabla_v\bar u, \nabla_v \sI\bar b_n,\bar b_n)+\bar b_n\circ\nabla_v\sI g_n,
\end{align*}
by \eqref{NP5},  \eqref{Lim0}, \eqref{Lim1},
Lemmas \ref{lem:para}, \ref{lem:com2} and some tedious calculations, we have
\begin{align}\label{ES10}
\lim_{n\to\infty}\bar b_n\circ\nabla_v(\bar u^\sharp+\nabla_v \bar u\prec  \sI \bar b_n+\sI g_n)=\bar b\circ\nabla_v \bar u
\mbox{ in $\mC^{1-2\alpha}_{T,a}$}.
\end{align}
{Here we use the decomposition in Lemma \ref{le:loc}
to deduce the convergence of $\bar b_n\circ\nabla_v \sI\bar b_n$ to $\bar b\circ\nabla_v \sI\bar b$.}
Hence, 
 by \eqref{ES9} and \eqref{ES10}, it remains  to prove that in suitable space,
\begin{align}\label{Lim9}
\lim_{n\to\infty}\bar b_n\circ\nabla_v(\bar u_n-\bar u^\sharp-\nabla_v \bar u\prec  \sI\bar b_n-\sI g_n)=:\lim_{n\to\infty}\Lambda_n=0.
\end{align}
Note that by \eqref{An22},
$$
\bar u^\sharp=\bar u-\nabla_v \bar u\prec  \sI \bar b-\sI g=\phi\(u^\sharp+(\nabla_v u\prec  \sI  b)+\sI  f\)-\nabla_v \bar u\prec  \sI  \bar b-\sI g,
$$
which together with \eqref{AG009} yields
\begin{align*}
\Lambda_n=\bar b_n\circ\nabla_v\big((\nabla_v u\prec  \sI  B_n)\phi-\nabla_v \bar u\prec  \sI  (B_n\psi)+\phi\sI  F_n-\sI  G_n\big),
\end{align*}
where
$$
B_n:=b_n-b,\ F_n:=f_n-f,\ G_n:=g_n-g.
$$
By commutator estimates (see Lemmas \ref{lem:para} and \ref{lem:com2}) and \eqref{Lim0}, \eqref{Lim1}, it is easy to see that
$$
\lim_{n\to\infty}\(\bar b_n\circ\nabla_v((\nabla_v u\prec  \sI  B_n)\phi)-\phi\nabla_v u(\bar b_n\circ\nabla_v\sI B_n)\)=0
$$
and
$$
\lim_{n\to\infty}\(\bar b_n\circ\nabla_v(\nabla_v \bar u\prec  \sI  (B_n\psi))-\psi\nabla_v\bar u(\bar b_n\circ\nabla_v\sI B_n)\)=0.
$$
Moreover, noting that
\begin{align*}
\phi\sI  F_n-\sI  G_n=-[\sI,\phi]F_n+\sI (B_n\cdot\nabla_v\phi u),
\end{align*}
by Lemma \ref{Le231} and Lemma \ref{commutator1}, we also have
$$
\lim_{n\to\infty}\(\bar b_n\circ\nabla_v(\phi\sI  F_n-\sI  G_n)-(\nabla_v\phi u)(\bar b_n\circ\nabla_v \sI  B_n)\)=0.
$$
Finally, since $\psi\nabla_v\bar u=\nabla_v(\phi u)$, we have
$$
(\psi\nabla_v \bar u-\phi\nabla_v u-\nabla_v\phi u)(\bar b_n\circ\nabla_v \sI  B_n)\equiv 0,
$$
which together with the above three limits yields \eqref{Lim9}. The proof is complete.
\end{proof}
\br\rm
The above result clearly holds for classical solutions by the chain rule.
However, for the paracontrolled solution we cannot directly apply the chain rule since the paracontrolled solution is in the renormalized sense,
i.e., $b\cdot\nabla_v\sI b$  and $b\cdot\nabla_v\sI f$ are understood in the approximation sense.
Therefore, we have to first construct suitable smooth approximations for the solution
so that we can use the chain rule. In the last step, an obvious difficulty is that
although
$$
\lim_{n\to\infty}\|b_n-b\|_{\mC^{-\alpha}_{T,a}(\rho)}=0, \ \
\lim_{n\to\infty}\|b_n\circ\nabla_v \sI  b_n-b\circ\nabla_v \sI  b\|_{\mC^{1-2\alpha}_{T,a}(\rho)}=0,
$$
it does not imply that
$$
\lim_{n\to\infty}b_n\circ\nabla_v \sI  (b_n-b)=0\mbox{ in any space}.
$$
\er
\subsection{Well-posedness for (\ref{PDE7})}
\label{sub:Schauder}
First of all we have the following well-posedness result for PDE \eqref{PDE7} in unweighted kinetic H\"older spaces.
Since by Lemmas \ref{Lem35}, \ref{Le11}, \ref{commutator1}
and Theorem \ref{Le32}, its proofs are essentially the same as in \cite[Section 3.2]{ZZZ20}.  The only difference is that  we do not introduce the notion $\Prec$ and cannot obtain time regularity of $u^\sharp$ which is used to deduce the convergence of $u^\sharp$. We can use similar argument as in the proof of Theorem \ref{Th33} below to obtain convergence of $u^\sharp$. Thus we omit the proof of the following theorem.
We would like to emphasize that the role of introducing $\lambda$ is only used in the proof of the following theorem.   We also mention that the maximal principle is easy for the \eqref{PDE7} when $b,f\in L^\infty_T\bC^\infty_b$, since the fundamental solution exists in this case (see \cite{DM10}).

\bt\label{Th12}
Let $T>0$ and $\varphi=0$. For any $(b,f)\in\mB^\alpha_T$, 
there is a unique paracontrolled solution $u$ to PDE \eqref{PDE7}
in the sense of Definition \ref{def:para1}. Moreover,
there are $q>1$ large enough only depending on $\alpha$ and $c_1,c_2>0$ such that
$$
\|u\|_{\mL^\infty_T}\leq c_1(\ell^{b}_T)^{\frac{5}{2-3\alpha}}\mA^{b,f}_{T,q},\ \
\|u\|_{\mS^{2-\alpha}_{T,a}}\leq c_2(\ell^{b}_T)^{\frac{9}{2-3\alpha}}\mA^{b,f}_{T,\infty}.
$$
\et

Now we give the main result of this section.
\bt\label{Th33}
Let $\alpha\in(\frac{1}{2},\frac{2}{3})$ and $\vartheta:=\frac{9}{2-3\alpha}$. Let $\kappa_1>0$ and $\kappa_2\in\mR$ with
$$
(2\vartheta+2)\kappa_1\leq 1,\ \ \kappa_3:=(2\vartheta+1)\kappa_1+\kappa_2.
$$
With notations in \eqref{ND2}, let
$$
\rho_i:=\varrho^{\kappa_i}\in\sP_{\rm w},\ i=1,2,3.
$$
Under \eqref{DW9}, for any $T>0$ and $\varphi\in\bC^{\gamma}_{a}(\rho_2/\rho_1)$, where $\gamma>1+\alpha$,
there is a unique paracontrolled solution $u\in \mS^{2-\alpha}_{T,a}(\rho_3)$
to PDE \eqref{PDE7} in the sense of Definition \ref{def:para1} so that
\begin{align}\label{MN1}
\|u\|_{\mS^{2-\alpha}_{T,a}(\rho_3)}\lesssim_C\|\varphi\|_{\bC^\gamma_{a}(\rho_2/\rho_1)}+\mA^{b,f}_{T,\infty}(\rho_1,\rho_2),
\end{align}
where $C=C(T,d,\alpha,\kappa_i,\ell^{b}_T(\rho_1))>0$. Moreover, let $(b_n, f_n)\in L^\infty_T C_b^\infty\times L^\infty_TC_b^\infty$
be the approximation in
Definition \ref{Def216}, and $\varphi_n\in C^\infty_b$ with
$$
\sup_n\|\varphi_n\|_{\bC^{\gamma}_{a}(\rho_2/\rho_1)}<\infty,
$$
and $\varphi_n$ converges to $\varphi$  in $\mR^{2d}$ locally uniformly.
Let $u_n$ be the classical solution of PDE \eqref{PDE7} corresponding to $(b_n,f_n)$ and $\varphi_n$.
Then for any $\beta>\alpha$ and $\rho_4\in\sP_{\rm w}$
with $\lim_{z\to\infty}(\rho_4/\rho_3)(z)=0$, we have
\begin{align}\label{MN2}
\lim_{n\to\infty}\|u_n-u\|_{\mS^{2-\beta}_{T,a}(\rho_4)}=0.
\end{align}
\et
\begin{proof}
We mainly concentrate on showing the a priori estimate \eqref{MN1} for any paracontrolled solution $u$ of PDE \eqref{PDE7}.
Without loss of generality we may assume $\lambda=0$ and $\varphi=0$ (see Remark \ref{Re42}).
We fix $0<r<\frac1{16}$.	Note that $\phi^z_{2r}=1$ on the support of $\phi^z_r$.
	For each $z\in\mR^d$,
by Proposition \ref{Pr42}, $u_z:=u\phi^z_r$ is a paracontrolled solution to the following PDE:
	$$
	\p_t u_z=\Delta_v u_z+v\cdot\nabla_xu_z+b_z\cdot\nabla_v u_z+g_z,\ \ u_z(0)=0,
	$$
	where $b_z:=b\phi^z_{2r}$ and
	$$
	g_z:=f\phi^z_r-2\nabla_v u\cdot\nabla_v\phi^z_r-(\Delta_v\phi^z_r+v\cdot \nabla_x\phi^z_r)u-b\cdot\nabla_v\phi^z_ru.
	$$
	By Theorem \ref{Th12}, there are $q>1$ large enough and two constants $c_1,c_2>0$ such that for all $z\in\mR^d$ ,
	\begin{align}\label{HG7}
	\|u_z\|_{\mS^{2-\alpha}_T}
	\leq c_1(\ell^{b_z}_T)^{\vartheta}\mA^{b_z,g_z}_{T,\infty},\ \|u_z\|_{\mL^\infty_T}\leq c_2(\ell^{b_z}_T)^{\vartheta}\mA^{b_z,g_z}_{T,q}.
	\end{align}
Below, for simplicity of notations, we drop the time variable. By the definition of $g_z$, Lemma \ref{lem:para}, \eqref{ND313} and \eqref{ND4}, we have
	\begin{align}
	\|g_z\|_{\bC_a^{-\alpha}}&\leq \|f\phi^z_r\|_{\bC_a^{-\alpha}}+2\|\nabla_v u\cdot\nabla_v\phi^z_r\|_{\bC^{-\alpha}_a}+\|b\cdot\nabla_v\phi^z_ru\|_{\bC_a^{-\alpha}}
	\no\\
&\quad+\|u(\Delta_v\phi^z_r+v\cdot\nabla_x\phi^z_r)\|_{L^\infty}
	\no\\
	&\lesssim\|f\|_{\bC_a^{-\alpha}(\rho_2)}\|\phi^z_r\|_{\bC_a^{1}(\rho_2^{-1})}+\|\nabla_v u\|_{\bC^{-\alpha}_a(\rho_3)}
	\|\nabla_v\phi^z_r\|_{\bC^1_a(\rho_3^{-1})}\no\\
	&\quad+\|b\|_{\bC_a^{-\alpha}(\rho_1)}\|u\|_{\bC_a^1(\rho_3)}\|\nabla_v\phi^z_r\|_{\bC_a^1((\rho_1\rho_3)^{-1})}\no\\
	&\quad+\| u\|_{L^\infty(\rho_3)}
	\|\Delta_v\phi^z_r+v\cdot\nabla_x\phi^z_r\|_{L^\infty(\rho_3^{-1})}
	\no\\
	&\lesssim\rho_2^{-1}(z)\|f\|_{\bC_a^{-\alpha}(\rho_2)}+(\varrho\rho^{-1}_1\rho^{-1}_3)(z)\|u\|_{\bC^1_a(\rho_3)}.\label{DB9}
	\end{align}
	Hence,
	\begin{align}
	\|g_z\|_{L^q_T\bC_a^{-\alpha}}&\lesssim
	\rho^{-1}_2(z)\|f\|_{L_T^q\bC_a^{-\alpha}(\rho_2)}+(\varrho\rho_1^{-1}\rho^{-1}_3)(z)\|u\|_{L^q_T\bC^1_a(\rho_3)}.\label{ND5}
	\end{align}
		Moreover, we have
	\begin{align*}
	\|(b_z\circ\nabla_v \sI_\lambda g_z)\|_{\bC_a^{1-2\alpha}}&\leq
	\|b_z\circ\nabla_v \sI_\lambda({f}\phi^z_r)\|_{\bC_a^{1-2\alpha}}+\|b_z\circ\nabla_v \sI_\lambda(b\cdot\nabla_v \phi^z_ru )\|_{\bC_a^{1-2\alpha}}
	\\&\quad+\|b_z\circ\nabla_v \sI_\lambda(u(\Delta_v\phi^z_r+v\cdot\nabla_x\phi^z_r)+2\nabla_v u\cdot\nabla_v\phi^z_r)\|_{L^\infty}
\\&=:I_1^z+I_2^z+I_3^z.
	\end{align*}
	For $I^z_1$, by \eqref{EK0} with $\rho_3=\rho_1^{-1}$, $\rho_4=\rho_2^{-1}$ and ${\phi}=\phi^z_r$, we have
	\begin{align*}
	I^z_1\lesssim \|\phi^z_{2r}\|_{\bC_a^1(\rho^{-1}_1)}\|\phi^z_r\|_{\bC_a^1(\rho^{-1}_2)}\mA^{b,f}_{t,\infty}(\rho_1,\rho_2)
	\lesssim (\rho^{-1}_1\rho^{-1}_2)(z)\mA^{b,f}_{t,\infty}(\rho_1,\rho_2).
	\end{align*}
	For $I^z_2$, by \eqref{EK0} with $\rho_3=\rho^{-2}_1,\rho_4\equiv1$, and $\psi=\nabla \phi^z_r u$, we have
	\begin{align*}
	I^z_2&\lesssim \|\phi^z_{2r}\|_{\bC_a^1(\rho^{-2}_1)}\|\nabla_v \phi^z_r u\|_{\mS_{t,a}^1}
	\mA^{b,b}_{t,\infty}(\rho_1,\rho_1)\lesssim (\varrho\rho^{-2}_1\rho^{-1}_3)(z)\|u\|_{\mS_{t,a}^1(\rho_3)},
	\end{align*}
	where by \eqref{STr} and \eqref{GD3}, we have
	\begin{align*}
\|\nabla_v \phi^z_r u\|_{\mS_{t,a}^1}\lesssim \varrho(z)\|\phi^z_{2r}u\|_{\mS_{t,a}^1}\lesssim (\varrho\rho^{-1}_3)(z)\|u\|_{\mS_{t,a}^1(\rho_3)}.
\end{align*}
For $I^z_3$, as in \eqref{DB9}, by \eqref{GZ2}, Lemma \ref{Le11} and  \eqref{ND4}, we have
	\begin{align*}
	I^z_3&\lesssim \|b_z\|_{\bC_a^{-\alpha}}\|\nabla_v \sI_\lambda(u(\Delta_v\phi^z_r+v\cdot\nabla_x\phi^z_r)+2\nabla_v u\cdot\nabla_v\phi^z_r)\|_{\bC_a^1}
	\\&\lesssim \rho^{-1}_1(z)\|b\|_{\bC_a^{-\alpha}(\rho_1)}\|u(\Delta_v\phi^z_r+v\cdot\nabla_x\phi^z_r)+2\nabla_v u\cdot\nabla_v\phi^z_r\|_{\mC^0_{t,a}}
	\\&\lesssim (\varrho\rho_1^{-1}\rho_3^{-1})(z)\|u\|_{\mC^1_{t,a}(\rho_3)},
	\end{align*}
where in the second step we used
\begin{align}\label{eqb}\|b_z\|_{\bC_a^{-\alpha}}\lesssim \|b\|_{\bC_a^{-\alpha}(\rho_1)}\|\phi^z_r\|_{\bC_a^{1}(\rho^{-1}_1)}\lesssim \rho_1^{-1}(z)\|b\|_{\bC_a^{-\alpha}(\rho_1)}.
\end{align}
	Combining the above calculations and since $\rho_1$ is bounded, we get for any $t\in[0,T]$,
	$$
	\|(b_z\circ\nabla_v \sI_\lambda g_z)(t)\|_{\bC_a^{1-2\alpha}}\leq(\rho^{-1}_1\rho^{-1}_2)(z)\mA^{b,f}_{t,\infty}(\rho_1,\rho_2)
	+(\varrho\rho_1^{-2}\rho_3^{-1})(z)\|u\|_{\mS^1_{t,a}(\rho_3)}.
	$$
Now by the definition of $\mA^{b_z,g_z}_{T,q}$, \eqref{ND5}, \eqref{eqb} and the calculations above, we get
	\begin{align}
	\mA^{b_z,g_z}_{T,q}
	&=\sup_\lambda\|b_z\circ \nabla_v \sI_\lambda g_z\|_{L^q_T\bC_a^{1-2\alpha}}+(\|b_z\|_{\mC_{T,a}^{-\alpha}}{+1})\|g_z\|_{L^q_T\bC_a^{-\alpha}}\no
	\\&\lesssim(\rho^{-1}_1\rho^{-1}_2)(z)\mA^{b,f}_{T,{\infty}}(\rho_1,\rho_2)
	+(\varrho\rho_1^{-2}\rho_3^{-1})(z)\left(\int^T_0\|u\|_{\mS^1_{t,a}(\rho_3)}^q\dif t\right)^{1/q}.\label{NDC1}
	\end{align}
	On the other hand, by \eqref{ND313}  and \eqref{EK0} with  $\rho_3=\rho_1^{-1}$, $\rho_4={\rho_1^{-1}}$ and ${\phi}=\psi=\phi^z_r$, we have
	$$
	\|b_z\circ\nabla\sI_\lambda b_z\|_{\mC^{1-2\alpha}_{T,a}}\lesssim\rho_1^{-2}(z)
	(\|b\circ\nabla\sI_\lambda b\|_{\mC^{1-2\alpha}_{T,a}(\rho_1^2)}{+\|b\|^2_{\mC^{-\alpha}_{T,a}(\rho_1)}}).
	$$
	Hence, by \eqref{eqb}
	 \begin{align*}
	 \ell_T^{b_z}=\mA^{b_z,b_z}_{T,\infty}(1,1){+1} \lesssim\rho_1^{-2}(z)\ell_T^b(\rho_1),
	 \end{align*}
	 	Then, by \eqref{HG7} and \eqref{NDC1} with $q=\infty$, we have
	\begin{align*}
\|u_z\|_{\mS^{2-\alpha}_{T,a}}&\lesssim \rho_1^{-2\vartheta}(z)\[(\rho^{-1}_1\rho^{-1}_2)(z)\mA^{b,f}_{T,\infty}(\rho_1,\rho_2)
	+(\varrho\rho_1^{-2}\rho_3^{-1})(z)\|u\|_{\mS^1_{T,a}(\rho_3)}\]\no\\
&=(\rho^{-1-2\vartheta}_1\rho^{-1}_2)(z)\mA^{b,f}_{T,\infty}(\rho_1,\rho_2)
	+(\varrho\rho_1^{-2-2\vartheta}\rho_3^{-1})(z)\|u\|_{\mS^1_{T,a}(\rho_3)},
\end{align*}
and
$$
\|u_z\|_{\mL_T^\infty}\lesssim
(\rho^{-1-2\vartheta}_1\rho^{-1}_2)(z)\mA^{b,f}_{T,{\infty}}(\rho_1,\rho_2)
	+(\varrho\rho_1^{-2-2\vartheta}\rho_3^{-1})(z)\left(\int^T_0\|u\|_{\mS^1_{t,a}(\rho_3)}^q\dif t\right)^{1/q}.
$$
 From these two estimates, and noting that
 $$
 \rho_3=\rho^{1+2\vartheta}_1\rho_2,\ \ \varrho\rho_1^{-2-2\vartheta}\leq 1,
 $$
 by Lemmas \ref{cha0} and \ref{cha}, we get
\begin{align}\label{ND8}
\|u\|_{\mS^{2-\alpha}_{T,a}(\rho_3)}\lesssim \mA^{b,f}_{T,\infty}(\rho_1,\rho_2)+\|u\|_{\mS^1_{T,a}(\rho_3)}
\end{align}
and
\begin{align}\label{ND9}
\|u\|_{\mL^\infty_T(\rho_3)}\lesssim\mA^{b,f}_{T,{\infty}}(\rho_1,\rho_2)+\(\int_0^T\|u\|_{\mS^1_{t,a}(\rho_3)}^q\dif t\)^{1/q}.
\end{align}
Note that by \eqref{Embq0} and Definition \ref{kinetics},
$$
\|u\|_{\mS^{1}_{T,a}(\rho_3)}\lesssim \|u\|^{1/(2-\alpha)}_{\mS^{2-\alpha}_{T,a}(\rho_3)}\|u\|^{(1-\alpha)/(2-\alpha)}_{\mL^\infty_T(\rho_3)},
$$
which by Young's inequality implies that for any $\eps>0$, there is a constant $C_\eps>0$ such that
	$$
	\|u\|_{\mS^1_{t,a}(\rho_3)}
	\leq \eps\|u\|_{\mS^{2-\alpha}_t(\rho_3)}+C_\eps\|u\|_{\mL^\infty_t(\rho_3)}.
	$$
Substituting this into \eqref{ND8} and choosing $\eps$ small enough, we get
$$
\|u\|_{\mS^{2-\alpha}_{t,a}(\rho_3)}\lesssim \mA^{b,f}_{T,\infty}(\rho_1,\rho_2)+\|u\|_{\mL^\infty_t(\rho_3)},
$$
which together with \eqref{ND9} and by Gronwall's inequality, we obtain \eqref{MN1}.

{\bf (Uniqueness)} Let $u_1, u_2$ be two paracontrolled solutions of  PDE \eqref{PDE7}.  By definition, it is easy to see that $u=u_1-u_2$ is still a paracontrolled solution of \eqref{PDE7} with $\varphi=f\equiv0$. Thus by \eqref{MN1}, we immediately have
$u=0$.

{\bf (Existence)}
Let $(b_n, f_n)\in L^\infty_T C_b^\infty\times L^\infty_TC_b^\infty$ be the approximation in
Definition \ref{Def216}, and $u_n$ be the corresponding solution of PDE \eqref{PDE7}.
By the priori estimate \eqref{MN1}, \eqref{CC9} and \eqref{Lim010}, we have the following uniform estimate:
\begin{align}\label{AX2}
\sup_n\Big(\|u_n\|_{\mS^{2-\alpha}_{T,a}(\rho_3)}+\|u^\sharp_n\|_{\mC_{T,a}^{3-2\alpha}(\rho_4)}\Big)<\infty.
\end{align}
By Lemma \ref{CptE}, for any $\beta>\alpha$ and $\rho_5\in\sP_{\rm w}$
with $\lim_{z\to\infty}(\rho_5/\rho_3)(z)=0$, there are $u\in\mS^{2-\alpha}_{T,a}(\rho_3)$ and a subsequence $n_k$ such that
$$
\lim_{k\to\infty}\|u_{n_k}-u\|_{\mS^{2-\beta}_{T,a}(\rho_5)}=0.
$$
Moreover, let $u^\sharp:=u-\nabla_v u\prec  \sI_\lambda b-\sI_\lambda f$. By the above limit, \eqref{GZ1} and \eqref{SchS},
it is easy to see that for some $\rho_6\in\sP_{\rm w}$,
\begin{align*}
\lim_{k\to\infty}\|u^\sharp_{n_k}-u^\sharp\|_{\mL^\infty_T(\rho_6)}=0,
\end{align*}
which, together with \eqref{AX2}, and by Fatou's lemma and the interpolation inequality \eqref{Embq0}, implies that
$u^\sharp\in\mC^{3-2\alpha}_{T,a}(\rho_4)$ and for any $\beta>\alpha$,
$$
\lim_{k\to\infty}\|u^\sharp_{n_k}-u^\sharp\|_{\mC^{3-2\beta}_{T,a}(\rho_6\rho_4)}=0.
$$
By a standard limit procedure, one finds that $u$ is a paracontrolled solution in the sense of Definition \ref{def:para1} (see \cite{GIP15}).
Finally, by the uniqueness of paracontrolled solutions, the full limit \eqref{MN2} holds.
\end{proof}

\section{Well-posedness of singular mean field equations}\label{nonlinear}
{In this section we study the nonlinear singular kinetic equations.}
Throughout this section we fix $T>0$, $\alpha\in(\frac12,\frac23)$, $\vartheta:=\frac9{2-3\alpha}$ and
\begin{align}\label{Con1}
 \kappa_0<0,\ \ 0\leq \kappa_1\leq 1/(2\vartheta+2),
\end{align}
and let
$$
\kappa_2:=\kappa_1,\  \kappa_3:=(2\vartheta+2)\kappa_1,\ \rho_i:=\varrho^{\kappa_i},\ \ i=0,1,2,3,
$$
where $\varrho$ is given in \eqref{ND2}.
Consider the following nonlinear kinetic equation with distributional drift
\begin{align}\label{SMFL}
\p_t\rrho=\Delta_v\rrho-v\cdot\nabla_x\rrho-\WW\cdot\nabla_v\rrho-K*\mv\cdot\nabla_v\rrho,\quad \rrho(0)=\varphi,
\end{align}
where $\WW(t,x,v)$ satisfies that
\begin{align}\label{DW19}
\WW\in  \mB^\alpha_T(\rho_1) \mbox{ has the approximation sequence $\WW_n$ with $\div_v \WW_n\equiv0$.}
\end{align}
Here $\mv(t,x):=\int_{\mR^{d}}u(t,x,v)\dif v$ stands for the mass,
$K(x):\mR^d\to\mR^d$ is a kernel function and satisfies that
\begin{align}\label{KER}
K\in\cup_{\beta>\alpha-1}\bC^{\beta/3}.
\end{align}
\br\rm\label{Kpro}
\begin{enumerate}[(i)]
\item For $\widetilde K(x,v)=K(x)$, it is easy to see that
\begin{align*}
\widetilde K\in\bC^\beta_a\iff K\in\bC^{\beta/3},\quad\forall \beta\in\mR.
\end{align*}
Moreover, for $K(x)=|x|^{-r}$ with $r<(1-\alpha)/3$, \eqref{KER} holds.

\item Since $\div_v W\equiv 0$ and $K$ does not depend on $v$, one can write \eqref{SMFL} as the following divergence form:
\begin{align}\label{SMFL0}
\p_t\rrho=\Delta_v\rrho-v\cdot\nabla_x\rrho-\div_v((\WW+K*\mv) \rrho),\quad \rrho(0)=\varphi.
\end{align}
In particular, when $W$ and $K$ are smooth, if $\varphi$ is a probability density function, then so is the solution $u$. 
\end{enumerate}
\er
To use the framework of the above sections we define the solution to \eqref{SMFL} by the transform introduced in the introduction. Now we define this transform for $f\in \sS'(\mR^{2d}),\varphi\in \sS(\mR^{2d})$
$$\tau f(\varphi):=f(\tau \varphi)\qquad \tau \varphi(x,v):=\varphi(x,-v).$$
It is easy to see this transform does not change Besov norm.
\bd
We call $\rrho\in\mS^{2-\alpha}_{T,a}(\rho_3)$ a  probability density paracontrolled solution to PDE \eqref{SMFL} if
$\tau u$ is a paracontrolled solution to PDE \eqref{PDE7} with $\lambda=0$ and $b=\tau\WW+K*\mv$ and initial value $\tau\varphi$
$$
u\geq 0, \ \ \int_{\mR^{2d}}u(t,z)\dif z=1,\  \ t\in[0,T].
$$
\ed
\br\label{Re52}\rm (i) This definition should be equivalent to the definition using the semigroup associated with $\Delta_v-v\cdot\nabla_x$.

(ii) Let $u$ be a probability density paracontrolled solution to PDE \eqref{SMFL}. Under \eqref{DW19} and \eqref{KER}, by Lemma \ref{Le317},
it is easy to see that $b=\tau\WW+K*\mv\in \mB^\alpha_T(\rho_1)$, whose approximation sequence can be taken as
$$
b_n=\tau\WW_n+K_n*\mv\ \mbox{ with } \div_v b_n\equiv0,
$$
where $K_n=K*\phi_n$ with $\phi_n$ being the usual mollifier.
\er

{For a density solution the nonlinear term can be bounded easily. To prove the existence of solutions we use smooth approximation and need to prove the convergence not only in the kinetic H\"{o}lder space but also in $L^1$ space since the nonlinear term contains a nonlocal interaction.  The proof of the uniqueness part is more involved. To deal with the nonlinear term, we have to bound the difference of solutions in $L^1$ space which requires an uniform $L^2_tL^1$ bound of the gradient of the solutions. To this end we use an entropy method and introduce the following entropy.}
For a probability density function $f$,  one says that $f$ has a finite  entropy if
$$
H(f):=\int_{\mR^{2d}} f(z)\ln f(z)\dif z\in(-\infty,\infty).
$$
The main result of this section is the following theorem.
\bt\label{thm72}
Suppose that \eqref{Con1}, \eqref{DW19} and \eqref{KER} hold. Let $\gamma>1+\alpha$.

\noindent {\bf (Existence)} For any probability density function $\varphi\in L^1(\rho_0)\cap\bC^\gamma_a$,
there exists at least a probability density paracontrolled solution $u\in\mS^{2-\alpha}_{T,a}(\rho_3)$ to PDE \eqref{SMFL}.
Moreover, there is a constant $C>0$ such that for all $t\in[0,T]$,
\begin{align}\label{Mom}
\|\rrho(t)\|_{L^1(\rho_0)}\leq C\|\varphi\|_{L^1(\rho_0)}
\end{align}
and if $|H(\varphi)|<\infty$, then it holds that
\begin{align}\label{Rev4}
H(\rrho(t))+\|\nabla_v\rrho\|^2_{L^2_tL^1_z}\leq H(\varphi),
\end{align}
and
\begin{align}\label{Rev5}
|H(\rrho(t))|+\|\nabla_v\rrho\|^2_{L^2_tL^1_z}\leq H(\varphi)+C(\|\varphi\|_{L^1(\rho_0)}+1).
\end{align}
\noindent {\bf (Stability)} If in addition that $K$ is bounded, then for any $\varphi_1,\varphi_2\in L^1(\rho_0)\cap\bC^\gamma_a$ with
$H(\varphi_1)<\infty$,
and any probability density paracontrolled solutions $u_1$ and $u_2$ with initial values $\varphi_1$ and $\varphi_2$, respectively,
there is a constant $C>0$ only depending on $\|K\|_{L^\infty}$, $\|\varphi_1\|_{L^1(\rho_0)}$, $H(\varphi_1)$ and $\|\e^{-\rho_0}\|_{L^1}$
such that for all $t\in[0,T]$,
\begin{align}\label{SX89}
\|u_1(t)-u_2(t)\|_{L^1}\leq \e^{Ct}\|\varphi_1-\varphi_2\|_{L^1}.
\end{align}
\et
\br\rm
$\varphi\in L^1(\rho_0)$ is a moment requirement, i.e.,
$$
\int_{\mR^{2d}}|z|^{|\kappa_0|}\varphi(z)\dif z<\infty.
$$
This is a common assumption in the entropy method (see \cite{JW16}),
which can be seen from the following Lemma \ref{ulnu}.
\er

We need the following elementary lemma.
\bl\label{ulnu}
It holds that for any measurable $\phi,f\geq 0$, $\delta\in[0,1)$ and $\rho\in \sP_{\rm w}$,
\begin{align*}
\int\phi|f\ln (f+\delta)|\le \int\phi f\ln (f+\delta)+2\(\int \phi\rho f+\int\phi\e^{-\rho}\).
\end{align*}
\el
\begin{proof}
By Young's inequality, we have
\begin{align*}
-r\ln(r+\delta)\le -r\ln r\le ar+\e^{-a},\quad \forall r\in[0,1],\ \ a\geq 0.
\end{align*}
Hence,
$$
|r\ln(r+\delta)|=r\ln(r+\delta)-2r\ln(r+\delta)\1_{\{0<r\leq 1-\delta\}}\leq r\ln(r+\delta)+2(ar+\e^{-a}).
$$
The desired estimate follows by taking $a=\rho$.
\end{proof}


We recall the following result (cf. \cite{RXZ21}).
\bl\label{Le52}
Let $b\in L^\infty_TC^\infty_b(\mR^{2d})$ and
let $Z^{z_0}_t=(X_t, V_t)$ be the unique solution of the following SDE:
\begin{align}\label{SDE0}
\dif X_t=V_t\dif t,\ \dif V_t=\sqrt{2}\dif B_t+b(t,X_t,V_t)\dif t,\ \ (X_0,V_0)=z_0\in\mR^{2d}.
\end{align}
Then for any initial probability measure $\mu_0$,
$$
\mu(t,\dif z)=\int_{\mR^{2d}}\bP(Z^{z_0}_t\in\dif z)\mu_0(\dif z_0)
$$
is the unique solution to the following Fokker-Planck equation in the distributional sense:
$$
\p_t \mu=\Delta_v\mu{-}v\cdot\nabla_x\mu{-}\div_v(b\mu),\ \ \mu(0)=\mu_0.
$$

\el
{Now we first derive the following a priori moment and entropy estimates. The proof is divided into three steps. First for given solution $u$ we can find a linear approximation such that Theorem \ref{Th33} can be applied. Second we prove \eqref{Mom} by a probabilistic  method. Finally we use entropy method to prove \eqref{Rev4} and \eqref{Rev5}.}
\bl\label{lem:priori}
Under \eqref{DW19},
let $u\in\mS^{2-\alpha}_{T,a}(\rho_3)$ be a probability density  paracontrolled solution of \eqref{SMFL} with initial value $\varphi\in L^1(\rho_0)\cap \bC_a^\gamma$. Then \eqref{Mom} holds. Moreover, if ${H(\varphi)}<\infty$ then \eqref{Rev4} and \eqref{Rev5} hold.
\el
\begin{proof}
	\newcounter{UnifN} 
	\refstepcounter{UnifN} 
		
	
	({\sc Step 1}) 
 Let $b_n\in L^\infty_TC^\infty_b(\mR^{2d})$ be the approximation sequence as in Remark \ref{Re52} and
$\varphi_n=\varphi*\phi_n$ with $\phi_n$ being the usual mollifier.
Since $b_n\in L^\infty_T C^\infty_b(\mR^{2d})$, it is well known that there is a unique probability density solution $\rrho_n\in L^\infty_TC^\infty_b(\mR^{2d})$ to
the following approximation Fokker-Planck equation: 
\begin{align}\label{PDE8}
\p_t\rrho_n=\Delta_v\rrho_n-v\cdot\nabla_x\rrho-\tau b_n\cdot\nabla_v\rrho_n
=\Delta_v\rrho_n-v\cdot\nabla_x\rrho_n-\div_v(\tau b_n\rrho_n),
\end{align}
with $\rrho_n(0)=\varphi_n$. It is easy to see that $\tau\rrho_n$ satisfies the following equation:
\begin{align*}
\p_t\tau\rrho_n=\Delta_v\tau\rrho_n+v\cdot\nabla_x\tau\rrho_n+b_n\cdot\nabla_v\tau\rrho_n.
\end{align*}
By \eqref{MN2} and definition of solutions, we have for some $\rho\in\sP_{\rm w}$ and $\beta\in(\alpha,2)$,
\begin{align*}
\lim_{n\to\infty}\|\tau u_{n}-\tau u\|_{\mS^{2-\beta}_{T,a}(\rho)}=0,
\end{align*}
which implies that
\begin{align}\label{uestimate}
\lim_{n\to\infty}\| u_{n}- u\|_{\mS^{2-\beta}_{T,a}(\rho)}=0.
\end{align}
To show \eqref{Mom}, \eqref{Rev4} and \eqref{Rev5},  it suffices to show that for some $C>0$ independent of $n$,
\begin{align}\label{SX2}
\|u_n(t)\|_{L^1(\rho_0)}\lesssim_C\|\varphi_n\|_{L^1(\rho_0)}\lesssim_C\|\varphi\|_{L^1(\rho_0)},
\end{align}
and if ${H(\varphi)}<\infty$, then
\begin{align}\label{Rev44}
H(\rrho_n(t))+\|\nabla_v\rrho_n\|^2_{L^2_tL^1}\leq H(\varphi_n).
\end{align}
Indeed, it is easy to see that \eqref{SX2} implies \eqref{Mom} by Fatou's lemma. Now we prove how to derive \eqref{Rev4} and \eqref{Rev5} from \eqref{Rev44} and \eqref{SX2}.
First,  since $r\mapsto r\log r$ is convex on $[0,\infty)$ and $H(\varphi)<\infty$, by Jensen's inequality, we have
\begin{align}\label{Rev43}
H(\varphi_n)=H(\varphi*\phi_n)\leq H(\varphi),
\end{align}
and by the lower semi-continuity of $u\mapsto \|\nabla_v\rrho\|_{L^2_tL^1}$,
\begin{align}\label{Rev42}
\|\nabla_v\rrho\|_{L^2_tL^1}\leq\varliminf_{n\to\infty}\|\nabla_v\rrho_n\|_{L^2_tL^1}.
\end{align}
On the other hand, let $\kappa_0<\kappa<0$ and $\rho:=\varrho^\kappa$. Recalling \eqref{ND2} and $\rho_0=\varrho^{\kappa_0}$, for any $R>0$, we have
\begin{align*}
\|\rrho_n(t)-\rrho(t)\|_{L^1(\rho)}
&\le \int |\rrho_n(t,z)-\rrho(t,z)|\cdot\1_{\{|z|_a\leq R\}}\cdot\rho(z)\dif z\\
&\quad+\int |\rrho_n(t,z)-\rrho(t,z)|\cdot\1_{\{|z|_a>R\}}\cdot\rho(z)\dif z\\
&\le \int |\rrho_n(t,z)-\rrho(t,z)|\cdot\1_{\{|z|_a\leq R\}}\cdot\rho(z)\dif z\\
&\quad+C\sup_n\|\rrho_n(t)\|_{L^1(\rho_0)}/R^{\kappa-\kappa_0},
\end{align*}
which implies  by first letting $n\to\infty$ and then $R\to\infty$,
\begin{align}\label{Lim8}
\lim_{n\to\infty}\|\rrho_n-\rrho\|_{L^\infty_TL^1(\rho)}=0.
\end{align}
Now we define the relative entropy for nonnegative measurable function $f$,
\begin{align}\label{Ent}
H_{\rho}(f):=\int f \ln (f \e^{\rho})=H(f)+\|f\|_{L^1(\rho)}.
\end{align}
Since $r(\ln r-1)\geq -1$ for $r\geq 0$, we have
$$
\inf_n u_n(t)\big(\ln (u_n(t)\e^{\rho})-1\big)\geq -\e^{-\rho}\in L^1,
$$
which by Fatou's lemma implies that
\begin{align*}
H_{\rho}(u(t))-1&\leq\varliminf_{n\to\infty}\int u_n(t)\e^{\rho} \big(\ln (u_n(t)\e^{\rho})-1\big)\e^{-\rho}
=\varliminf_{n\to\infty}H_{\rho}(u_n(t))-1.
\end{align*}
This together with \eqref{Lim8} and \eqref{Ent} yields
$$
H(\rrho(t))\leq\varliminf_{n\to\infty}H(\rrho_n(t)).
$$
Combining this with \eqref{Rev44}-\eqref{Rev42}, we obtain \eqref{Rev4}.
{Moreover, by \eqref{Mom} and \eqref{Rev4} and Lemma \ref{ulnu},  \eqref{Rev5} follows.}

({\sc Step 2})
In this step we show \eqref{SX2} { by showing a moment estimate of solution to \eqref{SDE0} which is achieved by establishing a Krylov's type of estimate for the singular drift term.}  For simplicity, we drop the subscripts $n$ below. By Lemma \ref{Le52}
one has
$$
\|u(t)\|_{L^1(\rho_0)}=\int_{\mR^{2d}}\bE\rho_0(Z^{z_0}_t)\varphi(z_0)\dif z_0,
$$
where  $Z^{z_0}_t=(X_t,V_t)$ is the unique solution to SDE \eqref{SDE0} with $b=\tau b_n$. 
Hence, to show \eqref{SX2}, it suffices to prove that for some $C>0$ independent of $n$,
\begin{align}\label{SX5}
\bE\rho_0(Z^{z_0}_t)\lesssim_C\rho_0(z_0),\ \ \forall z_0\in\mR^{2d}.
\end{align}
By It\^o's formula, we have
\begin{align*}
\bE\rho_0(Z^{z_0}_t)=\rho_0(z_0)+\bE\int^t_0(\Delta_v\rho_0+v\cdot\nabla_x\rho_0)(Z^{z_0}_s)\dif s
+\bE\int^t_0(b\cdot\nabla_v\rho_0)(s,Z^{z_0}_s)\dif s.
\end{align*}
Noting that by \eqref{AS000}, for some $C_0>0$,
$$
|\Delta_v\rho_0+v\cdot\nabla_x\rho_0|(z)\leq C_0\rho_0(z),
$$
we obtain
$$
\bE\rho_0(Z^{z_0}_t)\leq\rho_0(z_0)+C_0\bE\int^t_0\rho_0(Z^{z_0}_s)\dif s+\bE\int^t_0(b\cdot\nabla_v\rho_0)(s,Z^{z_0}_s)\dif s.
$$
To estimate the last term, we use Theorem \ref{Th33} { to deduce a Krylov's type of estimate. More precisely,} for fixed $t\in[0,T]$, let $w^{t}$ be the unique smooth solution of the following backward PDE:
$$
\p_sw^{t}+(\Delta_v+v\cdot\nabla_x+b\cdot\nabla_v)w^t=b\cdot\nabla_v\rho_0,\ \ w^t(t)=0.
$$
By It\^o's formula again, we have
$$
0=\bE w^{t}(t,Z^{z_0}_t)=w^{t}(0,z_0)+\bE\int^t_0(b\cdot\nabla_v\rho_0)(s, Z^{z_0}_s)\dif s.
$$
Hence,
\begin{align}\label{SX1}
\bE\rho_0(Z^{z_0}_t)\leq\rho_0(z_0)+C_0\bE\int^t_0\rho_0(Z^{z_0}_s)\dif s-w^t(0,z_0).
\end{align}
Let $\beta\in(\alpha,1)$ and $\rho_4:=(\rho_0\varrho)^{-1}$. By \eqref{AS000} and \eqref{FG1}, we have
$$
\|\nabla_v\rho_0\|_{\bC^{\beta}_a(\rho_4)}<\infty,
$$
which by \eqref{GZ3} yields that
$$
\|b\cdot\nabla_v\rho_0\|_{\mC^{-\alpha}_{T,a}(\rho_1\rho_4)}
\lesssim\|b\|_{\mC^{-\alpha}_{T,a}(\rho_1)}\|\nabla_v\rho_0\|_{\bC^\beta_a(\rho_4)}\lesssim\|b\|_{\mC^{-\alpha}_{T,a}(\rho_1)}.
$$
Moreover, by Lemma \ref{le:loc} we obtain
\begin{align*}
&\|b\circ\nabla_v\sI_\lambda(b\cdot\nabla_v\rho_0)\|_{\mC^{1-2\alpha}_{T,a}(\rho^2_1\rho_4)}\leq
\|(b\circ\nabla_v\sI_\lambda b)\nabla_v\rho_0\|_{\mC^{1-2\alpha}_{T,a}(\rho^2_1\rho_4)}\\
&\qquad+\|b\circ\nabla_v\sI_\lambda(b\cdot\nabla_v\rho_0)
-(b\circ\nabla_v\sI_\lambda b)\nabla_v\rho_0\|_{\mC^{1-2\alpha}_{T,a}(\rho^2_1\rho_4)}\\
&\quad\lesssim \|b\circ\nabla_v\sI_\lambda b\|_{\mC^{1-2\alpha}_{T,a}(\rho^2_1)}\|\nabla_v\rho_0\|_{\bC^{\beta}_a(\rho_4)}
+\|b\|^2_{\mC^{-\alpha}_{T,a}(\rho_1)}\|\nabla_v\rho_0\|_{\bC^{\beta}_a(\rho_4)}.
\end{align*}
Since $(2\vartheta+2)\kappa_1\leq 1$ and $\rho_1=\varrho^{\kappa_1}$, $\rho_4=\varrho^{-\kappa_0-1}$, by Theorem \ref{Th33} we have
\begin{align*}
\|w^t\|_{\mL^\infty_T(\rho_0^{-1})}\lesssim \mA^{b,b\cdot\nabla_v\rho_0}_{T,\infty}(\rho_1,\rho_1\rho_4)<\infty,
\end{align*}
which implies that for some $C_1>0$ independent of $n$ and $z_0$,
$$
|w^t(0,z_0)|\leq C_1\rho_0(z_0).
$$
Substituting this into \eqref{SX1} and by Gronwall's inequality we obtain \eqref{SX5}.

({\sc Step 3}) 
In this step we show \eqref{Rev44} by entropy method.  Recall $\chi$ in \eqref{ChiCut}. For $\delta\in(0,1)$ and $R\geq 1$, let
\begin{align*}
\beta_\delta(r):=r\ln(r+\delta),\ \ \chi_R(x,{v}):=\chi\big(\tfrac{x}{R^3},\tfrac{{v}}{R}\big).
\end{align*}
Since $u$ is a smooth solution of PDE \eqref{PDE8}, by the chain rule, it is easy to see that
$$
\p_t\beta_\delta(\rrho)=\Delta_{v}\beta_\delta(u)-{v}\cdot\nabla_x\beta_\delta(u)-b\cdot\nabla_{v}\beta_\delta(\rrho)
-\beta_\delta''(\rrho)|\nabla_{v}\rrho|^2.
$$
Multiplying both sides by $\chi_R$, then integrating over $[0,t]\times\mR^{2d}$ and by integration by parts and {$\textrm{div}_vb=0$}, we obtain
\begin{align}
&\int\chi_R\beta_\delta(\rrho(t))-\int\chi_R\beta_\delta(\varphi)+\int^t_0\int\chi_R\beta_\delta''(\rrho)|\nabla_{v}\rrho|^2\no\\
&\quad=\int^t_0\int\(\Delta_{v}\chi_R+{v}\cdot\nabla_x\chi_R+b\cdot\nabla_{v}\chi_R\)\beta_\delta(\rrho)\no\\
&\quad\le\|\Delta_{v}\chi_R+{v}\cdot\nabla_x\chi_R+b\cdot\nabla_{v}\chi_R\|_{\mL^\infty_T}\int^t_0\int\chi_{2R}|\beta_\delta(\rrho)|\no\\
&\quad\le C_\chi(1+\|b\|_{L^\infty}) R^{-1}\int^t_0\int \chi_{2R}|\beta_\delta(\rrho)|,\label{GronB2}
\end{align}
where $C_\chi$ only depends on $\chi$.
For $m\in\mN$, define
$$
G^m_R(t):=\int\chi_{2^mR}|\beta_\delta(\rrho(t))|.
$$
By Lemma \ref{ulnu}, \eqref{GronB2} and \eqref{SX2}, we obtain
\begin{align*}
G^m_R(t)&\le \int\chi_{2^mR}\beta_\delta(\rrho(t))+2\left(\int\rrho(t)\rho_0+\int\e^{-\rho_0}\right)\\
&\leq \frac{C_b}{2^mR}\int_0^t G^{m+1}_R(s)\dif s+\int|\beta_\delta(\varphi)|+C(\|\varphi\|_{L^1(\rho_0)}+1)\\
&\leq \frac{C_b}{R}\int_0^t G^{m+1}_R(s)\dif s+A_0,
\end{align*}
where $C_b:=C_\chi(1+\|b\|_{L^\infty})$ and
$$
A_0:=\int|\beta_\delta(\varphi)|+C(\|\varphi\|_{L^1(\rho_0)}+1)\leq\int|\varphi\ln\varphi|+1+C(\|\varphi\|_{L^1(\rho_0)}+1)<\infty.
$$
Here the first inequality is due to
\begin{align}\label{AX1}
|\beta_\delta(r)|\leq |r\log r|+r,\ \ \delta\in(0,1),\ \ r\geq 0,
\end{align}
{and the last inequality we used Lemma \ref{ulnu}.}
By iteration, we obtain that for any $m\in\mN$,
\begin{align*}
G^0_R(t)&\le A_0\sum_{k=0}^{m-1}\frac{C_b^{k}t^k}{R^k k!}
+\frac{C_b^m}{R^m}\int_{0}^t\int_0^{t_1}\cdot\cdot\cdot\int_0^{t_{m-1}}G^m_R(t_m)\dif t_m\cdot\cdot\cdot\dif t_1.
\end{align*}
Since $\rrho\in L^\infty_TC^\infty_b(\mR^{2d})$, there is a constant $C_{\delta}>0$ such that for any $R\geq 1$,
\begin{align*}
G^m_R(t)\le C_{\delta}\int\chi_{2^mR}\leq C_{\delta}(2^mR)^{2d}.
\end{align*}
Therefore,
\begin{align*}
G^0_R(t)\le A_0\e^{C_bt/R}+\frac{C^{m}_b}{R^m}C_{\delta}(2^mR)^{2d}\frac{t^m}{m!},
\end{align*}
which in turn implies that by first letting $m\to\infty$ and then $R\to\infty$,
\begin{align}\label{estimate}
\int|\beta_\delta(\rrho(t))|=\lim_{R\to\infty}G^0_R(t)\le A_0=\int|\beta_\delta(\varphi)|+C(\|\varphi\|_{L^1(\rho_0)}+1)<\infty.
\end{align}
Thus, by taking limits $R\to\infty$ on the both sides of \eqref{GronB2}, we obtain
\begin{align*}
\int \beta_\delta(\rrho(t))+\int_0^t\int\beta_\delta''(\rrho(s))|\nabla_{v}\rrho(s)|^2\dif s&\le \int \beta_\delta(\varphi).
\end{align*}
{By \eqref{estimate} \eqref{AX1} and Fatou's Lemma,} we further have
\begin{align}\label{Rev87}
\int|\rrho(t)\ln(\rrho(t))|\le \int|\varphi\ln\varphi|+C(\|\varphi\|_{L^1(\rho_0)}+1)<\infty,
\end{align}
Letting $\delta\downarrow 0$, by $\beta_\delta''(r)=\frac1{r+\delta}+\frac{\delta}{(r+\delta)^2}$ and Fatou's lemma,
\begin{align}\label{SX8}
\int\rrho(t)\ln(\rrho(t))+\int_0^t\int\frac{|\nabla_{v}\rrho(s)|^2}{\rrho(s)}\dif s\le \int\varphi\ln\varphi.
\end{align}
{Here for the first and last term we used \eqref{Rev87} \eqref{AX1} and dominated convergence theorem.}
On the other hand, by H\"older's inequality, we have
$$
\int_0^t\|\nabla_v \rrho(s)\|_{L^1}^2\dif t\le \int_0^t \left(\| \rrho(s)\|_{L^1}\int\frac{|\nabla_v\rrho(s)|^2}{\rrho(s)}\right)\dif s
=\int_0^t\left(\int\frac{|\nabla_v\rrho(s)|^2}{\rrho(s)}\right)\dif s.
$$
Substituting this into \eqref{SX8}, we obtain \eqref{Rev44}. The proof is complete.
\end{proof}

Now, we can give the proof of Theorem \ref{thm72}.
\begin{proof}[Proof of Theorem \ref{thm72}]
({\bf Existence}) By our definition of solutions it suffices to prove there exists a solution $u$ to the following equation:
\begin{align}\label{SMFL0}
\p_tu=\Delta_vu+v\cdot\nabla_xu+\tau\WW\cdot\nabla_vu+K*\mv\cdot\nabla_vu,\quad u(0)=\tau\varphi.
\end{align}
Let $\WW_n\in L^\infty_T C^\infty_b(\mR^{2d})$ be as in \eqref{DW19}.
 Let $\phi_n$ be the usual modifier and $K_n:=K*\phi_n\in C_b^\infty(\mR^d)$. Since the coefficients are bounded and Lipschitz
 and $\div_vW_n=\div_vK_n=0$, by standard fixed point argument,
one can show that there is a unique smooth probability density solution $\rrho_n$ to the following PDE
\begin{align}\label{SDSPDE}
\p_t\rrho_n=\Delta_{v}\rrho_n+{v}\cdot\nabla_x\rrho_n+(\tau\WW_n+K_n*\mvn)\cdot\nabla_{v}\rrho_n,
\ \ \rrho_n(0)={\tau\varphi_n}.
\end{align}
Define
$$
b_n(t,x,v):=\tau\WW_n(t,x,v)+K_n*\mvn(t,x).
$$
Since for $\beta>(\alpha-1)/3$
\begin{align*}
\|K_n*\mvn\|_{\bC^{\beta}}&\le\|K_n\|_{\bC^{\beta}}\|\mvn\|_{L^1}\leq \|K\|_{\bC^{\beta}}\|\rrho_n\|_{L^1}\lesssim1,
\end{align*}
by \eqref{GZ2}, \eqref{Lem34} and \eqref{EG1} and Remark \ref{Kpro} it is easy to see that
$$
\|b_n\circ\nabla_{v}\sI (K_n*\mvn)\|_{\bC^{3\beta+1-\alpha}_a(\rho_1)}
\lesssim \|b_n\|_{\bC^{-\alpha}_a(\rho_1)}\|K_n*\mvn\|_{\bC^{3\beta}_a}\lesssim1,
$$
where the implicit constant is independent of $n$.
Thus, by definition we have
\begin{align*}
\sup_n\ell^{ b_n}_{T}(\rho_1)\lesssim\sup_n\left(\ell^{W_n}_{T}(\rho_1)+\ell^{K_n*\mvn}_{T}(1)+{\sum_{i=1}^d\mA^{W_n,K^i_n*\mvn}_{T,\infty}(\rho_1,1)}\right)<\infty,
\end{align*}
and by Theorem \ref{Th33} and \eqref{CC9},
$$
\sup_{n}\Big(\|\rrho_n\|_{\mS^{2-\alpha}_{T,a}(\rho_3)}+\|\rrho^\sharp_n\|_{\mC^{3-2\alpha}_{T,a}(\rho_4)}\Big)<\infty.
$$
Thus, by Lemma \ref{CptE}, there are $\rrho\in\mS^{2-\alpha}_{T,a}(\rho_3)$
and subsequence $n_k$ such that for any $\beta>\alpha$ and $\rho_5\in\sP_{\rm w}$ with $\lim_{z\to\infty}(\rho_5/\rho_3)(z)=0$,
$$
\lim_{k\to\infty}\|\rrho_{n_k}-\rrho\|_{\mS^{2-\beta}_{T,a}({\rho_5})}=0.
$$
As in the proof of Theorem \ref{Th33}, one sees that
$u^\sharp:=u-P_t\varphi-\nabla_v u\prec  \sI b\in\mC^{3-2\alpha}_{T,a}(\rho_4)$
and for some $\rho_6\in\sP_{\rm w}$ and any $\beta>\alpha$,
$$
\lim_{k\to\infty}\|\rrho^\sharp_{n_k}-\rrho^\sharp\|_{{\mC^{3-2\beta}_{T,a}(\rho_6)}}=0.
$$
It is the same reason as in \eqref{Lim8}, 
we have
\begin{align*}
\lim_{k\to\infty}\|\rrho_{n_k}-\rrho\|_{L^\infty_T L^1}=0.
\end{align*}
In particular,
\begin{align}\label{eq:density}
\rrho\geq0,\ \  \int\rrho(t,z)\equiv1.
\end{align}
Since $K\in\bC^\beta_a$ for some $\beta>\alpha-1$,
\begin{align*}
\|K*\<u_{n_k}\>-K*\mv\|_{\mC^{\beta}_{T,a}}\to0\quad \text{as } k\to\infty.
\end{align*}
Let $\eps=(\beta-\alpha+1)/2>0$. By \eqref{FG1c}, we have
\begin{align*}
\|K_n*\mvn-K*\mvn\|_{\mC^{\alpha-1+\eps}_a}\lesssim n^{-\eps}\|K\|_{\bC^{\beta}}\to 0\quad \text{as } n\to\infty,
\end{align*}
which implies that
$$
\|b_n\circ \nabla \sI_\lambda b_n-b\circ \nabla \sI_\lambda b\|_{\mC^{1-2\alpha}_{T,a}(\rho_1^2)}\to0.
$$
Taking limits on the both sides of \eqref{SDSPDE}, one sees that $u$ is a probability density paracontrolled solution of PDE \eqref{SMFL0}.
\medskip\\
{\bf (Stability)} By our definition it only suffices to prove the result of solution to \eqref{SMFL0}.
Let $\rrho_1,\rrho_2$ be two paracontrolled solutions of PDE \eqref{SMFL}
with the initial values $\varphi_1$ and $\varphi_2$, respectively.
For $i=1,2$, let $u^n_i$ be the smooth approximation solution of the following linearized Fokker-Planck equation
\begin{align*}
\p_t\rrho^n_i=\Delta_{v}\rrho^n_i+{v}\cdot\nabla_x\rrho^n_i+(\tau\WW_n+K_n*\<u_i\>)\cdot\nabla_{v}\rrho^n_i,\ \rrho^n_i(0)=\varphi^n_i,
\end{align*}
where $\varphi^n_i=\varphi_i*\phi_n$ and $W_n$ is the approximation sequence in \eqref{DW19}, $K_n=K*\phi_n$.
By \eqref{MN2}, we have for some $\rho\in\sP_{\rm w}$,
\begin{align}\label{AM1}
\lim_{n\to\infty}\|u^{n}_i-u_i\|_{\mL^\infty_T(\rho)}=0,\ \ i=1,2.
\end{align}
Let
$$
w_n:=u^n_1-u^n_2,\ \ w:=u_1-u_2,
$$
and
$$
b_n:=\tau\WW_n+K_n*\mvo,\ \ f_n:=K_n*\<w\>\cdot\nabla_{v}\rrho^n_1,
$$
and for any $\delta>0$,
 $$
\beta_\delta(r):=\sqrt{r^2+\delta}-\sqrt\delta,\ \ \chi_R(x,{v}):=\chi\big(\tfrac{x}{R^3},\tfrac{{v}}{R}\big)
$$
It is easy to see that
$$
\p_tw_n=\Delta_{v}w_n+{v}\cdot\nabla_xw_n+b_n\cdot\nabla_{v}w_n+f_n,
$$
and {similar as \eqref{GronB2}} by the chain rule and the integration by parts,
\begin{align*}
\p_t\int\chi_R\beta_\delta(w_n)&=\int(\Delta_{v}\chi_R-{v}\cdot\nabla_x\chi_R) \beta_\delta(w_n)-\int\chi_R\beta_\delta''(w_n)|\nabla_{v}w_n|^2\\
&\quad-\int(b_n\cdot\nabla_{v}\chi_R)\beta_\delta(w_n)+\int f_n\chi_R\beta'_\delta(w_n).
\end{align*}
Since $|\beta_\delta(r)|\le |r|$, $|\beta'_\delta(r)|\le 1$ and $\int|w_n|\leq 2$, there is a constant $C>0$ independent of $R$ such that
\begin{align*}
\p_t\int\chi_R\beta_\delta(w_n)\leq C (R^{-2}+\|b_n\|_{L^\infty} R^{-1})+\int |f_n|\chi_R.
\end{align*}
Integrating both sides from $0$ to $t$ and  letting $R\to\infty$ and $\delta\to 0$, we obtain
\begin{align*}
\|w_n(t)\|_{L^1}\leq \|w_n(0)\|_{L^1}+\int^t_0 \|f_n\|_{L^1}\dif s.
\end{align*}
Note that by H\"older's inequality,
\begin{align*}
\int^t_0\|f_n\|_{L^1}\dif s&\leq\int^t_0\|K_n*\<w\>\|_{L^\infty}\|\nabla_v u^n_1\|_{L^1}\dif s\\
&\leq\|K\|_{L^\infty} \left(\int_0^t\|w\|_{L^1}^2 \dif s\right)^{\frac12}\|\nabla_v u^n_1\|_{L^2_tL^1}.
\end{align*}
Since $\tau$ does not change entropy and \eqref{Rev44}, \eqref{Rev43} and \eqref{Rev87} also hold for $u_n$ which implies that
\begin{align*}
\|\nabla_v u^n_1\|_{L^2_tL^1}&\leq H(\varphi^n_1)-H(u^n_1(t))\lesssim_C \int|\varphi_1\ln\varphi_1|+(\|\varphi_1\|_{L^1(\rho_0)}+1),
\end{align*}
where $C$ only depends on $\rho_0$. Thus,
$$
\|w_n(t)\|_{L^1}\leq \|w_n(0)\|_{L^1}+C_{K,\varphi_1,\rho_0}\left(\int_0^t\|w(s)\|_{L^1}^2 \dif s\right)^{\frac12}.
$$
Letting $n\to\infty$ and by \eqref{AM1} and Fatou's lemma,
$$
\|w(t)\|_{L^1}\leq \|w(0)\|_{L^1}+C_{K,\varphi_1,\rho_0}\left(\int_0^t\|w(s)\|_{L^1}^2 \dif s\right)^{\frac12},
$$
which implies \eqref{SX89} by Gronwall's inequality.
\end{proof}

\section{Nonlinear martingale problem with singular drifts}\label{kddsde}

Fix $T>0$. In this section we consider the following nonlinear kinetic DDSDE with distributional drift:
\begin{align}\label{ksde}
\dif X_t=V_t\dif t,\ \ \dif V_t=\WW(t,X_t,V_t)\dif t+(K*\mu_{X_t})(X_t)\dif t+\sqrt{2}\dif B_t,
\end{align}
where $W\in \mB^\alpha_T(\varrho^\kappa)$ for some $\kappa>0$ and $K(x):\mR^d\to\mR^d$ satisfies that
\begin{align*}
K\in\cup_{\beta>\alpha-1}\bC^{\beta/3}.
\end{align*}
Here $\mu_{X_t}$ stands for the law of $X_t$ in $\mR^d$, and for a probability measure $\mu$ in $\mR^d$,
$$
K*\mu(x):=\int_{\mR^d}K(x-y)\mu(\dif y).
$$
Fix $T>0$. Let $\cC_T$ be the space of all continuous functions from $[0,T]$ to $\mR^{2d}$,
and $\cP(\cC_T)$ the set of all probability measures over $\cC_T$.
Let  $\sB_t$ be the natural $\sigma$-filtration, and $z$ be the canonical process over $\cC_T$, i.e., for $\omega\in\cC_T$,
$$
z_t(\omega)=(x_t(\omega),v_t(\omega))=\omega_t.
$$
As mentioned in the introduction, we define the martingale problem by using the linear version of the Kolomogorov backward equation. More precisely,
for a continuous curve $\mu:[0,T]\to\cP(\mR^d)$ with respect to the weak convergence, define
\begin{align*}
\sL^\mu_t:=\Delta_v+v\cdot\nabla_x+(\WW(t)+K*\mu_t)\cdot\nabla_v.
\end{align*}
As in Remark \ref{Re52}, it is easy to see that $W+K*\mu_t\in\mB^\alpha_T(\varrho^\kappa)$.
Let $f\in L^\infty_TC_b(\mR^{2d})$ and $\varphi\in\bC^\gamma_a(\mR^{2d})$ for some $\gamma>1+\alpha$ and $\vartheta:=\frac{9}{2-3\alpha}$.
By Theorem \ref{Th33}, there exists a unique paracontrolled solution $u^\mu_f\in\mS^{2-\alpha}_{T;a}(\varrho^{2(\vartheta+1)\kappa})$
to the following equation:
\begin{align}\label{pde}
\p_t u^\mu_f+\sL^\mu_t u^\mu_f=f,\quad u_f^\mu(T)=\varphi.
\end{align}
For any $\delta>0$, let $\cP_\delta(\mR^{2d})$ be the space of all probability measures $\nu$ on $\mR^{2d}$ with
\begin{align*}
\int_{\mR^{2d}}\varrho(z)^{-\delta}\nu(\dif z)\asymp\int_{\mR^{2d}}(1+|z|^\delta_a)\nu(\dif z)<\infty.
\end{align*}

We introduce the following notion about the martingale problem.
 \bd\label{MP8}(Martingale problem)
Let $\delta>0$. A probability measure $\mP\in\cP(\cC_T)$ is called a martingale solution to SDE \eqref{ksde}
starting from $\nu\in\cP_\delta(\mR^{2d})$, if
 $\mP\circ Z_0^{-1}=\nu$ and for all $f\in L^\infty_TC_b(\mR^{2d})$ and $\varphi\in\bC^\gamma_a(\mR^{2d})$  with some $\gamma>1+\alpha$,
\begin{align*}
M_t:=u^\mu_f(t,z_t)-u^\mu_f(0,z_0)-\int_0^t f(s, z_s)\dif s
\end{align*}
is a martingale under $\mP$ with respect to $(\sB_t)$, where $\mu_t:=\mP\circ x^{-1}_t$ and $u^\mu_f$ is the paracontrolled solution to \eqref{pde}.
The set of all martingale solutions $\mP$ associated with $W,K$ and
starting from $\nu$ is denoted by $\sM_\nu(W,K)$.
\ed


\br\rm
The moment assumption for $\nu$ is necessary for making sense of $\mE u^\mu_f(0,{z_0})$
since the solution $u^\mu_f$ lives in weighted spaces.
\er

Our main result of this section is the following:
\bt\label{Main8}
Let $\alpha\in(\frac{1}{2},\frac{2}{3})$ and $\vartheta:=\frac9{2-3\alpha}$.
 Suppose that for some $\kappa\in(0,\frac1{2\vartheta+2}{]}$ and $\beta>\alpha-1$,
 $$
 \WW\in\mB^{\alpha}_T(\varrho^\kappa),\ \ K\in\bC^{\beta/3}.
 $$
 Then for any $\nu\in\cP_\delta(\mR^{2d})$ with $\delta>{(4\vartheta+4)\kappa}$,
there exists at least one martingale solution $\mP\in\sM_\nu(W,K)$ to SDE \eqref{ksde}. Moreover, if $K$ is bounded measurable, then
there exists at most one $\mP\in\sM_\nu(W,K)$.
\et
Let $W_n$ be the approximation sequence of $W$, and $K_n=K*\phi_n$ with $\phi_n$ being the usual mollifier.
We consider the following approximation SDE:
\begin{align}\label{Asde}
\dif X^n_t=V^n_t\dif t,\
\dif V^n_t=\WW_n(t, Z^n_t)\dif t+(K_n*\mu_{X^n_t})(X^n_t)\dif t+\sqrt{2}\dif B_t,
\end{align}
where $Z^n=(X^n,V^n)$ and $\bP^{-1}\circ Z^n_0=\nu$.
Since $W_n$ and $K_n$ are globally Lipschitz, it is well-known that there is a unique strong solution $Z^n$ to \eqref{Asde} (see   \cite[Theorem 2.1]{Wa18}). 
We first establish the following uniform moment estimates for $V^n_t$ by a PDE's method.
\bl\label{Lem82}
 Suppose $\delta>{(4\vartheta+4)\kappa}$.
For any $p\in({2},\frac{\delta}{(2\vartheta+2)\kappa}]$, there is a constant $C>0$ such that for all $0\le s<t\le T$,
 \begin{align*}
\sup_n\bE|V^n_t-V^n_s|^p\lesssim_C (t-s)^{p/2}.
\end{align*}
\el
\begin{proof}
By SDE \eqref{Asde}, it suffices to prove that
\begin{align}\label{SP1}
\sup_n\bE\left|\int_s^t b_n(r,Z^n_r)\dif r\right|^p\lesssim_C|t-s|^{ (2-\alpha)p/2},
\end{align}
where
$$
b_n(t,x,v):=W_n(t,x,v)+(K_n*\mu_{X^n_t})(x)\in L^\infty_TC^\infty_b(\mR^{2d}).
$$
Fix $t\in[0,T]$. Let $u_n$ be the smooth solution to the following kinetic equation
\begin{align*}
\p_t u_n=\Delta_v u_n+v\cdot\nabla_x u_n+b^t_n\cdot\nabla_v u_n-b^t_n,\quad u_n(0)=0,
\end{align*}
where $b^t_n(s,z)=b_n(t-s,z)$.
By Theorem \ref{Th33}, for $\sigma:=(2\vartheta+2)\kappa$, we have
\begin{align}\label{PF82}
\sup_n\|u_n\|_{\mS^{2-\alpha}_{t,a}(\varrho^\sigma)}<\infty.
\end{align}
Let
\begin{align*}
u^t_n(s):=u_n(t-s).
\end{align*}
Then $u^t_n$ satisfies the following equation
\begin{align*}
\p_ru^t_n+\Delta_v u^t_n+v\cdot\nabla_xu^t_n+b_n\cdot\nabla_v u^t_n=b_n,\quad u^t_n(t)=0.
\end{align*}
By \eqref{Asde} and It\^o's formula, one sees that
\begin{align*}
\int_s^tb_n(r,Z^n_r)\dif r&=u^t_n(t,Z^n_t)-u^t_n(s,Z^n_s)-\sqrt{2}\int_s^t\nabla_v u^t_n(r,Z^n_r)\dif B_r\\
&=u^t_n(t,\Gamma_{t-s}Z^n_s)-u^t_n(s,Z^n_s)-\sqrt{2}\int_s^t\nabla_v u^t_n(r,Z^n_r)\dif B_r,
\end{align*}
where the second step is due to
\begin{align*}
u^t_n(t,Z^n_t)=0=u^t_n(t,\Gamma_{t-s}Z^n_s).
\end{align*}
By BDG's inequality, \eqref{SS0} and \eqref{Lem34}, we have for any $p\in({2},\frac{\delta}{(2\vartheta+2)\kappa}]$,
\begin{align}
\bE\left|\int_s^tb_n(r,Z^n_r)\dif r\right|^p&\lesssim(t-s)^{(2-\alpha)p/2}\| u_n\|_{\mS^{2-\alpha}_{t,a}
(\varrho^\sigma)}^p\bE\varrho^{-p\sigma}(Z^n_s)\no\\
&+\| \nabla_v u_n\|_{\mL^\infty_{t}(\varrho^{\sigma})}^p\bE\left(\int_s^t\varrho^{-2\sigma}(Z^n_r)\dif r\right)^{p/2} \no\\
&\lesssim ((t-s)^{(2-\alpha)p/2}{+(t-s)^{p/2}})\|u_n\|_{\mS^{2-\alpha}_{t,a}(\varrho^{\sigma})}^p \sup_{r\in[0,t]}\bE\varrho^{-p\sigma}(Z^n_r).\label{SP2}
\end{align}
Finally, since $p\sigma\leq\delta$, as in showing \eqref{SX5}, we have
\begin{align}\label{DM1}
\sup_{n}\sup_{s\in[0,t]}\bE\varrho^{-\delta}(Z^n_s)\lesssim \int_{\mR^{2d}}\varrho^{-\delta}(z)\nu(\dif z).
\end{align}
By \eqref{SP2}, \eqref{DM1} and \eqref{PF82}, we obtain \eqref{SP1}. The proof is complete.
\end{proof}

Now we give the following convergence result.
\bl\label{Le65}
Let $(\mu_n)_{n\in\mN}$ be a family of probability measures on $C([0,T];\mR^d)$.
Suppose that $\mu_n$ weakly converges to $\mu$ and $K\in\bC^{\beta}$, 
Then for any {$\beta_0<\beta$}, we have
$$
\lim_{n\to\infty}\|K_n*\mu_n-K*\mu\|_{L^\infty_T\bC^{\beta_0}}=0.
$$
\el
\begin{proof}
{It suffices to prove the result for $\beta_0$ satisfying $\beta-\beta_0\in(0,1)$.}
By Skorohod's representation theorem, there are a probability space $(\widetilde\Omega,\widetilde\sF,\widetilde\mP)$
and random variables $X_n$, $X$ with values in $C([0,T];\mR^d)$  so that
$$
\lim_{n\to\infty}\sup_{s\in[0,T]}|X_n(s)-X(s)|=0\ \ a.s.,
$$
and
$$
\widetilde\mP\circ(X_n)^{-1}=\mu_n,\ \ \widetilde\mP\circ X^{-1}=\mu.
$$
Let $\cR_j$ be the usual block operator with $a=(1,\cdots,1)$ in \eqref{Def2}. 
{By similar arguments as \eqref{Cor28} we have} for any $j\geq -1$ and $h\in\mR^d$,
\begin{align}\label{estiK}
\|\cR_jK(\cdot+h)-\cR_jK\|_{L^\infty}\lesssim |h|^{\beta-\beta_0}\|\cR_jK\|_{\bC^{\beta-\beta_0}}
\lesssim
2^{-\beta_0j}|h|^{\beta-\beta_0}\|K\|_{\bC^{\beta}}.
\end{align}
From this, we derive that
\begin{align}\label{estiK0}
\|\cR_jK_m-\cR_jK\|_{L^\infty}
\lesssim 2^{-\beta_0j}m^{-(\beta-\beta_0)}\|K\|_{\bC^{\beta}}.
\end{align}
Note that
\begin{align*}
|\cR_j(K_n*\mu_n(s)-K*\mu(s))(x)|&=|\widetilde\mE \cR_jK_n(x-X_n(s))-\widetilde\mE \cR_jK(x-X(s))|\\
&\leq\widetilde\mE |\cR_jK_n(x-X_n(s))-\cR_jK(x-X_n(s))|\\
&\quad+\widetilde\mE|\cR_jK(x-X_n(s))-\cR_jK(x-X(s))|\\
&=:\cJ^{(1)}_{n,j}(s,x)+\cJ^{(2)}_{n,j}(s,x).
\end{align*}
For $\cJ^{(1)}_{n,j}(s,x)$, by \eqref{estiK0} we have
\begin{align*}
\|\cJ^{(1)}_{n,j}\|_{L^\infty_TL^\infty}\leq \|\cR_jK_n-\cR_jK\|_{L^\infty}
\lesssim 2^{-\beta_0 j} n^{-(\beta-\beta_0)}\|K\|_{\bC^{\beta}}.
\end{align*}
For $\cJ^{(2)}_{n,j}(s,x)$, by the dominated convergence theorem and \eqref{estiK}, we have
\begin{align*}
\lim_{n\to\infty}\sup_j2^{\beta_0 j}\|\cJ^{(2)}_{n,j}\|_{L^\infty_TL^\infty}&{\lesssim\widetilde\mE\left(\lim_{n\to\infty}\sup_{s\in[0,T]}|\cR_jK(x-X_n(s))-\cR_jK(x-X(s))|\right)}\\
&\lesssim \widetilde\mE\left(\lim_{n\to\infty}\sup_{s\in[0,T]}|X_n(s)-X(s)|^{\beta-\beta_0}\right)\|K\|_{\bC^{\beta}}=0.
\end{align*}
From these two estimates, we derive the desired limit.
\end{proof}
Now we can give the proof of Theorem \ref{Main8}.

\begin{proof}[Proof of Theorem \ref{Main8}]
{\bf (Existence)} Let $\mP_n=\bP\circ Z^n_\cdot$ be the law of $Z^n$ in $(\cC_T,\sB_T)$.
By Lemma \ref{Lem82} and Kolmogorov's criterion, we have for each $\eps>0$,
\begin{align*}
\lim_{\delta\to 0}\sup_n\bP\left(\sup_{s\not=t\in[0,T],|t-s|\leq\delta}|V^n_t-V^n_s|>\eps\right)=0.
\end{align*}
Since $X^n_t=\int^t_0V^n_s\dif s{+X_0}$ and
$$
\lim_{R\to\infty}\sup_n\bP(|Z^n_0|>R)=\lim_{R\to\infty}\nu\{z: |z|>R\}=0,
$$
it is easy to see that for each $\eps>0$,
\begin{align*}
\lim_{\delta\to 0}\sup_n\bP\left(\sup_{s\not=t\in[0,T],|t-s|\leq\delta}|Z^n_t-Z^n_s|>\eps\right)=0.
\end{align*}
Thus  $(\mP_n)_{n\in\mN}$ is tight in $\cC_T$.

Let $\mP$ be any accumulation point of $(\mP_n)_{n\in\mN}$. Without loss of generality, we assume $\mP_n$ weakly converges to $\mP$. Let
$$
\mu_n:=\mP_n\circ X^{-1}_\cdot,\ \ \mu:=\mP\circ X^{-1}_\cdot.
$$
Let $\phi_n$ be the usual mollifier in $\mR^{2d}$ and define
$$
f_n(t,z):=f(t,\cdot)*\phi_n(z),\ \ \varphi_n(z):=\varphi*\phi_n(z)
$$
and
$$
b_n:=W_n+K_n*\mu_n,\ \ b:=W+K*\mu.
$$
Since $b_n,f_n\in L^\infty_TC^\infty_b(\mR^{2d})$, it is well known that there is a smooth solution $u_n\in L^\infty_TC^\infty_b(\mR^{2d})$
to the following PDE:
\begin{align}\label{SDE78}
\p_t u_n+(\Delta_v+v\cdot\nabla_x+b_n\cdot\nabla_v)u_n=f_n,\quad u_n(T)=\varphi_n.
\end{align}
Now we define two functionals on $\cC_T$:
\begin{align*}
M^n_t:=M^n_t(z):=u_n(t,z_t)-u_n(0,z_0)-\int_0^t f_n(s, z_s)\dif s
\end{align*}
and
\begin{align*}
M_t:=M_t(z):=u^\mu_f(t, z_t)-u^\mu_f(0, z_0)-\int_0^t f(s, z_s)\dif s.
\end{align*}
We want to show that for any $0\leq s<t\leq T$ and $\sB_{s}$-measurable bounded continuous functional $G_{s}$ on $\cC_T$,
\begin{align}\label{SA1}
\mE^{\mP}\big(M_{t}G_{s}\big)=\mE^{\mP}\big(M_{s}G_{s}\big).
\end{align}
For each $n\in\mN$, by \eqref{Asde}, \eqref{SDE78} and It\^o's formula,
\begin{align*}
M^n_t(Z^n)=\int_0^t\nabla_v u_n(s,Z^n_s)\dif B_s
\end{align*}
is a $\bP$-martingale. Hence,
\begin{align*}
\mE^{\mP_n}\big(M^n_{t}G_s\big)=\bE\big(M^n_{t}(Z^n)G_s(Z^n)\big)=\bE\big(M^n_{s}(Z^n)G_s(Z^n)\big)=\mE^{\mP_n}\big(M^n_sG_s\big).
\end{align*}
Thus, to show \eqref{SA1}, it suffices to show that
\begin{align}\label{convergence}
\lim_{n\to\infty}\mE^{\mP_n}\big(M^n_{t}G_s\big)=\mE^{\mP}\big(M_{t}G_s\big).
\end{align}
Note that by Lemma \ref{Le65},  for $\gamma\in(\alpha-1,\beta)$,
$$
\lim_{n\to\infty}\|K_n*\mu_n-K*\mu\|_{\mC^\gamma_{T,a}}=0,
$$
which by Lemma \ref{Le317} implies that
$$
(b,f)\in\mB^{\alpha}_T(\varrho^\kappa,1) \mbox{ with approximation sequence $(b_n,f_n)$}.
$$
Thus by Theorem \ref{Th33}, for any $\sigma>(2+2\vartheta)\kappa$,
\begin{align}\label{BZ1}
\sup_n\|u_n\|_{\mL^\infty_T(\varrho^{\sigma})}<\infty,\ \ \lim_{n\to\infty}\|u_n-u^\mu_f\|_{\mL^\infty_T(\varrho^{\sigma})}=0.
\end{align}
Moreover, by \eqref{DM1} we have for any $\delta>{(4+4\vartheta)\kappa}$,
\begin{align*}
\sup_n\mE^{\mP_n}\Big(\varrho^{{-\delta}}(z_t)+\varrho^{{-\delta}}(z_0)\Big)<\infty.
\end{align*}
Note that
$$
|M^n_{t}-M_{t}|\leq \|u_n-u^\mu_f\|_{\mL^\infty_T(\varrho^{\sigma})}\Big(\varrho^{{-\sigma}}(z_t)+\varrho^{{-\sigma}}(z_0)\Big)
+\int_0^t |f_n-f|(s, z_s)\dif s.
$$
Since for each $s\in[0,t]$, $(\mP_n\circ z_s^{-1})_{n\in\mN}$ is tight, and for any $R>0$,
$$
\lim_{n\to\infty}\sup_{|z|\leq R}|f_n(s,z)-f(s,z)|=0,
$$
it is easy to see that
$$
{\lim_{n\to\infty}\mE^{\mP_n}\int_0^t |f_n-f|(s, z_s)\dif s\leq \int_0^t\lim_{n\to\infty}\sup_{|z|\leq R}|f_n(s,z)-f(s,z)|\dif s+\frac{C}{R^\delta}\int_0^t\sup_n\mE^{\mP_n}\varrho^{-\delta}(z_s)\dif s.}
$$
Thus, by \eqref{BZ1},
\begin{align}\label{ME8}
\lim_{n\to\infty}\big|\mE^{\mP_n}\big(M^n_{t}G_s\big)-\mE^{\mP_n}\big(M_{t}G_s\big)\big|
\leq\|G_s\|_\infty\lim_{n\to\infty}\mE^{\mP_n}\big|M^n_{t}-M_{t}\big|=0.
\end{align}
Moreover, since $M_t$ is a continuous functional on $\cC_T$, and
$$
\sup_n\mE^{\mP_n}|M_tG_s|^{{2}}
\stackrel{\eqref{BZ1}}{\lesssim}
\left[\sup_n\mE^{\mP_n}\Big(\varrho^\delta(z_t)+\varrho^\delta(z_0)\Big)+{\|f\|^2_{L^\infty}}\right]\|G\|_{L^\infty}\stackrel{\eqref{DM1}}{<}\infty,
$$
it is easy to see that
\begin{align*}
\lim_{n\to\infty}\mE^{\mP_n}\big(M_{t}G_s\big)=\mE^{\mP}\big(M_{t}G_s\big).
\end{align*}
Combining the above calculations, we get \eqref{convergence}.
Thus, we complete the proof of the existence of a martingale solution.

{\bf (Uniqueness)} First of all, we show the uniqueness for linear SDE, i.e., $K=0$.
Let $\mP_1,\mP_2\in\sM_\nu(W,0)$ be two solutions of the martingale problem. 
Let $f\in L_T^\infty C_b(\mR^{2d})$ and let ${u}$ be the unique paracontrolled solution to \eqref{pde} with $u(t)=0$.
By Definition \ref{MP8}, we have
$$
\int_{\mR^{2d}}u(0,z)\nu(\dif z)=-\mE^{\mP_i}{\int_0^Tf(s,Z_s)}\dif s,\quad i=1,2,
$$
which means that
\begin{align*}
{\int_0^T\mE^{\mP_1} f(s,Z_s)\dif s=\int_0^T\mE^{\mP_2} f(s,Z_s)\dif s.}
\end{align*}
Hence, for any $f\in C_b(\mR^{2d})$ and $t\in[0,T]$,
$$
\mE^{\mP_1} f(Z_t)=\mE^{\mP_2} f(Z_t).
$$
From this, by a standard way (see Theorem 4.4.3 in \cite{EK86}), we derive that
$$
\mP_1=\mP_2.
$$

For general nonlinear SDE, we use Girsanov's transformation method. Let $\mP_1,\mP_2\in\sM_\nu(W,K)$ be two solutions of the martingale problem.
Let $W_n, K_n$ be the approximations of $W$ and $K$ as above.
We consider the following approximation of linearized SDEs: for $i=1,2$,
\begin{align}\label{Asde0}
\dif X^{i,n}_t=V^{i,n}_t\dif t,\
\dif V^{i,n}_t=\WW_n(Z^{i,n}_t)\dif t+(K_n*\mu^i_t)(X^{i,n}_t)+\sqrt{2}\dif B_t,
\end{align}
where $\mu^i_t:=\mP_i\circ x_t^{-1}$.
As in the proof of the existence part,
and due to the uniqueness of linear SDEs, for $i=1,2$,
the law of $Z^{i,n}=(X^{i,n},V^{i,n})$ weakly converges to $\mP_i$ as $n\to\infty$. In particular, for any $\varphi\in C_b(\mR^d)$,
$$
\bE\varphi(X^{i,n}_t)\to \mE^{\mP_i}\varphi(x_t),\ \ i=1,2.
$$
On the other hand, we define
$$
A^{i,n}_t:=\exp\left\{-\frac{1}{\sqrt{2}}\int^t_0(K_n*\mu^i_s)(X^{i,n}_s)\dif B_s-\frac14\int^t_0|(K_n*\mu^i_s)(X^{i,n}_s)|^2\dif s\right\}.
$$
Since
\begin{align}\label{AS9}
\|K_n*\mu^i_s\|_{L^\infty}\leq\|K\|_{L^\infty},
\end{align}
by Girsanov's theorem, under the new probability measure $Q^{i,n}:=A^{i,n}_T\bP$, {for $t\in[0,T]$}
$$
\widetilde B^{i,n}_t:=\frac{1}{\sqrt{2}}\int^t_0(K_n*\mu^i_s)(X^{i,n}_s)\dif s+B_t
$$
is still a Brownian motion, and
$$
\dif X^{i,n}_t=V^{i,n}_t\dif t,\ \dif V^{i,n}_t=\WW_n(Z^{i,n}_t)\dif t+\sqrt{2}\dif \widetilde B^{i,n}_t.
$$
Since the above SDE admits a unique weak solution, we have
$$
Q^{1,n}\circ (Z^{1,n})^{-1}=Q^{2,n}\circ (Z^{2,n})^{-1}.
$$
Thus, for any $\varphi\in C_b(\mR^d)$,
\begin{align*}
\bE\varphi(X^{1,n}_t)=\bE(A^{1,n}_T\varphi(X^{1,n}_t)/A^{1,n}_T)
=\bE(A^{2,n}_T\varphi(X^{2,n}_t)Y^n_T)
\end{align*}
and
\begin{align*}
|\bE\varphi(X^{1,n}_t)-\bE\varphi(X^{2,n}_t)|\leq\|\varphi\|_{L^\infty}\bE|A^{2,n}_TY^n_T-1|,
\end{align*}
where
$$
Y^n_T:=\exp\left\{\frac{1}{\sqrt{2}}\int^T_0(K_n*\mu^1_s)(X^{2,n}_s)\dif B_s+\frac14\int^T_0|(K_n*\mu^1_s)(X^{2,n}_s)|^2\dif s\right\}.
$$
On the other hand, by It\^o's formula, we have
\begin{align*}
A^{2,n}_TY^n_T-1&=\frac{1}{\sqrt{2}}\int^T_0A^{2,n}_sY^n_s(K_n*\mu^1_s-K_n*\mu^2_s)(X^{2,n}_s)\dif B_s\\
&+\frac14\int^T_0A^{2,n}_sY^n_s(|(K_n*\mu^1_s)(X^{2,n}_s)|^2-|(K_n*\mu^2_s)(X^{2,n}_s)|^2)\dif s\\
&+\frac14\int^T_0A^{2,n}_sY^n_s|(K_n*\mu^1_s)(X^{2,n}_s)-(K_n*\mu^2_s)(X^{2,n}_s)|^2\dif s.
\end{align*}
By \eqref{AS9}, it is by now standard to derive that for any $p\geq2$,
$$
\sup_n\sup_{s\in[0,T]}\bE|A^{2,n}_s|^p+\sup_n\sup_{s\in[0,T]}\bE|Y^{n}_s|^p<\infty.
$$
Hence, by BDG's inequality and H\"older's inequality, we arrive at
\begin{align*}
\bE|A^{2,n}_TY^n_T-1|&\lesssim\left(\int^T_0\|\mu^2_s-\mu^1_s\|^2_{TV}\dif s\right)^{1/2}+\int^T_0\|\mu^2_s-\mu^1_s\|_{TV}\dif s,
\end{align*}
where $\|\cdot\|_{TV}$ stands for the total variation norm of a signed measure.
Combining the above calculations, we obtain that for all $\varphi\in C_b(\mR^d)$,
\begin{align*}
&|\mE^{\mP_1}\varphi(x_T)-\mE^{\mP_2}\varphi(x_T)|=\lim_{n\to\infty}|\bE\varphi(X^{1,n}_T)-\bE\varphi(X^{2,n}_T)|\\
&\qquad\qquad\lesssim\|\varphi\|_\infty\left(\int^T_0\|\mu^2_s-\mu^1_s\|^2_{TV}\dif s\right)^{1/2},
\end{align*}
which in turn implies that
$$
\|\mu^2_T-\mu^1_T\|^2_{TV}\lesssim\int^T_0\|\mu^2_s-\mu^1_s\|^2_{TV}\dif s.
$$
By Gronwall's inequality, $\mu^1_t=\mu^2_t$. Finally,  we  use {the same argument as} the uniqueness for linear equations to derive the uniqueness for nonlinear SDEs.
\end{proof}

\section{Existence of renormalized pairs in probabilistic sense}\label{Sub6}

{In this section we perform the construction of stochastic objects, i.e. renormalized pairs of the stochastic kinetic equations by probabilistic calculations. We state the main result in Subsection \ref{sec:7.1}. In Subsection \ref{sec:pex} we give examples of Gaussian noise satisfying the general assumptions. In the last subsection we give the proof of the main result. }
\subsection{Statement of main result}\label{sec:7.1}
Let $\mu$ be a symmetric tempered measure on $\mR^{2d}$, that is, for some $l\in\mN$,
\begin{align}\label{MU}
\int_{\mR^{2d}} (1+|\xi|)^{-l}\mu(\dif\xi)<\infty.
\end{align}
Let $L^2_\mC(\mu)$ be the complex-valued Hilbert space with inner product
\begin{align*}
\<f,g\>_{L^2_\mC(\mu)}:=\int_{\mR^{2d}}f(\zeta)\,\overline{g(\zeta)}\mu(\dif\zeta)<\infty.
\end{align*}
{Let $\mH$ be the completion of $\sS(\mR^{2d})$ with respect to the  inner product
$$
\<f,g\>_{\mH}:=\<\hat f,\hat g\>_{L^2_\mC(\mu)}.
$$}
\bd
Let $X$ be a Gaussian field on $\mH$, i.e., $X$ is a continuous linear operator from $\mH$ to $L^2(\Omega,\bP)$,
and for each $f\in\mH$,
$X(f)$ is a real-valued Gaussian random variable with mean zero and variance $\|f\|^2_\mH$. In particular,
\begin{align}\label{ISO}
\bE\big(X(f)X(g)\big)=\int_{\mR^{2d}}\hat f(\zeta)\,\hat g(-\zeta)\mu(\dif\zeta).
\end{align}
We call $X$ the Gaussian noise with spectral measure $\mu$ (see \cite{SV97}).
\ed
The following result is the main result of this section.
\bt\label{StochM}
Suppose that $\mu$ is a Radon measure and satisfies
\begin{align}\label{Sy}
\mu(\dif\xi,\dif\eta)=\mu(\dif\xi,-\dif\eta)=\mu(-\dif\xi,\dif\eta), \tag{$\bf S$}
\end{align}
and for some $\beta\in(\frac12,\frac23)$,
\begin{align}\label{Stoch01}
\sup_{\zeta'\in\mR^{2d}}\int_{\mR^{2d}}\frac{\mu(\dif\zeta)}{(1+|\zeta'+\zeta|_a)^{2\beta}}<\infty. \tag{$\bf A^\beta$}
\end{align}
Let $W=(X_1, \cdots, X_d)$ be $d$-independent Gaussian noise with common spectral measure $\mu$.
Then for any $\kappa>0$ and $\alpha>\beta$, it holds that
$$
\bP\big\{\omega: W(\cdot,\omega)\in\mB^\alpha_T(\varrho^\kappa)\big\}=1.
$$
\et
\br\rm
(i) Condition \eqref{Stoch01} implies that for any $\sigma,\gamma\geq 0$ with $\sigma+\gamma=2\beta$,
\begin{align}
\sup_{\zeta'\in\mR^{2d}}\int_{\mR^{2d}}\frac{\mu(\dif\zeta)}{(1+|\zeta'+\zeta|_a)^{\sigma}(1+|\zeta|_a)^{\gamma}}<\infty.\label{StoRk1}
\end{align}
Indeed, it follows by the simple observation:
\begin{align*}
&\left(\int_{{|\zeta|_a>|\zeta'+\zeta|_a}}+\int_{|\zeta|_a\le|\zeta'+\zeta|_a}\right)\frac{\mu(\dif\zeta)}{(1+|\zeta'+\zeta|_a)^{\sigma}(1+|\zeta|_a)^{\gamma}}\\
&\qquad\le\int_{\mR^{2d}}\frac{\mu(\dif\zeta)}{(1+|\zeta'+\zeta|_a)^{2\beta}}+\int_{\mR^{2d}}\frac{\mu(\dif\zeta)}{(1+|\zeta|_a)^{2\beta}}.
\end{align*}
(ii) The symmetric assumption of $\mu$ in the second variable $\eta$ allows us to use some
cancelation to show the convergence in \eqref{JJR} below (see \eqref{CA1} below). {In the classical case by symmetry the terms in the $0$th  Wiener chaos are zero. Here the terms in the $0$th  Wiener chaos are not zero, but they converge after minus  renormalization terms which are zero by symmetry.}
\er

Let $\varphi\in\sS(\mR^{2d})$ be a symmetric function and define
$$
X_\varphi(z):=X(\varphi(z-\cdot)).
$$
Then by Lemma \ref{lem:B1} in appendix, $z\mapsto X_\varphi(z)$ has a smooth version.

Let $\varphi,\varphi'\in \sS(\mR^{2d})$ be two symmetric functions.
For $H\in \sS(\mR^{4d})$, define
\begin{align}\label{K30}
(X_{\varphi}\otimes X_{\varphi'})(H):=\int_{\mR^{2d}}\int_{\mR^{2d}}H(z,z')X_{\varphi}(z)X_{\varphi'}(z')\dif z\dif z'.
\end{align}
The following result is easy by the properties of Gaussian fields (see \cite{SV97}) {and we put the proof in Appendix \ref{sec:app}.}
\bl\label{lem:w}
For any $H\in \sS(\mR^{4d})$, it holds that
\begin{align}\label{KKN1}
\bE\Big((X_{\varphi}\otimes X_{\varphi'})(H)\Big)=\int_{\mR^{2d}}\hat H(\zeta,-\zeta)\hat\varphi(\zeta){\hat\varphi'}(\zeta)\mu(\dif \zeta),
\end{align}
and
\begin{align}\label{KKN2}
{\rm Var}\Big((X_{\varphi}\otimes X_{\varphi'})(H)\Big)=2\int_{\mR^{4d}}|({\rm Sym}\,\hat H_{\varphi,\varphi'})(\zeta,\zeta')|^2\mu(\dif \zeta)\mu(\dif\zeta'),
\end{align}
where $\hat H_{\varphi,\varphi'}(\zeta,\zeta'):=\hat H(\zeta,\zeta')\hat\varphi(\zeta)\hat\varphi'(\zeta')$ and
\begin{align}\label{KKN3}
({\rm Sym} \hat H_{\varphi,\varphi'})(\zeta,\zeta'):=( \hat H_{\varphi,\varphi'}(\zeta,\zeta')+\hat H_{\varphi,\varphi'}(\zeta',\zeta))/2.
\end{align}
\el

{If we do Wiener chaos decomposition for $(X_{\varphi}\otimes X_{\varphi'})(H)$ (see \cite[Ch.1]{Nua06}),  $I_0:=\mE\Big((X_{\varphi}\otimes X_{\varphi'})(H)\Big)$ corresponds to the term in the  0th Wiener chaos and $I_2:=(X_{\varphi}\otimes X_{\varphi'})(H)-\mE\Big((X_{\varphi}\otimes X_{\varphi'})(H)\Big)$ gives the term in the second Wiener chaos. }

\br\rm
If $X,Y$ are two independent Gaussian fields with the same spectral measure, then
$\bE((X_{\varphi}\otimes Y_{\varphi'})(H))=0$ and
\begin{align}\label{KKN22}
\bE\big((X_{\varphi}\otimes Y_{\varphi'})(H)\big)^2=\int_{\mR^{2d}}\int_{\mR^{2d}}|\hat H(\zeta,\zeta')|^2|\hat\varphi(\zeta)|^2|\hat\varphi'(\zeta')|^2\mu(\dif \zeta)\mu(\dif\zeta').
\end{align}
\er
\subsection{Examples for (\ref{Stoch01})}\label{sec:pex}
In this subsection we provide three examples for condition \eqref{Stoch01} to illustrate our result.
We need the following simple lemma.
\bl
For  $\beta_1,\beta_2\in[0,d)$ and $\gamma_1,\gamma_2\geq 0$ with
$$
\gamma_1+\beta_1>d, \quad 3\beta_1+\beta_2+\gamma_2>4d,
$$
it holds that
\begin{align}\label{Ho61}
\sup_{\xi'\in\mR^{d}}\int_{\mR^d}\frac{\dif \xi}{|\xi|^{\beta_1}(1+|\xi+\xi'|)^{\gamma_1}}<\infty,
\end{align}
and for $\zeta=(\xi,\eta)\in\mR^{2d}$,
\begin{align}\label{Ho62}
\sup_{\zeta'\in\mR^{2d}}\int_{\mR^{2d}}\frac{\dif \zeta}{|\xi|^{\beta_1}|\eta|^{\beta_2}(1+|\zeta+\zeta'|_a)^{\gamma_2}}<\infty.
\end{align}
\el
\begin{proof}
	For \eqref{Ho61}, we have
	\begin{align*}
	&\int_{\mR^d}\frac{\dif\xi}{|\xi|^{\beta_1}(1+|\xi+\xi'|)^{\gamma_1}}=\left(\int_{|\xi+\xi'|\le|\xi|}+\int_{|\xi+\xi'|>|\xi|}\right)
	\frac{\dif\xi}{|\xi|^{\beta_1}(1+|\xi+\xi'|)^{\gamma_1}}\\
	&\le \int_{|\xi+\xi'|\le|\xi|}\frac{\dif\xi}{|\xi+\xi'|^{\beta_1}(1+|\xi+\xi'|)^{\gamma_1}}+\int_{|\xi+\xi'|>|\xi|}\frac{\dif\xi}{|\xi|^{\beta_1}(1+|\xi|)^{\gamma_1}}\\
	&\le \int_{\mR^d}\frac{\dif\xi}{|\xi+\xi'|^{\beta_1}(1+|\xi+\xi'|)^{\gamma_1}}+\int_{\mR^d}\frac{\dif\xi}{|\xi|^{\beta_1}(1+|\xi|)^{\gamma_1}}
	=2 \int_{\mR^d}\frac{\dif\xi}{|\xi|^{\beta_1}(1+|\xi|)^{\gamma_1}},
	\end{align*}
	which is finite by $\gamma_1+\beta_1>d$ and $\beta_1<d$.
	
	Next we show \eqref{Ho62} by \eqref{Ho61}. Let
	\begin{align*}
	\theta:=\frac{3(d-\beta_1)}{3(d-\beta_1)+(d-\beta_2)}=\frac{3(d-\beta_1)}{4d-3\beta_1-\beta_2}\in(0,1).
	\end{align*}
	Since $\gamma_2>4d-3\beta_1-\beta_2,$ we have
	\begin{align}\label{5pf}
	\beta_1+\theta\gamma_2/3>d,\quad \beta_2+(1-\theta)\gamma_2>d.
	\end{align}
	By $|\xi|^{1/3}\vee|\eta|\le|\zeta|_a$, we have
	\begin{align*}
	\int_{\mR^{2d}}\frac{\dif\zeta}{|\xi|^{\beta_1}|\eta|^{\beta_2}(1+|\zeta+\zeta'|_a)^{\gamma_2}}
	&\leq\int_{\mR^{d}}\frac{\dif\xi}{|\xi|^{\beta_1}(1+|\xi+\xi'|^{1/3})^{\theta\gamma_2}}\\
	&\times\int_{\mR^{d}}\frac{\dif\eta}{|\eta|^{\beta_2}(1+|\eta+\eta'|)^{(1-\theta)\gamma_2}},
	\end{align*}
	which in turn gives \eqref{Ho62} by \eqref{5pf} and \eqref{Ho61}.
\end{proof}
\bx\label{Ex412}\rm
Fix $\beta\in(\frac12,\frac23)$ and $\gamma\in(d-\frac{2}{3}\beta,d)$. Let
$$
\mu(\dif\xi,\dif \eta)=|\xi|^{-\gamma}\dif\xi\delta_0(\dif\eta),
$$
where $\dif\xi$ is the Lebsgue measure on $\mR^d$ and $\delta_0(\dif\eta)$ is the Dirac measure on $\mR^d$ concentrated at $0$.
By \eqref{Ho61}, one sees that \eqref{Stoch01} holds. In this case, it is well known that
for some $c_{d,\gamma}>0$ (see \cite[p117, Lemma 2]{St70}),
\begin{align*}
\hat\mu(z)=\hat\mu(x,v)=c_{d,\gamma}|x|^{\gamma-d},\ \ z=(x,v).
\end{align*}
In particular, for any $f,g\in\sS(\mR^{2d})$,
\begin{align*}
\bE\Big(X(f)X(g)\Big)&=\int_{\mR^{2d}}\hat f(\zeta)\hat g(-\zeta)\mu(\dif\zeta)=\int_{\mR^{2d}}\int_{\mR^{2d}}f(z)g(z')\hat\mu(z-z')\dif z'\dif z\\
&=\int_{\mR^{d}}\int_{\mR^{d}}\left(\int_{\mR^{d}}f(x,v)\dif v\right)\left(\int_{\mR^{d}}g(x',v)\dif v\right)\frac{c_{d,\gamma}
	\dif x\dif x'}{|x-x'|^{d-\gamma}}.
\end{align*}
Fix $\varphi\in \sS(\mR^d)$ with $\int_{\mR^d}\varphi=1$. For any $f\in\sS(\mR^d)$, if we define
\begin{align*}
X_1(f):=X(\widetilde f),\ \ \widetilde f(x,v):=f(x)\varphi(v),
\end{align*}
then for any $f,g\in\sS(\mR^d)$,
\begin{align*}
\bE \Big(X_1(f)X_1(g)\Big)=c_{d,\gamma}\int_{\mR^{d}}\int_{\mR^{d}}f(x)g(x')\frac{\dif x\dif x'}{|x-x'|^{d-\gamma}}=\int_{\mR^d}\hat f(\xi)\hat g(-\xi)\frac{\dif\xi}{|\xi|^\gamma},
\end{align*}
where the right hand side is just the inner product of homogenous Bessel potential space ${\dot\mH}^{-\gamma}$
in $\mR^d$ (see \cite{BCD11}). In particular, $X_1(f)$ can be extended to all $f\in {\dot\mH}^{-\gamma}$.
{This corresponds to the noise independent of $v$ variable.}
Let $d=1$ and define
\begin{align*}
B_\gamma(y):=\Big(X_1(\1_{[0,y]})\1_{y\geq 0}-X_1(\1_{[y,0]})\1_{y<0}\Big)\gamma^{1/2}(1+\gamma)^{1/2}(2c_{d,\gamma})^{-1/2}.
\end{align*}
By the elementary calculation, we have
\begin{align*}
\bE\Big(B_\gamma(y)B_\gamma(y')\Big)=\tfrac{1}{2}(|y|^{1+\gamma}+|y'|^{1+\gamma}-|y-y'|^{1+\gamma}).
\end{align*}
Hence, $B_\gamma(y)$ is a fractional Brownian motion with Hurst parameter $H=\frac{1+\gamma}{2}\in (1-\frac{\beta}{3},1)$,
and for any $g\in\sS(\mR)$,
\begin{align*}
X_1(g)=-\bar c_{d,\gamma}\int_{\mR}g'(y)B_\gamma(y)\dif y.
\end{align*}
In other words,  $X_1=\bar c_{d,\gamma}B'_\gamma$ in the distributional sense.
\ex
\bx\rm
For $\beta\in(\frac12,\frac23)$ and $0\leq \gamma\in(d-2\beta,d)$, let
$$
\mu(\dif\xi,\dif \eta)=|\eta|^{-\gamma}\delta_0(\dif\xi)\dif\eta.
$$
By \eqref{Ho61}, one sees that \eqref{Stoch01} holds.
When $d=1$ and $\gamma=0$, we have
$$
\hat\mu(x,v)=\delta_0(\dif v)
$$
and
\begin{align*}
\bE\Big(X(f)X(g)\Big)=\int_{\mR^d}\left(\int_{\mR^d}f(x,v)\dif x\right)\left(\int_{\mR^d}g(y,v)\dif y\right)\dif v.
\end{align*}
In particular, for $\varphi\in \sS(\mR^d)$ with $\int_{\mR^d}\varphi=1$, if we define
\begin{align*}
X_2(f):=X(\widetilde f),\ \ \widetilde f(x,v):=\varphi(x) f(v)
\end{align*}
then {$X_2$ is independent of $x$} and is a space white noise on $\mR$. {As Example \ref{Ex412}, for general $\gamma\in{[}0,1)$,
	$X_2$ corresponds to {the derivative of} a fractional Brownian motion with Hurst parameter $H=\frac{1+\gamma}{2}\in [\frac12,1)$.}
\ex

\bx\rm
For  $\beta\in(\frac12,\frac23)$ and $\gamma_1,\gamma_2\in[0,d)$ with $3\gamma_1+\gamma_2>4d-2\beta$, let
$$
\mu(\dif\xi,\dif \eta)=|\xi|^{-\gamma_1}|\eta|^{-\gamma_2}\dif \xi\dif \eta,
$$
By \eqref{Ho62}, one sees that \eqref{Stoch01} holds.
When $\gamma_1\gamma_2\ne0$, we have
$$
\hat\mu(x,v)=c_{d,\gamma}|x|^{\gamma_1-d}|v|^{\gamma_2-d}.
$$
When $\gamma_2=0$, since $\beta\in(\frac12,\frac23)$, we have
\begin{align*}
(4d-2\beta)/3<\gamma_1<d\Rightarrow d<2\beta\Rightarrow d=1,
\end{align*}
and
\begin{align*}
\bE\Big(X(f)X(g)\Big)=c_{1,\gamma}\int_{\mR^{3}}f(x,v)g(y,v)\frac{\dif v\dif x\dif y}{|x-y|^{\gamma_1-1}}.
\end{align*}
In particular, one can regard $W$ being white in $v$-direction and colored in $x$-direction. In general $W$ is
the generalized derivative of
a fractional Brownian sheet with $H_i=\frac{\gamma_i+1}{2}$ satisfying $3H_1+H_2>4-\beta$.
\ex

\subsection{Proof of Theorem \ref{StochM}}
Let
\begin{align}\label{PSI}
\psi(\zeta,\zeta'):=\sum_{|i-j|\le1}\phi_i^a(\zeta)\phi_j^a(\zeta'),
\end{align}
where $(\phi^a_j)_{j\geq -1}$ are defined by \eqref{Phj}.

Now we recall some notations used before. Let $z=(x,v)\in\mR^{2d}$ and $\zeta=(\xi,\eta)$. For {$t\in \mR$}, we define
$$
\Gamma_t z:=(x+tv,v),\ \ \hat\Gamma_t \zeta:=(\xi,\eta+t\xi),
$$
and for a function $f$ on $\mR^{2d}$ and $y,z\in\mR^{2d}$,
$$
(\Gamma_t f)(z):=f(\Gamma_t z),\ \ (\tau_{y}f)(z):=f(z-y).
$$
Clearly,
$$
\Gamma_t\Gamma_{-t}z=z,\ \ \<\Gamma_tz,\zeta\>=\<z,\hat\Gamma_t\zeta\>,
$$
and
\begin{align}\label{DK1}
(f*g)(z)=\int_{\mR^{2d}}\tau_y f(z)g(y)\dif y
\end{align}
and
$$
\widehat{\Gamma_t f}(\zeta)=\hat\Gamma_{-t} \hat f(\zeta),\ \ \Gamma_t(f*g)=(\Gamma_tf)*(\Gamma_t g).
$$
Recalling \eqref{PPT},  we have for some $c_0>0$,
\begin{align}\label{RH8}
\begin{split}
\hat p_s(\xi,\eta)=\e^{-s|\eta|^2-s^3|\xi|^2/3-s^2\<\xi,\eta\>}\leq\e^{-c_0(s^3|\xi|^2+s|\eta|^2)}.
\end{split}
\end{align}
Now let $\varphi$ be a smooth probability density function with compact support and symmetric in the variable $v$. For $\eps\in(0,1)$, let
$$
\varphi_\eps(z):=\eps^{-2d}\varphi(z/\eps),\ \ X_\eps(z):=X_{\varphi_\eps}(z)=X(\varphi_\eps(z-\cdot)).
$$

{To verify Theorem \ref{StochM} it suffices to prove $X\in \bC^{-\alpha}_a(\rho^\kappa)$ $\mathbf{P}$-a.s.  and $X\circ \nabla_v\sI X\in C([0,T], \bC^{1-2\alpha}_a(\rho^\kappa))$ $\mathbf{P}$-a.s.. Now we consider them separately. }

(i)\textbf{ Regularity of $X$}. As in \eqref{DHG1}, by the hypercontractivity of Gaussian random variables, for any $\bar\alpha\in(\beta,\beta+1)$, we have
	\begin{align*}
	\bE|\cR^a_jX_\eps(z)-\cR^a_jX(z)|^p&\lesssim\(\bE|\cR^a_jX_\eps(z)-\cR^a_jX(z)|^2\)^{p/2}\\
	&=\left(\int_{\mR^{2d}}|\phi^a_j(\zeta)|^2|\hat \varphi_\eps(\zeta)-1|^2\mu(\dif \zeta)\right)^{p/2}\\
	&\stackrel{\eqref{SpRI}}{\lesssim}
	2^{\bar\alpha pj}\left(\int_{\mR^{2d}}\frac{|\hat \varphi_\eps(\zeta)-1|^2}{(1+|\zeta|_a)^{2\bar\alpha}}\mu(\dif\zeta)\right)^{p/2},
	\end{align*}
	where the implicit constant does not depend on $z$.
	Noting that
	$$
	|\hat \varphi_\eps(\zeta)-1|=|\hat \varphi(\eps\zeta)-1|\lesssim \eps^{(\bar\alpha-\beta)/3}|\zeta|^{\bar\alpha-\beta}_a,
	$$
	 by definition, we have for any $\alpha'>\bar\alpha$ and {$p>4d/\kappa$}
	\begin{align}\label{GX1}
	\bE\|X_\eps-X\|_{\bB^{-\alpha',a}_{p,p}(\varrho^\kappa)}^p
	&=\sum_{j}2^{-\alpha' pj}\int_{\mR^{2d}}\bE|\cR^a_jX_\eps(z)-\cR^a_jX(z)|^p|\varrho^\kappa(z)|^p\dif z\no\\
	&\lesssim\left(\int_{\mR^{2d}}\frac{\eps^{2(\bar\alpha-\beta)/3}}{(1+|\zeta|_a)^{2\beta}}\mu(\dif\zeta)\right)^{p/2}
	\int_{\mR^{2d}}|\varrho^\kappa(z)|^p\dif z,
	\end{align}
	which, by \eqref{Stoch01}, converges to zero as $\eps\to 0$.
	Furthermore, for $\alpha>\bar\alpha$,
	by Besov's embedding Theorem \ref{Embedding}, for $p$ large enough, we have
	\begin{align*}
	\lim_{\eps\to0}\bE\|X_\eps-X\|_{\mathbf{C}_a^{-\alpha}(\varrho^\kappa)}^p=0.
	\end{align*}
(ii) \textbf{Regularity of $X\circ \nabla_v \sI X$.}

{ Since $X$ is independent of $t$, by Lemma \ref{lem:lambda} we only need to show
$$\bE\sup_{0\leq s<t\leq T}\|X\circ \nabla_v \sI X(t)-X\circ \nabla_v \sI X(s)\|_{\bC^{1-2\alpha}_a(\varrho^\kappa)}<\infty.$$
 }
We represent
$\cR^a_\ell(X_\eps\circ \nabla_v \sI X_\eps(t))$ in  terms of $(X_\eps\otimes X_\eps)(H^\ell_t)$ as given in the following lemma.
\bl
For any $t\geq 0$ and $\ell\geq-1$, we have
\begin{align}\label{006}
\cR^a_\ell(X_\eps\circ \nabla_v \sI X_\eps(t))=(X_\eps\otimes X_\eps)(H^\ell_t),
\end{align}
where
$$
H^\ell_t(y,y'):=\sum_{|i-j|\leq 1}\int^t_0\cR^a_\ell(\tau_{y'}\check{\phi^a_i}\cdot\tau_{\Gamma_{-s} y}(\cR^a_j\nabla_v\Gamma_{s}p_{s}))\dif s.
$$
Moreover,  for $\zeta=(\xi,\eta)\in\mR^{2d}$, we have
\begin{align}\label{009}
\widehat{H^\ell_t}(\zeta,\zeta')={-}{\rm i}\int^t_0 \e^{{-}{\rm i} \<\cdot,\hat\Gamma_s\zeta+\zeta'\>}
\phi^a_\ell(\hat\Gamma_s\zeta+\zeta')\psi(\hat\Gamma_{s}\zeta,\zeta')(\eta+s\xi)\hat p_s(\zeta)\dif s,
\end{align}
where $\psi$ is defined by \eqref{PSI}.
\el
\begin{proof}
	By definition, we have
	\begin{align*}
	&\cR^a_i X_\eps\cdot (\cR^a_j\nabla_v \sI X_\eps)(t)
	=\int^t_0(\check{\phi^a_i}*X_\eps)\cdot (\cR^a_j\nabla_v\Gamma_{s}p_{s})*(\Gamma_{s}X_\eps)\dif s\\
	&\quad=
	\int_{\mR^{2d}}\int_{\mR^{2d}}\left(\int^t_0\tau_{y'}\check{\phi^a_i}\cdot\tau_{\Gamma_{-s} y}(\cR^a_j\nabla_v\Gamma_{s}p_{s})\dif s\right)
	X_\eps(y) X_\eps(y')\dif y\dif y'.
	\end{align*}
	which implies \eqref{006}.
	For \eqref{009}, letting $h:=\cR^a_j(\nabla_v\Gamma_{s}p_{s})$, by easy calculations, we have
	\begin{align*}
	&\int_{\mR^{2d}}\int_{\mR^{2d}}\e^{{-}{\rm i}(\zeta\cdot y+\zeta'\cdot y')}\check\phi^a_\ell*(\tau_{y'}
	\check{\phi^a_i}\cdot\tau_{\Gamma_{-s} y}h)\dif y\dif y'\\
	&\quad=\e^{{-}{\rm i} \<\cdot,\hat\Gamma_s\zeta+\zeta'\>}\phi^a_\ell(\hat\Gamma_{s}\zeta+\zeta') \phi^a_i(\zeta')\hat h({-}\hat\Gamma_{s}\zeta),
	\end{align*}
	where for $\zeta=(\xi,\eta)\in\mR^{2d}$,
	$$
	\hat h(\zeta)=\phi^a_j(\zeta)({\rm i}\eta)(\hat \Gamma_{-s}\hat p_s)(\zeta).
	$$
	Thus we obtain \eqref{009} by $\hat\Gamma_{-s}\hat\Gamma_s\zeta=\zeta$.
\end{proof}
\br\rm
Notice that for each $y,y'\in\mR^{2d}$, $H^\ell_t(y,y')$ is a $\mR^d$-valued function of $z$.
In expressions \eqref{006} and \eqref{009}, we have suppressed  the variable $z$ for simplicity.
Without further declaration, we also use such a convention below.
\er
For simplicity of notations, we write
	$$
	M^\eps_t(z):=(X_\eps\circ\nabla_v\sI X_\eps(t))(z)
	$$
	and
	\begin{align}\label{LL18}
	G_{t,s}^{\eps,\eps'}(z):=M^\eps_t(z)-M^{\eps'}_t(z)-M^\eps_s(z)+M^{\eps'}_s(z).
	\end{align}
	Below we drop the variable $z$. {It is easy to see $M_t^\eps=\E M_t^\eps+M_t^\eps-\E M_t^\eps$ as the Wiener chaos decomposition for $M_t^\eps$ with $\E M_t^\eps$ in the $0$th  Wiener chaos and $M_t^\eps-\E M_t^\eps$ in the second  Wiener chaos. In the following we consider them separately.}

\textbf{Terms in the $0$th  Wiener chaos} {First we have the following estimates for  the terms in the $0$th  Wiener chaos. This terms are not zero as the classical case. After subtracting  formally divergent terms (see $\cJ^t_{22,\ell}$ below) which are zero by symmetry, the terms in the $0$th  Wiener chaos converge in the corresponding spaces.}
Note that by \eqref{006},
	\begin{align}\label{LL28}
	\cR^a_\ell M^\eps_t=(X_\eps\otimes X_\eps)(H_t^\ell).
	\end{align}
	and by  \eqref{KKN1},
	$$
	\cR^a_\ell \bE M^\eps_t=\bE \cR^a_\ell M^\eps_t
	=\int_{\mR^{2d}}\widehat {H^\ell_t}(\zeta,-\zeta)\hat\varphi^2_\eps(\zeta)\mu(\dif \zeta)=:\Lambda^{\ell,\eps}_t.
	$$
This corresponds to the zeroth  Wiener chaos of random field
$(X_\eps\otimes X_\eps)(H^\ell_t)$.

	By \eqref{009} we make the following decomposition:
	\begin{align*}
	\widehat{H^\ell_t}(\zeta,-\zeta)&={-}{\rm i}\int^t_0 \e^{{-}{\rm i} \<\cdot,\hat\Gamma_s\zeta-\zeta\>}
	\phi^a_\ell(\hat\Gamma_s\zeta-\zeta)\Big(\psi(\hat\Gamma_{s}\zeta,-\zeta)-\psi(\zeta,-\zeta)\Big)\eta\hat p_s(\zeta)\dif s\\
	&\quad{-}{\rm i}\int^t_0 \e^{{-}{\rm i} \<\cdot,\hat\Gamma_s\zeta-\zeta\>}
	\phi^a_\ell(\hat\Gamma_s\zeta-\zeta)\psi(\zeta,-\zeta)\eta\hat p_s(\zeta)\dif s\\
	&\quad{-}{\rm i}\int^t_0 \e^{{-}{\rm i} \<\cdot,\hat\Gamma_s\zeta-\zeta\>}
	\phi^a_\ell(\hat\Gamma_s\zeta-\zeta)\psi(\hat\Gamma_{s}\zeta,-\zeta)s\xi\,\hat p_s(\zeta)\dif s\\
	&=:\cJ^t_{1,\ell}(\zeta)+\cJ^t_{2,\ell}(\zeta)+\cJ^t_{3,\ell}(\zeta).
	\end{align*}
	For $\cJ^t_{1,\ell}(\zeta)$, noting that for $\zeta=(\xi,\eta)$,
	\begin{align}\label{JJE}
	\hat\Gamma_s\zeta-\zeta=(0,s\xi),
	\end{align}
	by \eqref{SpRI} and \eqref{SpL5} with $ \gamma=2\alpha-1$, we have
	\begin{align*}
	\|\cJ^t_{1,\ell}(\zeta)\|_{L^\infty}
	&\leq\int^t_0 |\phi^a_\ell(\hat\Gamma_s\zeta-\zeta)|\,|\psi(\hat\Gamma_{s}\zeta,-\zeta)-\psi(\zeta,-\zeta)|\,|\eta|\,\hat p_s(\zeta)\dif s\\
	&\lesssim 2^{(2\alpha-1)\ell}\int^t_0 (1+|s\xi|)^{1-2\alpha}\frac{|s\xi|^{2\alpha-1}}{(1+|\zeta|_a)^{2\alpha-1}}|\eta|\hat p_s(\zeta)\dif s\\
	&\lesssim 2^{(2\alpha-1)\ell}(1+|\zeta|_a)^{2-2\alpha}\int^t_0\hat p_s(\zeta)\dif s\\
	&\!\!\!\!\!\!\!\!\!\!\!\stackrel{\eqref{RH8} ,\eqref{UL6}}{\lesssim}2^{(2\alpha-1)\ell}(1+|\zeta|_a)^{-2\alpha}.
	\end{align*}
	For $\cJ^t_{3,\ell}(\zeta)$, since $|\psi|\lesssim 1$, by \eqref{JJE}, we have
	\begin{align*}
	\|\cJ^t_{3,\ell}(\zeta)\|_{L^\infty}&\lesssim 2^{(2\alpha-1)\ell}\int^t_0 (1+|s\xi|)^{1-2\alpha}|s\xi|\hat p_s(\zeta)\dif s\\
	&\lesssim 2^{(2\alpha-1)\ell}\int^t_0 |s\xi|^{2-2\alpha}\hat p_s(\zeta)\dif s\\
	&\!\!\!\!\!\!\!\!\!\!\!\stackrel{\eqref{RH8} ,\eqref{UL6}}{\lesssim}2^{(2\alpha-1)\ell}(1+|\zeta|_a)^{-2\alpha}.
	\end{align*}
	For $\cJ^t_{2,\ell}(\zeta)$, by we can write
	\begin{align*}
	\cJ^t_{2,\ell}(\zeta)&={-}{\rm i}\int^t_0 \e^{{-}{\rm i} \<\cdot,\hat\Gamma_s\zeta-\zeta\>}
	\phi^a_\ell(\hat\Gamma_s\zeta-\zeta)\psi(\zeta,-\zeta)\eta \e^{-s|\eta|^2-s^3|\xi|^2/3}(\e^{-s^2\<\xi,\eta\>}-1)\dif s\\
	&\qquad{-}{\rm i}\int^t_0 \e^{{\rm i} \<\cdot,\hat\Gamma_s\zeta-\zeta\>}
	\phi^a_\ell(\hat\Gamma_s\zeta-\zeta)\psi(\zeta,-\zeta)\eta \e^{-s|\eta|^2-s^3|\xi|^2/3}\dif s\\
	&=:\cJ^t_{21,\ell}(\zeta)+\cJ^t_{22,\ell}(\zeta).
	\end{align*}
	For $\cJ^t_{21,\ell}(\zeta)$, noting that
	$$
	|\e^{-s^2\<\xi,\eta\>}-1|\leq |s\xi|\, |s\eta|\,\e^{s^2|\xi|\,|\eta|},
	$$
	by \eqref{RH8} and \eqref{UL6}, we have
	\begin{align*}
	\|\cJ^t_{21,\ell}(\zeta)\|_{L^\infty}&\lesssim 2^{(2\alpha-1)\ell}\int^t_0 (1+|s\xi|)^{1-2\alpha}|s\xi||s\eta|{|\eta|}\,\e^{-c_0(s^3|\xi|^2+s|\eta|^2)}\dif s\\
	&\lesssim 2^{(2\alpha-1)\ell}\int^t_0 |s\xi|^{2-2\alpha}|s\eta|{|\eta|}\,\e^{-c_0(s^3|\xi|^2+s|\eta|^2)}\dif s\\
	&\lesssim 2^{(2\alpha-1)\ell}|\xi|^{2-2\alpha}|\eta|^2\int^t_0 s^{3-2\alpha}\e^{-c_0(s^3|\xi|^2+s|\eta|^2)}\dif s\\
	&\!\!\!\stackrel{\eqref{UL6}}{\lesssim}2^{(2\alpha-1)\ell}(1+|\zeta|_a)^{-2\alpha}.
	\end{align*}
	On the other hand, by \eqref{JJE}, one sees that
	$$
	\cJ^t_{22,\ell}(\xi,-\eta)=-\cJ^t_{22,\ell}(\xi,\eta).
	$$
	Since $\mu(\dif \xi,-\dif\eta)=\mu(\dif \xi,\dif\eta)$ {and $\varphi_\eps$ is symmetric w.r.t. $v$ variable}, we have
	\begin{align}\label{CA1}
	\int_{\mR^{2d}}\cJ^t_{22,\ell}(\zeta)\hat\varphi^2_\eps(\zeta)\mu(\dif \zeta)\equiv0.
	\end{align}
	Thus, we get
	\begin{align*}
	\Lambda^{\ell,\eps}_t=\int_{\mR^{2d}}\Big(\cJ^t_{1,\ell}(\zeta)+\cJ^t_{21,\ell}(\zeta)+\cJ^t_{3,\ell}(\zeta)\Big)\hat\varphi^2_\eps(\zeta)\mu(\dif \zeta),
	\end{align*}
	and
	\begin{align*}
	\|\Lambda^{\ell,\eps}_t-\Lambda^{\ell,\eps'}_{t}\|_{L^\infty}&\lesssim
	\int_{\mR^{2d}}\Big(\|\cJ^t_{1,\ell}(\zeta)\|_{L^\infty}+\|\cJ^t_{21,\ell}(\zeta)\|_{L^\infty}+\|\cJ^t_{3,\ell}(\zeta)\|_{L^\infty}\Big)\\
	&\qquad\qquad\times|\hat\varphi^2_\eps(\zeta)-\hat\varphi^2_{\eps'}(\zeta)|\mu(\dif \zeta)\\
	&\lesssim
	2^{(2\alpha-1)\ell}\int_{\mR^{2d}}(1+|\zeta|_a)^{-2\alpha}|\hat\varphi^2_\eps(\zeta)-\hat\varphi^2_{\eps'}(\zeta)|\mu(\dif \zeta).
	\end{align*}
	By the dominated convergence theorem, we obtain
\begin{align}\label{JJR}
\lim_{\eps,\eps'\to 0}\sup_{\ell\geq -1}\sup_{t\in[0,T]}2^{(1-2\alpha)\ell}\|\Lambda^{\ell,\eps}_t-\Lambda^{\ell,\eps'}_{t}\|_{L^\infty}=0,
\end{align}
where  the norm $\|\cdot\|_{L^\infty}$ is with respect to variable $z$. 
Thus,  we have
	\begin{align*}
	\lim_{\eps,\eps'\to0}\sup_{t\in[0,T]}\big\|\bE M^\eps_t-\bE M^{\eps'}_t\big\|_{\bC^{1-2\alpha}_a}=0.
	\end{align*}

\textbf{Terms in the second  Wiener chaos}
By Kolmogorov's continuity criterion and Besov's embedding Theorem \ref{Embedding}, it suffices to show that for some $\delta>0$, and any $\alpha>\beta$ and $p\geq 2$,
	$$
	\lim_{\eps,\eps'\to0}\sup_{0\leq s<t\leq T}(t-s)^{-\delta p}
	\bE\Big(\|G_{t,s}^{\eps,\eps'}-\bE G_{t,s}^{\eps,\eps'}\|_{\bB^{1-2\alpha,a}_{p,p}(\varrho^\kappa)}^p\Big)=0.
	$$
	Since $G_{t,s}^{\eps,\eps'}-\bE G_{t,s}^{\eps,\eps'}$ belongs to the second  Wiener chaos space, as in \eqref{GX1}, we only need to show that
	\begin{align}\label{ZX7}
	\lim_{\eps,\eps'\to0}\sup_{0\leq s<t\leq T}
	\sup_{\ell\geq -1}(t-s)^{-\delta} 2^{(2-4\alpha)\ell}\|{\rm Var}(\cR^a_\ell G_{t,s}^{\eps,\eps'})\|_{L^\infty}=0.
	\end{align}
	Noting that by \eqref{LL18} and \eqref{LL28},
	\begin{align*}
	\cR^a_\ell G^{\eps,\eps'}_{t,s}
	&=(X_\eps\otimes X_\eps)(H_t^\ell-H_s^\ell)-(X_{\eps'}\otimes X_{\eps'})(H_t^\ell-H_s^\ell)\\
	&=(X_{\varphi_\eps-\varphi_{\eps'}}\otimes X_{\varphi_\eps})(H_t^\ell-H_s^\ell)+(X_{\varphi_{\eps'}}\otimes X_{\varphi_\eps-\varphi_{\eps'}})(H_t^\ell-H_s^\ell),
	\end{align*}
	by \eqref{KKN2}, we have
	\begin{align}
	{\rm Var}(\cR^a_\ell G_{t,s}^{\eps,\eps'})
	&\leq 2{\rm Var}((X_{\varphi_\eps-\varphi_{\eps'}}\otimes X_{\varphi_\eps})(H_t^\ell-H_s^\ell))\no\\
	&\quad+2{\rm Var}((X_{\varphi_{\eps'}}\otimes X_{\varphi_\eps-\varphi_{\eps'}})(H_t^\ell-H_s^\ell))\no\\
	&=4\int_{\mR^{2d}}\int_{\mR^{2d}}|{\rm Sym}\big((\widehat{H^\ell_t}-\widehat{H^{\ell}_s})K^{(1)}_{\eps,\eps'}\big)(\zeta,\zeta')|^2\mu(\dif\zeta)\mu(\dif \zeta')\no\\
	&\quad+4\int_{\mR^{2d}}\int_{\mR^{2d}}|{\rm Sym}\big((\widehat{H^\ell_t}-\widehat{H^{\ell}_s})K^{(2)}_{\eps,\eps'}\big)(\zeta,\zeta')|^2\no\\
	&\leq{4}\int_{\mR^{2d}}\int_{\mR^{2d}}|\big((\widehat{H^\ell_t}-\widehat{H^{\ell}_s})K^{(1)}_{\eps,\eps'}\big)(\zeta,\zeta')|^2\mu(\dif\zeta)\mu(\dif \zeta')\no\\
	&\quad+{4}\int_{\mR^{2d}}\int_{\mR^{2d}}|\big((\widehat{H^\ell_t}-\widehat{H^{\ell}_s})K^{(2)}_{\eps,\eps'}\big)(\zeta,\zeta')|^2
	\mu(\dif\zeta)\mu(\dif \zeta'),\label{ZX8}
	\end{align}
	where
	$$
	K^{(1)}_{\eps,\eps'}(\zeta,\zeta'):=(\hat\varphi_\eps(\zeta)-\hat\varphi_{\eps'}(\zeta))\hat\varphi_{\eps}(\zeta')
	$$
	and
	$$
	K^{(2)}_{\eps,\eps'}(\zeta,\zeta'):=\hat\varphi_{\eps'}(\zeta)(\hat\varphi_\eps(\zeta')-\hat\varphi_{\eps'}(\zeta')).
	$$
For any $\theta\in(0,1)$, we have
	\begin{align*}
	|K^{(1)}_{\eps,\eps'}(\zeta,\zeta')|\leq|(\hat\varphi_\eps-\hat\varphi_{\eps'})(\zeta)|\lesssim|\eps-\eps'|^{\theta/3}|\zeta|_a^\theta
	\end{align*}
	and
	\begin{align*}
	|K^{(2)}_{\eps,\eps'}(\zeta,\zeta')|\leq|(\hat\varphi_\eps-\hat\varphi_{\eps'})(\zeta')|
	\lesssim|\eps-\eps'|^{\theta/3}|\zeta'|_a^\theta.
	\end{align*}
{Moreover, by \eqref{009} we clearly have
	\begin{align*}
	\|(\widehat{H^\ell_t}-\widehat{H^{\ell}_s})K^{(2)}_{\eps,\eps'}(\zeta,\zeta')\|_{L^\infty}\lesssim
	|\eps-\eps'|^{\theta/3}\int^t_s\Phi^\ell_r(\zeta,\zeta')|\zeta'|_a^\theta|\eta+r\xi|\,\hat p_r(\zeta)\dif r,
	\end{align*}
	and
	$$
	\Phi^\ell_r(\zeta,\zeta'):=|\phi^a_\ell(\hat\Gamma_r\zeta+\zeta')|\,|\psi(\hat\Gamma_{r}\zeta,\zeta')|.
	$$
	Let $\sigma,\gamma\geq 0$ with $\sigma+\gamma=2\beta$. Noting that by \eqref{SpRI}, \eqref{SpR0} and \eqref{StoRk1},
	\begin{align*}
	\|\Phi^\ell_r(\zeta,\cdot)|\cdot|_a^\theta\|_{L^2(\mu)}^2&\lesssim\int_{\mR^{2d}}
	\frac{2^{{\sigma}\ell}(1+|\hat\Gamma_s\zeta|_a)^{\gamma+2\theta}}{(1+|\hat\Gamma_s\zeta+\zeta'|_a)^{\sigma}(1+|\zeta'|_a)^{\gamma}}\mu(\dif\zeta')
	\lesssim 2^{{\sigma}\ell}(1+|\hat\Gamma_s\zeta|_a)^{{\gamma+2\theta}},
	\end{align*}
	we have by Minkowski's inequality,
	\begin{align}\label{SX9}
	\Big\|\big\|(\widehat{H^\ell_t}-\widehat{H^{\ell}_s})K^{(2)}_{\eps,\eps'}(\zeta,\cdot)\big\|_{L^\infty}\Big\|_{L^2(\mu)}
	&\lesssim\int^t_s|\eps-\eps'|^{\theta/3}\|\Phi^\ell_r(\zeta,\cdot)\|_{L^2(\mu)}\,|\eta+r\xi|\,\hat p_r(\zeta)\dif r\no\\
	&\lesssim |\eps-\eps'|^{\theta/3}2^{\frac{\sigma\ell}2}\int^t_s(1+|\hat\Gamma_r\zeta|_a)^{{\frac{\gamma+2\theta}2}}|\eta+r\xi|\,\hat p_r(\zeta)\dif r.
	\end{align}
	Since $|\eta|\vee|\xi|^{1/3}\leq|\zeta|_a$, by $\hat\Gamma_r\zeta=(\xi,\eta+r\xi)$, we have
	\begin{align*}
	&(1+|\hat\Gamma_r\zeta|_a)^{\frac{\gamma+2\theta}2}|\eta+r\xi|\\
	&\lesssim (1+|\zeta|_a)^{\frac{{\gamma+2\theta}}2}|\eta|+(r|\xi|)^{\frac{{\gamma+2\theta}}2}|\eta|+(1+|\zeta|_a)^{\frac{{\gamma+2\theta}}2}|r\xi|+(r|\xi|)^{\frac{{\gamma+2\theta}}2+1}\\
	&\lesssim (1+|\zeta|_a)^{\frac{{\gamma+2\theta}}2+1}+r^{\frac{\gamma+2\theta} 2}|\zeta|_a^{\frac{3(\gamma+2\theta)}2+1}
	+r(1+|\zeta|_a)^{\frac{{\gamma+2\theta}}2+3}+r^{\frac{{\gamma+2\theta}}2+1}|\zeta|_a^{^{\frac{3(\gamma+2\theta)}2+3}}.
	\end{align*}
	If we choose $\sigma=4\alpha-2$ for some $\alpha>\beta$, then
	$$
	\tfrac\gamma 2-1=\beta-2\alpha<-\beta.
	$$
	Thus, by \eqref{RH8} and \eqref{UL6}, for $\theta$ small enough there is a $\delta>0$ such that for all $0\leq s<t\leq T$ and $\zeta=(\xi,\eta)\in\mR^{2d}$,
	$$
	\int^t_s(1+|\hat\Gamma_r\zeta|_a)^{\frac{\gamma+\theta}2}|\eta+r\xi|\,\hat p_r(\zeta)\dif r
	\lesssim (t-s)^{\delta/2}(1+|\zeta|_a)^{-\beta}.
	$$
	Substituting this into \eqref{SX9}, we get
	$$
	\Big\|\big\|(\widehat{H^\ell_t}-\widehat{H^{\ell}_s})K^{(2)}_{\eps,\eps'}(\zeta,\cdot)\big\|_{L^\infty}\Big\|^2_{L^2(\mu)}\lesssim 	|\eps-\eps'|^{2\theta/3}2^{(4\alpha-2)\ell} (t-s)^{\delta}(1+|\zeta|_a)^{-2\beta}.
	$$
For the term containing $K^{(1)}$ we have the similar bounds.
	Substituting these into the left hand side of \eqref{ZX8}, we obtain the regularity of the term in the second Wiener chaos.
Thus we complete the proof.}





\begin{appendix}
\renewcommand{\thetable}{A\arabic{table}}
\numberwithin{equation}{section}

\section{Characterizations for $\bB^{s,a}_{p,q}(\rho)$}\label{AnWB}

In this appendix, we provide a detailed proof for Theorem \ref{Th26}.
First of all, we prepare two useful lemmas for later use.

\bl
For any $\alpha>0$, there is a constant $C=C(d,a,m,\alpha)>0$ such that for all $\lambda>0$,
\begin{align}\label{CoVA}
\int_{|h|_a{\leq}\lambda}|h|_a^{\alpha-a\cdot m}\dif h\lesssim_C\lambda^\alpha,\ \ \int_{|h|_a>\lambda}|h|_a^{-\alpha-a\cdot m}\dif h\lesssim_C \lambda^{-\alpha}.
\end{align}
\el
\begin{proof}
Let $h=(h_1,..,h_n)\in\mR^N$ with $h_i\in\mR^{m_i}$. Define a transform $h\to\widetilde h$ by
\begin{align*}
\widetilde h:=(\widetilde h_1,\cdots,\widetilde h_n),\ \ \widetilde{h}_i:=|h_i|^{\frac{1}{a_i}-1}h_i.
\end{align*}
Clearly, for each $i=1,\cdots,n$,
$$
|h_i|=|\widetilde{h}_i|^{a_i},\ \ h_i=|\widetilde{h}_i|^{a_i-1}\widetilde{h}_i
$$
and
\begin{align*}
\big|\det(\p h_i/\p \widetilde{h}_i)\big|\le {a^{m_i}_i}|\widetilde{h}_i|^{a_im_i-m_i}\le {a^{m_i}_i}|\widetilde{h}|_1^{a_im_i-m_i},
\end{align*}
where $\p h_i/\p \widetilde{h}_i$ stands for the Jacobian matrix of the inverse transform $\widetilde h_i\to h_i$, and
$|\widetilde h|_1:=\sum_{i=1}^n|\widetilde h_i|$. Thus by the change of variable,
\begin{align*}
\int_{|h|_a{\leq}\lambda}|h|_a^{\alpha-a\cdot m}\dif h&=\int_{|\widetilde{h}|_1{\leq}\lambda}|\widetilde{h}|_1^{\alpha-a\cdot m}\Pi_{i=1}^n\big|\det(\p h_i/\p \widetilde{h}_i)\big|\dif \widetilde{h}\\
&\le {\prod _{i=1}^na^{m_i}_i} \int_{|\widetilde{h}|_1{\leq}\lambda}|\widetilde{h}|_1^{\alpha-N}\dif \widetilde{h}\lesssim\lambda^\alpha,
\end{align*}
where $N=m_1+\cdots+m_n$, and
\begin{align*}
 \int_{|h|_a>\lambda}|{h}|_a^{-\alpha-a\cdot m}\dif{h}\le {\prod _{i=1}^na^{m_i}_i}\int_{|\widetilde{h}|_1>\lambda}|\widetilde{h}|_1^{-\alpha-N}\dif \widetilde{h}
 \lesssim\lambda^{-\alpha}.
\end{align*}
The proof is complete.
\end{proof}

By the following lemma we can estimate the norm in $\bB^{s,a}_{p,q}(\rho)$ by duality.
\bl\label{DuaA}
Let $\rho\in\sW$ with $\rho^{-1}\in\sW$, $s\in\mR$ and $p,q,p',q'\in[1,\infty]$ with $1/p+1/p'=1/q+1/q'=1$.
\begin{enumerate}[(i)]
\item For any $\varphi\in\sS$ and $f\in\bB^{s,a}_{p,q}(\rho)$, it holds that
$$
|\<f,\varphi\>|\le \|f\|_{\bB^{s,a}_{p,q}(\rho)}\|\varphi\|_{\bB^{-s,a}_{p',q'}(\rho^{-1})} .
$$
\item There is a constant $C=C(\rho,d,s,p,q)>0$ such that for any $f\in\bB^{s,a}_{p,q}(\rho)$,
$$
\|f\|_{\bB^{s,a}_{p,q}(\rho)}\le C\sup_{\varphi\in\sS}\<f,\varphi\>/\|\varphi\|_{\bB^{-s,a}_{p',q'}(\rho^{-1})}.
$$
\end{enumerate}
\el
\begin{proof}
(i) By \eqref{KJ2} and H\"older's inequality, we have
\begin{align*}
\<f,\varphi\>=\sum_{j\geq -1}\<\cR_j^af,\widetilde{\cR}_j^a\varphi\>&\leq\sum_{j\geq -1}\|\cR_j^af\|_{L^p(\rho)}\|\widetilde{\cR}_j^a\varphi\|_{L^{p'}(\rho^{-1})}\\
&\leq\|f\|_{\bB^{s,a}_{p,q}(\rho)}\|\varphi\|_{\bB^{-s,a}_{p',q'}(\rho^{-1})} .
\end{align*}
(ii) We follow the proof in \cite{BCD11}.
For $M\in\mN$, let
$$
U^{q'}_M:=\Big\{(c_j)_{j\in\mN}: \sum_{j\leq M}|c_j|^{q'}\leq 1,\ c_j=0,\ \ j>M\Big\}.
$$
By the definition of $\bB^{s,a}_{p,q}(\rho)$, we have
\begin{align*}
\|f\|_{\bB^{s,a}_{p,q}(\rho)}&=\lim_{M\to\infty}\left(\sum_{j\leq M}2^{jsq}\|\cR_j^a f\|^q_{L^p(\rho)}\right)^{1/q}\\
&=\lim_{M\to\infty}\sup_{(c_j)\in U^{q'}_M}\sum_{j\leq M}c_j2^{js}\|\cR_j^a f\|_{L^p(\rho)}.
\end{align*}
Fix $\eps>0$ and $(c_j)\in U^{q'}_M$. Since
\begin{align*}
\|g\|_{L^p}=\sup_{h\in \sS}\<g,h\>/\|h\|_{L^{p'}},
\end{align*}
for any $j\le M$, there is a $\psi_j\in\sS$ with $\|\psi_j\|_{L^{p'}}=1$ such that
\begin{align*}
\|\cR_j^a f\|_{L^p(\rho)}&\le \int_{\mR^N}\rho(x)\cR_j^af(x)\psi_j(x)\dif x+\frac{\eps 2^{-js}}{(|c_j|+1)(j^2+1)}\\
&=\int_{\mR^N}f(x)\cR_j^a(\rho\psi_j)(x)\dif x+\frac{\eps 2^{-js}}{(|c_j|+1)(j^2+1)}.
\end{align*}
Now,  if we define $\varphi^{(c_j)}_M\in\sS$ by
 \begin{align*}
\varphi^{(c_j)}_M(x):=\sum_{j\le M}c_j 2^{js}\cR_j^a(\rho\psi_j)(x),
\end{align*}
then
 \begin{align}\label{PfApp2}
\|f\|_{\bB^{s,a}_{p,q}(\rho)}\leq\lim_{M\to\infty}\sup_{(c_j)\in U^{q'}_M}\<f,\varphi^{(c_j)}_M\>+\sum_{j\geq -1}\frac{\eps}{j^2+1}.
\end{align}
Note that for by \eqref{Crapp},
 \begin{align}\label{PfApp1}
\|\varphi^{(c_j)}_M\|^{q'}_{\bB^{-s,a}_{p',q'}(\rho^{-1})}&=\sum_{k\geq -1}2^{-kq's}\left\|\sum_{j\leq M, |j-k|\leq 3}c_j 2^{js}\cR^a_k\cR_j^a(\rho\psi_j)\right\|^{q'}_{L^{p'}(\rho^{-1})}\no\\
&\lesssim \sum_{j\leq M}c^{q'}_j\|\rho\psi_j\|^{q'}_{L^{p'}(\rho^{-1})}=\sum_{j\leq M}c^{q'}_j\|\psi_j\|^{q'}_{L^{p'}}=\sum_{j\leq M}c^{q'}_j\leq 1.
\end{align}
Hence, by \eqref{PfApp2} and \eqref{PfApp1},
\begin{align*}
\|f\|_{\bB^{s,a}_{p,q}(\rho)}\le C\sup_{\varphi\in\sS}\<f,\varphi\>/\|\varphi\|_{\bB^{-s,a}_{p',q'}(\rho^{-1})}+ \sum_{j\geq -1}\frac{\eps}{j^2+1}.
\end{align*}
The proof is complete by letting $\eps\to0$.
\end{proof}

Now we can give the proof of Theorem \ref{Th26}.

\begin{proof}[Proof of Theorem \ref{Th26}]
(i) {In this step we prove
\begin{align}\label{fbound}
\|f\|_{\widetilde\bB^{s,a}_{p,q}(\rho)}\lesssim \|f\|_{\bB^{s,a}_{p,q}(\rho)}.
\end{align}} For simplicity, we set $M:=[s]+1$. Note that by \eqref{BETA},
$$
\|\delta_h\cR^a_jf\|_{L^p(\rho)}\lesssim(1+|h|_a^\kappa)\sum_{i=1}^n\|\delta_{h_{i}}\cR^a_jf\|_{L^p(\rho)},
$$
where for $h=(h_1,\cdots,h_n)$ and $x=(x_1,\cdots,x_n)$,
$$
\delta_{h_i}f(x):=f(\cdots, x_{i-1}, x_i+h_i,x_{i+1},\cdots)-f(\cdots,x_{i-1}, x_i, x_{i+1}\cdots).
$$
By induction, one sees that
$$
\|\delta^{(M)}_h\cR^a_jf\|_{L^p(\rho)}\lesssim(1+|h|_a^{M\kappa})
\sum_{i_1=1}^n\cdots\sum_{i_M=1}^n\|\delta_{h_{i_1}}\cdots\delta_{h_{i_M}}\cR^a_jf\|_{L^p(\rho)}.
$$
Let $|h|_a\leq 1$. By \eqref{BETA} and Bernstein's inequality \eqref{Ber}, we have
\begin{align*}
\|\delta_{h_{i_1}}\cdots\delta_{h_{i_M}}\cR^a_jf\|_{L^p(\rho)}&\lesssim
|h_{i_1}|\|\nabla_{x_{i_1}}\delta_{h_{i_2}}\cdots\delta_{h_{i_M}}\cR^a_jf\|_{L^p(\rho)}\\
&\lesssim\cdots\cdots\cdots\\
&\lesssim
|h_{i_1}|\cdots|h_{i_M}|\,\|\nabla_{x_{i_1}}\cdots \nabla_{x_{i_M}}\cR^a_j f\|_{L^p(\rho)}\\
&\lesssim|h_{i_1}|2^{a_{i_1}j}\cdots|h_{i_M}|2^{a_{i_M}j}\|\cR^a_j f\|_{L^p(\rho)}\\
&\lesssim (2^j|h|_a)^{a_{i_1}+\cdots+a_{i_M}}\|\cR^a_jf\|_{L^p(\rho)}.
\end{align*}
Moreover, by \eqref{BETA}, we clearly have
$$
\|\delta_{h_{i_1}}\cdots\delta_{h_{i_M}}\cR^a_jf\|_{L^p(\rho)}\lesssim \|\cR^a_j f\|_{L^p(\rho)}.
$$
Hence,
\begin{align*}
\|\delta_{h_{i_1}}\cdots\delta_{h_{i_M}}\cR^a_jf\|_{L^p(\rho)}&\lesssim ((2^j|h|_a)^{a_{i_1}+\cdots+a_{i_M}}\wedge 1)\|\cR^a_j f\|_{L^p(\rho)}\\
&\lesssim ((2^j|h|_a)^{M} \wedge 1)\|\cR^a_jf\|_{L^p(\rho)},
\end{align*}
where the second inequality is due to $a_{i_1}+\cdots+a_{i_M}\geq M$.
Thus we obtain
\begin{align}\label{ZX1}
\|\delta^{(M)}_h\cR^a_jf\|_{L^p(\rho)}&\lesssim ((2^j|h|_a)^{M} \wedge 1)\|\cR^a_jf\|_{L^p(\rho)}
=((2^j|h|_a)^{M} \wedge 1)2^{-sj}c_{j},
\end{align}
where
$$
c_{j}:=2^{sj}\|\cR^a_jf\|_{L^p(\rho)}.
$$
For $q=\infty$, we have
\begin{align*}
\|\delta^{(M)}_hf\|_{L^p(\rho)}&\leq\sum_{j}\|\delta^{(M)}_h\cR^a_jf\|_{L^p(\rho)}
\lesssim\sum_{j}((2^j|h|_a)^{M} \wedge 1)2^{-sj} c_{j}\lesssim |h|_a^s\|f\|_{\bB^{s,a}_{p,\infty}(\rho)}.
\end{align*}
Next we assume $q\in[1,\infty)$. For $h\in\mR^N$ with $|h|_a\leq 1$, we choose $j_{h}\in\mN$ such that
\begin{align}\label{BC1}
|h|_a^{-1}\leq 2^{j_{h}}\leq 2|h|_a^{-1}.
\end{align}
Then by \eqref{ZX1},
\begin{align*}
\|\delta^{(M)}_hf\|_{L^p(\rho)}&\leq\sum_{j\geq -1}\|\delta^{(M)}_h\cR^a_j f\|_{L^p(\rho)}
\lesssim\sum_{j\geq -1}((2^j|h|_a)^{M}\wedge 1)2^{-sj}c_{j}\\
&\leq |h|_a^{M} \sum_{j<j_{h}}2^{(M-s)j}c_{j}+\sum_{j\geq j_{h}}2^{-sj}c_{j}=:I_1(h)+I_2(h).
\end{align*}
For $I_1(h)$, by H\"older's inequality, we have
\begin{align*}
I^q_1(h)&\leq |h|_a^{qM} \left(\sum_{j<j_{h}}2^{(M-s)j}\right)^{q-1}\sum_{j<j_{h}}2^{(M-s)j}c^q_{j}\\
&\leq |h|_a^{M-s(1-q)}\sum_{j<j_{h}}2^{(M-s)j}c^q_{j}.
\end{align*}
Thus by \eqref{BC1}, Fubini's theorem and \eqref{CoVA},
\begin{align*}
\int_{|h|_a\leq 1}|h|_a^{-sq}I^q_1(h)\frac{\dif h}{|h|_a^{a\cdot m}}
&\leq \int_{|h|_a\leq 1}|h|_a^{M-s}\sum_{j<j_{h}}2^{(M-s)j}c^q_{j}\frac{\dif h}{|h|_a^{a\cdot m}}\\
&\leq \sum_{j\geq -1}2^{(M-s)j}c^q_{j}\int_{|h|_a\leq 2^{-j}}|h|_a^{M-s-a\cdot m}\dif h\\
&\lesssim \sum_{j\geq -1}2^{(M-s)j}c^q_{j}2^{-(M-s)j}=\|f\|^q_{\bB^{s,a}_{p,q}(\rho)}.
\end{align*}
Similarly, one can show
\begin{align*}
\int_{|h|_a\leq 1}|h|_a^{-sq}I^q_2(h)\frac{\dif h}{|h|_a^{a\cdot m}}\lesssim \|f\|^q_{\bB^{s,a}_{p,q}(\rho)}.
\end{align*}
Moreover, for $s>0$, we clearly have
$$
\|f\|_{L^p(\rho)}\lesssim \|f\|_{\bB^{s,a}_{p,q}(\rho)}.
$$
Thus we obtain \eqref{fbound}.

(ii) {In this step we prove the converse part of \eqref{fbound}.} For $j\geq 0$, since $\int_{\mathbb{R}^N} \check{\phi}^a_j(h)\dif h=(2\pi)^{d/2}\phi^a_j(0)=0$, by \eqref{Def8} and the change of variable, we have
\begin{align*}
\int_{\mathbb{R}^N} \check{\phi}^a_j(h)\delta^{(M)}_hf(x)\dif h
&=\sum^{M}_{k=0}(-1)^{M-k}{M\choose k}\int_{\mathbb{R}^N} \check{\phi}^a_j(h)f(x+kh)\dif h\\
&=\sum^{M}_{k=1}(-1)^{M-k}{M\choose k}\int_{\mathbb{R}^N} \check{\phi}^a_j(h)f(x+kh)\dif h\\
&=\sum^{M}_{k=1}(-1)^{M-k}{M\choose k}\int_{\mathbb{R}^N} \phi^a_{j}(k\cdot)^{\check{ }}(h)f(x+h)\dif h.
\end{align*}
In particular, if we define for $j\geq-1$,
$$
\phi^{a, M}_j(\xi):=(-1)^{M+1}\sum^{M}_{k=1}(-1)^{M-k}{M\choose k} \phi^{a}_{j}(k \xi),
$$
then
$$
(-1)^{M+1}\int_{\mathbb{R}^N} \check{\phi}^a_j(h)\delta^{(M)}_hf(x)\dif h=[\phi^{a, M}_j]^{\check{}}*f(x)=:\mathcal{R}^{a, M}_jf(x),
$$
and for $j\geq 0$,
$$
\|\cR^{a, M}_j f\|_{L^p(\rho)}\leq\int_{\mR^N} |\check{\phi}^{a}_j(h)|\,\|\delta^{(M)}_hf\|_{L^p(\rho)}\dif h=I^0_j+I^1_j+I^2_j,
$$
where
\begin{align*}
I^0_j&:=\int_{|h|_a>1} |\check{\phi}^{a}_j(h)|\,\|\delta^{(M)}_hf\|_{L^p(\rho)}\dif h,\\
I^1_j&:=\int_{|h|_a\leq 2^{-j}} |\check{\phi}^{a}_j(h)|\,\|\delta^{(M)}_hf\|_{L^p(\rho)}\dif h ,\\
I^2_j&:=\int_{2^{-j}< |h|_a\leq 1} |\check{\phi}^{a}_j(h)|\,\|\delta^{(M)}_hf\|_{L^p(\rho)}\dif h.
\end{align*}
For $I^0_j$, by \eqref{Def8} and \eqref{BETA}, there is a $\kappa>0$ such that
\begin{align*}
\|\delta^{(M)}_h f\|_{L^p(\rho)}\lesssim (1+|h|_a^{\kappa})\|f\|_{L^p(\rho)},
\end{align*}
which implies that
\begin{align*}
I^0_j&\lesssim \|f\|_{L^p(\rho)}\int_{|h|_a>1}|\check\phi^{a}_j(h)|(1+|h|_a^{\kappa})\dif h\\
&=\|f\|_{L^p(\rho)}\int_{|h|_a>2^j}|\check\phi^{a}_0(h)|(1+2^{-j\kappa}|h|_a^{\kappa})\dif h\\
&\leq \|f\|_{L^p(\rho)}2^{-j(s+1)}\int_{|h|_a>2^j}|\check\phi^{a}_0(h)|(1+|h|_a^{\kappa})|h|_a^{s+1}\dif h\\
&\lesssim 2^{-j(s+1)}\|f\|_{L^p(\rho)},
\end{align*}
and
\begin{align*}
\sum_{j\ge{0}}2^{sqj}(I^0_j)^q\lesssim\|f\|^q_{L^p(\rho)}.
\end{align*}
For $I^1_j$, by H\"older's inequality and  change of variable, we have
\begin{align*}
(I^1_j)^q&\leq\left(\int_{|h|_a\leq 2^{-j}} |\check{\phi}^{a}_j(h)|^{q'}\dif h\right)^{q-1}\int_{|h|_a\leq 2^{-j}} \|\delta^{(M)}_hf\|^q_{L^p(\rho)}\dif h\\
&=2^{a\cdot m j}\left(\int_{|h|_a\leq {1}} |\check{\phi}^{a}_0(h)|^{q'}\dif h\right)^{q-1}\int_{|h|_a\leq2^{-j}} \|\delta^{(M)}_hf\|^q_{L^p(\rho)}\dif h.
\end{align*}
Thus, by Fubini's theorem,
\begin{align*}
\sum_{j\geq {0}}2^{sqj}(I^1_j)^q&\lesssim\sum_{j\geq 0} 2^{sq j+a\cdot mj}\int_{|h|_a\leq 2^{-j}} \|\delta^{(M)}_hf\|^q_{L^p(\rho)}\dif h\\
&=\int_{|h|_a\leq 1}\sum_{j\geq0, |h|_a\leq 2^{-j}} 2^{sq j+a\cdot mj} \|\delta^{(M)}_hf\|^q_{L^p(\rho)}\dif h\\
&\lesssim\int_{|h|_a\leq 1}|h|_a^{-sq-a\cdot m} \|\delta^{(M)}_hf\|^q_{L^p(\rho)}\dif h.
\end{align*}
For $I^2_j$, by H\"older's inequality with respect to measure $\frac{\dif h}{|h|_a^{a\cdot m}}$, we also have
\begin{align*}
(I^2_j)^q&=2^{-Mqj}
\left(\int_{2^{-j}{<} |h|_a\leq 1} |2^{aj}h|_a^{a\cdot m+M}|\check{\phi}^{a}_0(2^{aj}h)|\,\frac{\|\delta^{(M)}_hf\|_{L^p(\rho)}}{|h|^M_a}\frac{\dif h}{|h|_a^{a\cdot m}}\right)^q\\
&\leq2^{-Mqj}\left(\int_{|h|_a\geq 1} (|h|_a^{a\cdot m+M}|\check{\phi}^{a}_0(h)|)^{q'}\frac{\dif h}{|h|_a^{a\cdot m}}\right)^{q-1}\!\!\!
\int_{2^{-j}\leq |h|_a\leq 1}\!\!\! \frac{\|\delta^{(M)}_hf\|^q_{L^p(\rho)}}{|h|^{Mq}_a}\frac{\dif h}{|h|_a^{a\cdot m}}.
\end{align*}
As above, by Fubini's theorem,
\begin{align*}
\sum_{j\geq {0}}2^{sqj}(I^2_j)^q&\lesssim\sum_{j\geq0} 2^{(s-M)q j}\int_{2^{-j}\leq |h|_a\leq 1} \frac{\|\delta^{(M)}_hf\|^q_{L^p(\rho)}}{|h|^{Mq}_a}\frac{\dif h}{|h|_a^{a\cdot m}}\\
&\lesssim\int_{|h|_a\leq 1}\frac{\|\delta^{(M)}_hf\|^q_{L^p(\rho)}}{|h|^{sq}_a}\frac{\dif h}{|h|_a^{a\cdot m}}.
\end{align*}
{For $j=-1$, estimate is easy.}
Hence,
$$
\sum_{j\geq -1}2^{sqj}\|\cR^{a, M}_j f\|^q_{L^p(\rho)}\lesssim\int_{|h|_a\leq 1}\frac{\|\delta^{(M)}_hf\|^q_{L^p(\rho)}}{|h|^{sq}_a}\frac{\dif h}{|h|_a^{a\cdot m}}
+\|f\|_{L^p(\rho)}.
$$
On the other hand, noting that
$$
\phi^{a, M}_j(\xi)= \phi^{a, M}_{{-1}}(2^{-a j}\xi)-\phi^{a, M}_{-1}(2^{-aj}\xi)
$$
and
$$
\phi^{a, M}_{-1}(\xi)=1\mbox{ for $\xi\in B^a_{1/{2M}}$  and }
\phi^{a, M}_{-1}(\xi)=0\mbox{ for $\xi\notin B^a_{{2/3}}$},
$$
we have
$$
\mathrm{supp}\ \phi^{a, M}_j\subset B^{a}_{2^{j+1}} \setminus B^{a}_{(2^{j-1})/M}.
$$
Thus, for any $i,j\geq-1$ with $|j-i|>\log_2 M+2=:\gamma$,
$$
\mathcal{R}^{a,M}_{j} \mathcal{R}^{a}_{i}f(x)=0.
$$
Moreover, noting that for any $\xi\in\mathbb{R}^N$,
\begin{align*}
\sum_{j\geq-1}\phi^{a, M}_j(\xi)
&= (-1)^{M+1}\sum^{M}_{k=1} \sum_{j\geq-1}(-1)^{M-k}{M\choose k} \phi^{a}_{j}(k\xi)\\
&= (-1)^{M+1}\sum^{M}_{k=1}(-1)^{M-k}{M\choose k}=1,
\end{align*}
we have
\begin{align*}
\cR^{a}_i f=\sum_{j\geq -1}\cR^{a}_i\cR^{a, M}_j f=\sum_{|j-i|\leq\gamma}\cR^{a}_i\cR^{a, M}_j f.
\end{align*}
Therefore, by \eqref{Ber},
\begin{align*}
\sum_{i\geq -1}2^{sq i} \|\cR^{a}_i f\|^q_{L^p(\rho)}
 &\leq \sum_{i\geq -1}2^{sq i}\sum_{|i-j|\leq \gamma}\| \cR^{a}_i\cR^{a, M}_j f\|^q_{L^p(\rho)}\\
&\lesssim \sum_{i\geq -1}2^{sq i}\sum_{|i-j|\leq \gamma}\|\cR^{a, M}_j f\|^q_{L^p(\rho)}\\
&\lesssim  \sum_{j\geq -1}2^{sq j}\|\cR^{a, M}_j f\|^q_{L^p(\rho)}\lesssim  \|f\|_{\widetilde\bB^{s,a}_{p,q}(\rho)}.
\end{align*}
(iii) In this step we prove the second equivalence in \eqref{FG1} for $s>0$.
For $s\in(0,1)$ and $|h|_a\leq 1$, by \eqref{BE1} and \eqref{BETA}, we have
\begin{align*}
\|[\delta_h,\rho]f\|_{L^p}{\lesssim}\|f\delta_h\rho\|_{L^p}\lesssim |h|\|f\|_{L^p(\rho)}\le |h|_a\|f\|_{L^p(\rho)},
\end{align*}
which implies that for $s\in(0,1)$,
\begin{align}\label{ISO01}
\|f\|_{{\bB}^{s,a}_{p,q}(\rho)}\asymp\|f\|_{\widetilde{\bB}^{s,a}_{p,q}(\rho)}\asymp\|\rho f\|_{\widetilde{\bB}^{s,a}_{p,q}}
\asymp\|\rho f\|_{{\bB}^{s,a}_{p,q}}.
\end{align}
For $s\in[1,2)$, we have
\begin{align*}
\|[\delta^{(2)}_h,\rho ]f\|_{L^p}
{\lesssim}\|f\delta^{(2)}_h\rho\|_{L^p}+\|\delta_h\rho\delta_h f\|_{L^p}
\lesssim |h|_a^2\|f\rho\|_{L^p}+|h|_a\|\rho\delta_h f\|_{L^p},
\end{align*}
which in turn implies \eqref{ISO01} for $s\in[1,2)$ by definition and the equivalence for $s\in(0,1)$.
By induction one can show \eqref{ISO01} for general $s\geq 2$.
\medskip\\
(iv) In this step we show \eqref{EqT} for $s\leq0$. For $s<0$, by Lemma \ref{DuaA} and the equivalence for $s>0$ proven in step (iii),
we have
\begin{align*}
\|f\|_{{\bB}^{s,a}_{p,q}(\rho)}\lesssim \sup_{\varphi\in\sS}\frac{|\<f,\varphi\>|}{\|\varphi\|_{{\bB}^{-s,a}_{p',q'}(\rho^{-1})}}\lesssim \sup_{\varphi\in\sS}\frac{|\<f,\varphi\>|}{\|\rho^{-1}\varphi\|_{{\bB}^{-s,a}_{p',q'}}}\lesssim \|\rho f\|_{{\bB}^{s,a}_{p,q}}.
\end{align*}
For $s=0$ and $q\in[1,\infty)$, we have
\begin{align*}
\|f\|_{\bB^{0,a}_{p,q}(\rho)}^q&=\sum_{j\geq-1}\|\cR^a_j f\|^q_{L^p(\rho)}
\leq\sum_{j\geq-1}\left(\sum_{k\geq -1}\|\cR^a_j(\rho^{-1}\cR_k^a(\rho f))\|_{L^p(\rho)}\right)^{q}\leq I_1+I_2,
\end{align*}
where
\begin{align*}
I_1&:=\sum_{j\geq-1}\left(\sum_{k\leq j}\|\cR^a_j(\rho^{-1}\cR_k^a(\rho f))\|_{L^p(\rho)}\right)^{q},\\
I_2&:=\sum_{j\geq-1}\left(\sum_{k>j}\|\cR^a_j(\rho^{-1}\cR_k^a(\rho f))\|_{L^p(\rho)}\right)^{q}.
\end{align*}
Fix $\alpha\in(0,1)$. By H\"older's inequality, we have
\begin{align*}
I_1&=\sum_{j\geq-1}\(\sum_{k\le j}2^{\alpha k}2^{-\alpha k}\|\cR_j^a (\rho^{-1}\cR_k^a(\rho f))\|_{L^p(\rho)}\)^q\\
&\le \sum_{j\geq-1}\(\sum_{k\le j}2^{\alpha kq/(q-1)}\)^{q-1}\sum_{k\le j}2^{-\alpha kq}\|\cR_j^a (\rho^{-1}\cR_k^a(\rho f))\|_{L^p(\rho)}^q\\
&\lesssim\sum_{j\geq-1}2^{\alpha qj}\sum_{k\geq -1}2^{-\alpha kq}\|\cR_j^a (\rho^{-1}\cR_k^a(\rho f))\|_{L^p(\rho)}^q\\
&=\sum_{k\geq-1}2^{-\alpha kq}\sum_{j\geq -1}2^{\alpha qj}\|\cR_j^a (\rho^{-1}\cR_k^a(\rho f))\|_{L^p(\rho)}^q.
 \end{align*}
Noting that
\begin{align*}
\sum_{j\geq-1} 2^{\alpha qj}\|\cR_j^a(\rho^{-1}g)\|_{L^p(\rho)}^q=\|\rho^{-1}g\|_{\bB^{\alpha,a}_{p,q}(\rho)}^q\lesssim \|g\|_{\bB^{\alpha,a}_{p,q}}^q,
\end{align*}
we further have
\begin{align*}
I_1&\lesssim\sum_{k\geq -1}2^{-\alpha kq}\|\cR_k^a(\rho f)\|_{\bB^{\alpha ,a}_{p,q}}^q
 \lesssim\sum_{k\geq -1}\|\cR_k^a(\rho f)\|_{L^p}^q=\|\rho f\|^q_{\bB^{0,a}_{p,q}}.
\end{align*}
Similarly, one can show
\begin{align*}
I_2\lesssim\sum_{k\geq-1}2^{\alpha kq}\|\cR_k^a(\rho f))\|_{\bB^{-\alpha,a}_{p,q}}^q \lesssim\|\rho f\|^q_{\bB^{0,a}_{p,q}}.
\end{align*}
Thus we get for $q\in[1,\infty)$,
\begin{align*}
\|f\|_{\bB^{0,a}_{p,q}(\rho)}^q\lesssim \|\rho f\|_{\bB^{0,a}_{p,q}}^q.
\end{align*}
For $q=\infty$, it is similar.
Moreover, for $s\leq 0$, by duality, we also have
\begin{align*}
\|\rho f\|_{{\bB}^{s,a}_{p,q}}\lesssim \sup_{\varphi\in\sS}\frac{|\<\rho f,\varphi\>|}{\|\varphi\|_{{\bB}^{-s,a}_{p',q'}}}\lesssim \sup_{\varphi\in\sS}\frac{|\<f,\rho\varphi\>|}{\|\rho\varphi\|_{{\bB}^{-s,a}_{p',q'}(\rho^{-1})}}\lesssim \| f\|_{{\bB}^{s,a}_{p,q}(\rho)}.
\end{align*}
The proof is complete.
\end{proof}

By \eqref{IS} and characterization \eqref{FG1}, the following compact embedding lemma is standard by Ascoli-Arzel\`{a}'s lemma.
\bl\label{CptE}
Let $T>0$, $\rho_1,\rho_2\in\sP_{\rm w}$ and $0<\alpha_1<\alpha_2<2$. If
$\rho_1(z)=\varrho(z)^{-\kappa}$ for some $\kappa>0$,
then the following embedding is compact
\begin{align*}
\mS^{\alpha_2}_{T,a}(\rho_1\rho_2)\hookrightarrow\mS^{\alpha_1}_{T,a}(\rho_2).
\end{align*}
\el
\begin{proof}
Let $f_n$ be a bounded sequence of $\mS^{\alpha_2}_{T,a}(\rho_1\rho_2)$. For any $R\geq 1$,
by \eqref{Cor28}, there is a constant $C=C(R,T)>0$ such that for any $z_1,z_2\in B^a_R$,
\begin{align*}
|f_n(t,z_1)-f_n(t,z_2)|\le C|z_1-z_2|_a^{\alpha_2/2}
\end{align*}
and for any $z=(x,v)\in B^a_R$
\begin{align*}
|f_n(t,z)-f_n(s,z)|&\le |f_n(t,z)-f_n(s,\Gamma_{t-s}z)|+|f_n(s,\Gamma_{t-s}z)-f_n(s,z)|\\
&\lesssim_C |t-s|^{\alpha_2/2}+|(t-s)v|^{\alpha_2/6}\lesssim_C|t-s|^{\alpha_2/6}.
\end{align*}
Hence, by Ascoli-Arzel\`{a}'s theorem and a diagonalization method, there are a subsequence $n_k$ and a continuous $f$ such that for any $R\geq 1$,
\begin{align}\label{UAA}
\lim_{k\to\infty}\sup_{t\in[0,T]}\sup_{z\in B^a_R}|f_{n_k}(t,z)-f(t,z)|=0.
\end{align}
In particular, $f\in\mS^{\alpha_2}_{T,a}(\rho_1\rho_2)$. It remains to show
\begin{align}\label{ZX2}
\lim_{k\to\infty}\|f_{n_k}-f\|_{\mS^{\alpha_1}_{T,a}(\rho_2)}=0.
\end{align}
Note that by definition, for any $R\geq 1$,
$$
\|\1_{\{|z|_a>R\}}(f_{n_k}-f)\|_{\mL^\infty_T(\rho_2)}\leq \|f_{n_k}-f\|_{\mL^\infty_T(\rho_1\rho_2)}/(1+R)^\kappa,
$$
which together with \eqref{UAA} implies that
\begin{align}\label{ZX3}
\lim_{k\to\infty}\|f_{n_k}-f\|_{\mL^\infty_T(\rho_2)}=0.
\end{align}
Since $(f_{n_k})_{k\in\mN}$ is bounded in $\mS^{\alpha_2}_{T,a}(\rho_2)$ and by the interpolation inequality \eqref{Embq0},
$$
\|f\|_{\mS^{\alpha_1}_{T,a}(\rho_2)}\lesssim \|f\|^{\alpha_1/\alpha_2}_{\mS^{\alpha_2}_{T,a}(\rho_2)}
\|f\|^{1-\alpha_1/\alpha_2}_{\mL^\infty_T(\rho_2)},
$$
we get \eqref{ZX2} by \eqref{ZX3}. The proof is complete.
\end{proof}

\section{Proof of Lemma \ref{lem:w}}\label{sec:app}

{In this section we collect some useful lemmas used in Section \ref{Sub6} and give the proof of Lemma \ref{lem:w}.}

\bl\label{lem:B1}
For any $p\geq 2$ and $k\in\mN$, we have
$$
\sup_{z\in\mR^{2d}}\bE|\nabla^kX_\varphi(z)|^p<\infty.
$$
In particular, $z\mapsto X_\varphi(z)$ has a smooth version.
\el
\begin{proof}
	Since $W$ is a bounded linear operator from $\mH$ to $L^2(\Omega)$, we have
	$$
	\nabla^k X_\varphi(z)=X(\nabla^k\varphi(z-\cdot)),\ \ a.s.
	$$
	By the hypercontractivity of Gaussian random variables and \eqref{ISO}, we have
	\begin{align}\label{DHG1}
	\bE|\nabla^k X_\varphi(z)|^p\lesssim(\bE|\nabla^kX_\varphi(z)|^2)^{p/2}
	=\left(\int_{\mR^{2d}}|\widehat {\nabla^k\varphi}(\zeta)|^2\mu(\dif\zeta)\right)^{p/2},
	\end{align}
	which is finite by $\varphi\in\sS(\mR^{2d})$ and \eqref{MU}. The proof is complete.
\end{proof}

\begin{proof}[Proof of Lemma \ref{lem:w}]
	Note that by \eqref{ISO},
	\begin{align}\label{ZX5}
	\bE\Big(X_{\varphi}(z)X_{\varphi'}(z')\Big)=\int_{\mR^{2d}}{\e^{{\rm i}\zeta(z'-z)}}\hat\varphi(\zeta)\hat\varphi'(\zeta)\mu(\dif\zeta)=:I_{\varphi,\varphi'}(z,z').
	\end{align}
	By \eqref{K30} and Fubini's theorem, we have
	\begin{align}
	\bE\Big((X_{\varphi}\otimes X_{\varphi'})(H)\Big)
	&=\int_{\mR^{2d}}\int_{\mR^{2d}}H(z,z')I_{\varphi,\varphi'}(z,z')\dif z\dif z'\label{ZX4}\\
	&=\int_{\mR^{2d}}\hat H(\zeta,-\zeta)\hat\varphi(\zeta){\hat\varphi'}(\zeta)\mu(\dif \zeta).\no
	\end{align}
	Next we look at \eqref{KKN2}. Noting that for Gaussian random variables $(\xi_1,\xi_2,\xi_3,\xi_4)$,
	$$
	\bE(\xi_1\xi_2\xi_3\xi_4)=\bE(\xi_1\xi_2)\bE(\xi_3\xi_4)+\bE(\xi_1\xi_3)\bE(\xi_2\xi_4)+\bE(\xi_1\xi_4)\bE(\xi_2\xi_3),
	$$
	by Fubini's theorem again and \eqref{ZX5}, we have
	\begin{align*}
	&\bE\Big((X_{\varphi}\otimes X_{\varphi'})(H)\Big)^2=
	\bE\left(\int_{\mR^{2d}}\int_{\mR^{2d}}H(z,z')X_{\varphi}(z)X_{\varphi'}(z')\dif z\dif z'\right)^2\\
	&=\int_{\mR^{2d}}\cdot\cdot\int_{\mR^{2d}}H(z,z')H(\bar z,\bar z')\bE\big(X_{\varphi}(z)X_{\varphi'}(z')X_{\varphi}(\bar z)X_{\varphi'}(\bar z')\big)\dif z\dif z'\dif \bar z\dif\bar z'\\
	&=\int_{\mR^{2d}}\cdot\cdot\int_{\mR^{2d}}H(z,z')H(\bar z,\bar z')\Big(I_{\varphi,\varphi'}(z,z')I_{\varphi,\varphi'}(\bar z,\bar z')\\
	&\qquad+I_{\varphi,\varphi}(z,\bar z)I_{\varphi',\varphi'}(z',\bar z')+I_{\varphi,\varphi'}(z,\bar z')I_{\varphi,\varphi'}(\bar z,z')\Big)\dif z\dif z'\dif \bar z\dif\bar z'.
	\end{align*}
	Hence, by \eqref{ZX4},
	\begin{align*}
	&{\rm Var}\Big((X_{\varphi}\otimes X_{\varphi'})(H)\Big)=\bE\Big((X_{\varphi}\otimes X_{\varphi'})(H)\Big)^2-\Big(\bE\big((X_{\varphi}\otimes X_{\varphi'})(H)\big)\Big)^2\\
	&\quad=\int_{\mR^{2d}}\cdot\cdot\int_{\mR^{2d}}H(z,z')H(\bar z,\bar z')I_{\varphi,\varphi}(z,\bar z)I_{\varphi',\varphi'}(z',\bar z')\dif z\dif z'\dif \bar z\dif\bar z'\\
	&\quad+\int_{\mR^{2d}}\cdot\cdot\int_{\mR^{2d}}H(z,z')H(\bar z,\bar z')I_{\varphi,\varphi'}(z,\bar z')I_{\varphi,\varphi'}(\bar z,z')\dif z\dif z'\dif \bar z\dif\bar z'\\
	&\quad=\int_{\mR^{2d}}\int_{\mR^{2d}}\hat H(\zeta,\zeta')\hat H(-\zeta,-\zeta')|\hat\varphi(\zeta)|^2|\hat\varphi'(\zeta')|^2\mu(\dif \zeta)\mu(\dif\zeta')\\
	&\quad+\int_{\mR^{2d}}\int_{\mR^{2d}}\hat H(\zeta,-\zeta')\hat H(\zeta',-\zeta)\hat\varphi(\zeta)\hat\varphi(\zeta')\hat\varphi'(\zeta)\hat\varphi'(\zeta')\mu(\dif \zeta)\mu(\dif\zeta')\\
	&\quad=\int_{\mR^{2d}}\int_{\mR^{2d}}\hat H(\zeta,\zeta')\overline{\hat H(\zeta,\zeta')}|\hat\varphi(\zeta)|^2|\hat\varphi'(\zeta')|^2\mu(\dif \zeta)\mu(\dif\zeta')\\
	&\quad+\int_{\mR^{2d}}\int_{\mR^{2d}}\hat H(\zeta,\zeta')\overline{\hat H(\zeta',\zeta)}\hat\varphi(\zeta)\hat\varphi(\zeta')\hat\varphi'(\zeta)\hat\varphi'(\zeta')\mu(\dif \zeta)\mu(\dif\zeta'),
	\end{align*}
	where the last step is due to the symmetry of $\hat\varphi$, $\hat\varphi'$ and $\mu$.
	From this we get the desired equality \eqref{KKN2}.
\end{proof}

{Recall \eqref{PSI} and we have the following elementary lemmas.}
\bl\label{LP6}
\begin{enumerate}[(i)]
	\item For any $\gamma\in\mR$, there is a constant $C>0$ such that
	\begin{align}\label{SpRI}
	|\phi^a_j(\zeta)|\lesssim_C1\wedge\big(2^{\gamma j}(1+|\zeta|_a)^{-\gamma}\big),\ \ j\geq -1,\ \ \zeta\in\mR^{2d},
	\end{align}
	and
	\begin{align}\label{SpR0}
	|\psi(\zeta,\zeta')|\lesssim_C 1\wedge\big({(1+|\zeta|_a)^{-\gamma}(1+|\zeta'|_a)^{\gamma}}\big),\ \ \zeta,\zeta'\in\mR^{2d}.
	\end{align}
	\item For any $\gamma\in[0,1]$, there is a constant $C>0$ such that
	\begin{align}\label{SpL5}
	|\psi(\zeta,\zeta')-\psi(\zeta,\zeta)|\lesssim_C|\zeta-\zeta'|_a^\gamma(1+ |\zeta|_a)^{-\gamma},\ \ \zeta,\zeta'\in\mR^{2d}.
	\end{align}
\end{enumerate}
\el
\begin{proof}
	(i) Note that
	$$
	K_j:=\text{supp}\phi^a_j\subset \{\zeta: 2^{j-1}\leq|\zeta|_a\leq{ 2^{j+1}}\},\ \ j\geq0.
	$$
	For any $\gamma\in\mR$, since for $j\geq 0$,
	$$
	(1+|\zeta|_a)^{\gamma}\1_{K_j}\lesssim 2^{j \gamma},\ \ (1+|\zeta|_a)^{\gamma}\1_{\{|\zeta|_a\leq 1\}}\lesssim 1,
	$$
	we have
	\begin{align*}
	|\phi^a_j(\zeta)|\le\frac{(1+|\zeta|_a)^\gamma}{(1+|\zeta|_a)^\gamma}\1_{K_j}(\zeta)\lesssim\frac{2^{\gamma j}}{(1+|\zeta|_a)^\gamma},
	\end{align*}
	and
	\begin{align*}
	|\psi(\zeta,\zeta')|&\le \sum_{|i-j|\le1}|\phi^a_i(\zeta)||\phi^a_j(\zeta')|
	\lesssim \sum_{|i-j|\le1}\frac{2^{\gamma i}\1_{K_i}(\zeta)}{(1+|\zeta|_a)^\gamma}\frac{(1+|\zeta'|_a)^\gamma}{2^{\gamma j}}\\
	&\lesssim\sum_{i\geq-1}\1_{K_i}(\zeta)\frac{(1+|\zeta'|_a)^\gamma}{(1+|\zeta|_a)^\gamma}\lesssim\frac{(1+|\zeta'|_a)^\gamma}{(1+|\zeta|_a)^\gamma}.
	\end{align*}
	(ii) Let $\gamma\in[0,1]$. For $j\geq 0$, we have
	$$
	|\phi_j^a(\zeta)-\phi_j^a(\zeta')|=|\phi_0^a(2^{-aj}\zeta)-\phi_0^a(2^{-aj}\zeta')|
	\lesssim |\zeta-\zeta'|_a^\gamma2^{-j\gamma}\|\phi_0^a\|_{\mathbf{C}^\gamma_a}
	$$
	and
	$$
	|\phi_{-1}^a(\zeta)-\phi_{-1}^a(\zeta')|\lesssim |\zeta-\zeta'|_a^\gamma\|\phi_{-1}^a\|_{\mathbf{C}^\gamma_a}.
	$$
	Thus, by \eqref{SpRI},
	\begin{align*}
	|\psi(\zeta,\zeta')-\psi(\zeta,\zeta)|&\lesssim |\zeta-\zeta'|_a^\gamma\sum_{j\geq -1}2^{-j\gamma}\phi_j^a(\zeta)
	\lesssim\frac{|\zeta-\zeta'|_a^\gamma}{(1+|\zeta|_a)^\gamma}\sum_{j\geq-1}\1_{K_j}(\zeta).
	\end{align*}
	The proof is complete.
\end{proof}
We also need the following simple lemma.
\bl\label{SpL3}
For any $T,\lambda>0$, $\theta\in[0,1]$ and $\gamma>0$, there is a constant $C=C(T,\gamma,\theta,\lambda)$ such that for any $0\leq s<t\leq T$ and $\zeta=(\xi,\eta)\in\mR^{2d}$,
\begin{align}\label{UL6}
\int_s^t r^{\gamma-1} \e^{-\lambda(r^3|\xi|^2+r|\eta|^2)}\dif r\lesssim_C|t-s|^{(\gamma\wedge 1)(1-\theta)}(1+|\zeta|_a)^{-2\theta\gamma}.
\end{align}
\el
\begin{proof}
	Note that
	\begin{align*}
	\int_0^t s^{\gamma-1}\e^{-\lambda s^3|\xi|^2}\dif s \lesssim |\xi|^{-\frac{2\gamma}{3}},\ \
	\int_0^t s^{\gamma-1} \e^{-\lambda s|\eta|^2}\dif s \lesssim |\eta|^{-{2\gamma}},
	\end{align*}
	and
	\begin{align*}
	\int_s^t r^{\gamma-1}\dif r= (t^\gamma-s^\gamma)/\gamma\lesssim (t-s)^{\gamma\wedge 1}.
	\end{align*}
	Let $g({r},\zeta):=\e^{-\lambda(r^3|\xi|^2+r|\eta|^2)}$. For any $\theta\in[0,1]$, we have
	\begin{align*}
	\int_s^t r^{\gamma{-1}}g(r,\zeta)\dif r&=\left(\int_s^t r^{\gamma-1}{g}
	(r,\zeta)\dif r\right)^{1-\theta}
	\left(\int_s^t r^{\gamma-1}g(r,\zeta)\dif r\right)^\theta\\
	&\leq\left(\int_s^t r^{\gamma-1}\dif r\right)^{1-\theta}
	\left(\int_0^t s^{\gamma-1}g(s,\zeta)\dif s\right)^\theta\\
	&\lesssim
	(t-s)^{(\gamma\wedge 1)(1-\theta)}\Big(1\wedge |\xi|^{-\frac{2\gamma}{3}}\wedge  |\eta|^{-{2\gamma}}\Big)^\theta,
	\end{align*}
	which in turn gives the result by $1\vee|\xi|^{1/3}\vee|\eta|\asymp 1+|\zeta|_a$.
\end{proof}

\end{appendix}

\end{document}